\documentclass[reqno]{amsart}

\usepackage{amsmath}
\usepackage{amsfonts}
\usepackage{rotating}
\usepackage{bbm}
\usepackage[all,cmtip]{xy}
\usepackage{amscd}
 
\usepackage{tikz}
 \usepackage{pgfplots}

 \usetikzlibrary{arrows}
 \tikzset{>=latex}
\usetikzlibrary{arrows.meta}
\pgfplotsset{compat=1.10}
\usetikzlibrary{decorations.markings}

\usepgfplotslibrary{fillbetween}
\usetikzlibrary{patterns}
\usepackage{mathrsfs}
\usepackage{color}
\usepackage{xcolor}
\usepackage{mathtools}
\usepackage{amssymb}
\usepackage[latin1]{inputenc}
\usepackage{graphicx}
\usepackage{dsfont}
\usepackage{color}
\usepackage{hyperref}
\usepackage{enumerate}
\usepackage{marginnote}

\usepackage{subcaption}
 \usepackage{comment}

\newtheorem{theorem}{Theorem}[section]

\newtheorem{lemma}[theorem]{Lemma}
\newtheorem{proposition}[theorem]{Proposition}
\newtheorem{corollary}[theorem]{Corollary}

\theoremstyle{definition}
\newtheorem{definition}[theorem]{Definition}

\newtheorem{example}[theorem]{Example}

\theoremstyle{remark}
\newtheorem{remark}[theorem]{Remark}

\newcommand{\abs}[1]{\lvert#1\rvert}

\newcommand{\C}{{\mathbb{C}}}

\newcommand{\R}{{\mathbb{R}}}

\newcommand{\Z}{{\mathbb{Z}}}
\newcommand{\N}{{\mathbb{N}}}

\renewcommand{\epsilon}{\varepsilon}
\renewcommand{\theta}{\vartheta}

\newcommand{\relmiddle}[1]{\mathrel{}\middle#1\mathrel{}}

\usepackage{graphicx}
 
\begin{document}

\title[]{On invariants of families of    lemniscate motions in the two-center problem}

 \author[H. H\"au{\ss}ler]{Hanna H\"au{\ss}ler}
   \address{Lehrstuhl f\"ur Analysis und Geometrie, Universit\"at Augsburg, D-86152 Augsburg}
   \email {hanna.haeussler@uni-a.de}

 \author[S. Kim]{Seongchan Kim}
   \address{ Department of Mathematics Education, Kongju National University, Gongju 32588, Republic of Korea}
   \email {seongchankim@kongju.ac.kr}

\setcounter{tocdepth}{3}

\date{Recent modification; \today}
 \begin{abstract}
 
We determine four topological invariants introduced by Cieliebak-Frauenfelder-Zhao \cite{CFZ23}, based on Arnold's $J^+$-invariant, of periodic lemniscate motions in Euler's two-center problem.

  \end{abstract}
   \maketitle

   


\section{Introduction}

The planar circular restricted three-body problem (PCR3BP) studies the dynamics of a massless body, attracted by two massive bodies according to Newton's law of gravitation. This problem is not only of theoretical interest, but also of considerable practical importance, in view of its profound relevance to space missions and astronomy.

According to Poincar\'e, ``periodic orbits are the only opening through which one can try to penetrate in a place which was supposed to be inaccessible" \cite{Poin1}. They have been sought ever since, yet it is notoriously difficult to find periodic orbits in the PCR3BP by a direct method.  In recent years, artificial intelligence (AI) has been employed to search for new periodic orbits. See, for instance, \cite{AI24, AI1}. While such approaches are an effective way to generate a large number of candidate solutions, they are of limited value without a rigorous understanding of the underlying mathematical structure. Different AI models may identify different periodic orbits, but in the absence of mathematical theory that can interpret these orbits, one is left with an overwhelming amount of data that is difficult to handle in a meaningful way. This highlights the fundamental importance of understanding the structure of a given system. One natural approach toward this goal is to look at families of periodic orbits emanating from an obvious periodic orbit, such as a critical point.

Based on Arnold's $J^+$-invariant for closed planar immersions, Cieliebak, Frauenfelder and van Koert introduced in \cite{CFvK17} two topological invariants $\mathcal{J}_1$ and $\mathcal{J}_2$ for families of periodic orbits in a class of Hamiltonian systems admitting a single-center, which includes the PCR3BP for energies below the first critical value.   These invariants provide valuable insights into the deformation properties of families of periodic orbits, allowing us to determine whether such families can be continuously deformed into one another. This, in turn, improves our understanding of how different families might be interconnected.   It was proved in \cite{Kimcorr, KimSZ} that if periodic orbits in the rotating Kepler problem and Euler's two-center problem, which are special cases of the PCR3BP, share the same knot type, then they also possess the same $\mathcal{J}_1$ and $\mathcal{J}_2$. This suggests that, despite the apparent differences between the two problems, there may exist an intrinsic relation between them.

For energies between the first and second critical values in the PCR3BP, the invariants $\mathcal{J}_1$ and $\mathcal{J}_2$ are not applicable, as the massless body can move back and forth between the two primaries. This limitation led Cieliebak, Frauenfelder and Zhao \cite{CFZ23} to introduce a new class of Hamiltonian systems, called two-center Stark-Zeeman systems, and to define four invariants $\mathcal{J}_0, \mathcal{J}_E, \mathcal{J}_M$ and $\left( \mathcal{J}_{E,M}, n \right)$, all of which are based on the $J^+$-invariant.  
The primary goal of this article is to determine these invariants for periodic orbits in Euler's two-center problem.  Even though the Euler problem does not arise as a limit of the PCR3BP, it is homotopic to the PCR3BP through two-center Stark-Zeeman systems by switching on angular momentum. It is worth noting that the rotating Kepler problem does not fall under the class of two-center Stark-Zeeman systems.

 Euler's two-center problem concerns the motion of a massless body subjected to the gravitational attraction of two fixed massive bodies, commonly referred to as primaries.  In this article, we designate the massless body as the satellite and the two primaries as the Earth and the Moon. This problem was first studied by Euler \cite{Euler1767} as a special case of the PCR3BP.  In sharp contrast to the rotating Kepler problem, the Euler  problem is not invariant under rotations: nevertheless,  it is a striking fact that the system is  completely integrable.  See below.

 We normalize the position of the two primaries to $E=(-1,0)$ and $M=(1,0)$  and scale the total mass to be $1$. Denoting by $\mu \in (0,1)$ the mass of $M$, the Hamiltonian of the Euler problem is given as
\begin{equation}\label{JHam:EUler}
H (q,p) = \frac{1}{2}\lvert p \rvert^2 - \frac{1-\mu}{\lvert q - E \rvert} - \frac{\mu}{\lvert q-M \rvert}, \quad (q,p) \in T^* \left( \R^2 \setminus \{ E, M \} \right).
\end{equation}

Throughout the article, we restrict our attention to negative energies $H=c<0$, ensuring that all motions are bounded. For such energies, the trajectories are classified into four types -- $P,L,S$, and $S'$ -- when $\mu\neq\frac{1}{2}$ and into three types -- $P,L$, and $S$ -- when $\mu=\frac{1}{2}$.

\begin{itemize}

    \item Type $P$ trajectories, referred to as planetary motions, describe orbits confined between two ellipses, circling around both $E$ and $M$.

\item Type $L$ trajectories, referred to as lemniscate motions, are confined to a single ellipse, allowing the satellite to pass between $E$ and $M$.

\item Type $S$ and $S'$ trajectories, referred to as satellite motions, are restricted to a neighborhood of either primary. 
    
\end{itemize}
 Note that $P$- and $L$-type orbits exist only for energies exceeding a unique critical value $c_J = -\frac{1}{2} - \sqrt{\mu - \mu^2}$. See Section \ref{Sec:lemni} for further details.

We now introduce the elliptic coordinates $(\lambda, \nu) \in \R \times S^1[-\pi,\pi] $ 
\begin{equation}\label{eq:ellipticco}
q_1 = \cosh \lambda \cos \nu, \quad q_2 = \sinh \lambda \sin \nu
\end{equation}
and the momenta $p_\lambda$ and $p_\nu$ are determined by the   relation $p_1dq_1+p_2dq_2 = p_\lambda d\lambda + p_\nu d \nu$. The Hamiltonian in the elliptic coordinates then becomes 
\[
H  = \frac{1}{\cosh^2 \lambda - \cos^2 \nu}\left( H_\lambda + H_\nu \right),
\]
where  
\[
H_\lambda = \frac{1}{2}p_\lambda^2 -  \cosh\lambda, \; \; \quad H_\nu = \frac{1}{2}p_\nu^2 + (1 - 2\mu)\cos\nu.
\]

 For each negative energy $H=c<0$, the energy level $H^{-1}(c)$ is noncompact due to the presence of collisions with the primaries. In order to regularize these collisions, we now introduce the new Hamiltonian 
 \begin{equation}\label{eq:FofEuler}
 F_{\mu, c}:= (H  - c)(\cosh^2\lambda - \cos^2\nu) =F_\lambda + F_\nu,
 \end{equation}
where $F_\lambda = H_\lambda - c \cosh^2 \lambda$ and $F_\mu = H_\mu + c \cos^2 \nu$. This Hamiltonian is decoupled, and hence $F_\lambda$ and $F_\mu$ are first integrals of motion. Therefore, the flow of $F_{\mu,c}$ satisfies $\varphi_{F_{\mu,c}}^t = \varphi_{F_\lambda}^t \circ \varphi_{F_\nu}^t$, where $\varphi_Q^t$ indicates the Hamiltonian flow generated by $Q$.  Hence, a trajectory  is periodic if and only if the
rotation number $R = T_\nu / T_{\lambda}$ is rational, where $T_{\lambda}$ and $T_{\nu}$ denote the minimal periods of the motions in  $\lambda$ and $\nu$, respectively.  
Given relatively prime positive integers $k$ and $l, $ we refer to a Liouville torus with $ R = k/l$ as a $T_{k, l}$-torus, and periodic orbits lying on such a torus as $T_{k,l}$-type orbits.

\begin{remark}
It is well-known that along lemniscate motions, the rotation number $R$  can attain any positive value in the interval $(0,+\infty)$. See, for instance, \cite[Section 3]{Kim18} or \cite[Section 3]{WDR04}.
\end{remark}

  The main result of this article is the computation of the four Cieliebak-Frauenfelder-Zhao invariants of periodic lemniscate motions.

 \begin{theorem}\label{thmmain} For  a $T_{k,l}$-type lemniscate motion, the Cieliebak-Frauenfelder-Zhao invariants are given by
 \begin{align*}
 \mathcal{J}_0  &= kl - k +1,\\
 \mathcal{J}_E    &= \mathcal{J}_M    = \begin{cases} \frac{ kl}{2} - k+1 & \text{if $l$ is even,}  \\   2kl-2k+1 & \text{if $l$ is odd,}\\ \end{cases}\\
(\mathcal{J}_{E,M}  , n ) &= \left( 1 -k + kl - l^2  ({\rm mod }  \; 2l ),  l \right).
\end{align*}
 \end{theorem}

\begin{example} The theorem above implies  that the invariants of   a $T_{3,2}$-type orbit~are 
\begin{equation}\label{eq:exak3l2}
\mathcal{J}_0  = 4, \;\; \mathcal{J}_E  =  \mathcal{J}_M  = 1, \;\; \left( \mathcal{J}_{E,M} , n  \right) = \left( 0({\rm mod }~4), 2 \right)
\end{equation}
and those of    a $T_{2,3}$-type orbit  are  
\begin{equation}\label{eq:exak2l3}
\mathcal{J}_0  = 5, \;\; \mathcal{J}_E  =  \mathcal{J}_M  = 9, \;\; \left( \mathcal{J}_{E,M}  , n  \right) = \left( 2 ({\rm mod }~6), 3 \right).
\end{equation}

In this example, the above results \eqref{eq:exak3l2} and \eqref{eq:exak2l3} are obtained by  direct calculation. 
   In order to compute $J^+$-invariant of a generic immersion, we make use of the following formula due to Viro \cite[Corollary 3.1.B and Lemma 3.2.A]{Viro96}:  
   let $K$ be a generic immersion of the circle into the complex plane $\C$. Denote by    $\Lambda_K$ the set of  connected components of the complement $\C \setminus K$, and by $D_K$ the set of double points of $K$. We then have
\[
J^+(K) = 1+ \# D_K- \sum_{C \in \Lambda_K}   w_C(K)  ^2 + \sum_{p \in D_K }   {\rm ind}_p (K)  ^2,
\]
where $w_c(K)$ indicates the winding number of $K$ around an interior point in $C \in \Lambda_K$, and ${\rm ind}_p(K)$ is the arithmetic mean of the winding numbers of the  connected components adjacent to the double point $p \in D_K$.

We first consider the $T_{3,2}$-type orbit $K$ given   in Figure \ref{fig:ex1-1}. With the given orientation, we have $w_E(K)=w_M(K)=0$, and  Viro's formula tells us that $J^+(K)= 1 + 9 - 6 + 0 =4$ from which we find
\[
\mathcal{J}_0 (K) = J^+(K) + \frac{ w_E(K)^2}{2} + \frac{ w_M(K)^2}{2} = 4 + 0 + 0 =4.
\]
To determine $\mathcal{J}_E(K)$ and $\mathcal{J}_M(K)$ we first note that since $K$ is symmetric with respect to the $q_2$-axis, it follows that  $L_E^{-1}(K) = L_M^{-1}(K)$, implying that  $\mathcal{J}_E(K)=\mathcal{J}_M(K)$. The winding number of $K$ around $E$ (and also around $M$) is even, so that the preimage $L_E^{-1}(K)$ consists of two connected components that are related by $180$ degree rotation. We choose one of them as in the top right of Figure \ref{fig:ex1-1}. We then find  $w_{M_1}(\tilde{K}_E)^2 = w_{M_2}(\tilde{K}_E)^2 =1$ and again by Viro's formula $J^+(\tilde{K}_E) = 1 + 3 - 4 +0=0$, 
implying that
\[
\mathcal{J}_E(K) = J^+(\tilde{K}_E) + \frac{w_{M_1}(\tilde{K}_E)^2}{2} +\frac{w_{M_2}(\tilde{K}_E)^2}{2}   = 0 + \frac{1}{2} + \frac{1}{2} =1.
\]

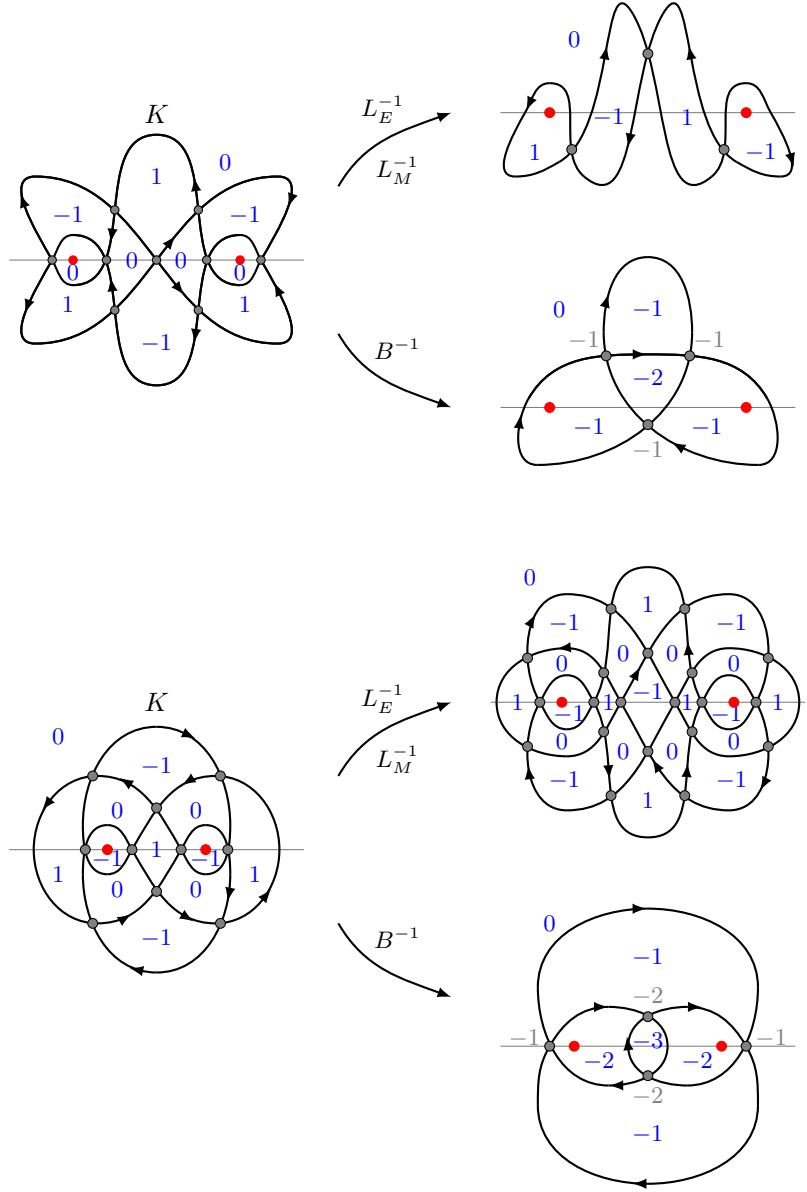
\begin{figure}[htbp]
\centering
\begin{tikzpicture} [scale=0.65]

\begin{scope}[yshift=12cm]
\draw[gray] (-3,0) to (3,0);

\begin{scope}[scale=0.85]

       \filldraw[draw=red, fill=red] (-2,0) circle (0.1cm);
       \filldraw[draw=red, fill=red] (2,0) circle (0.1cm);

\draw[thick]       (-2.5,0)  [out=60, in=180] to (-2,0.6); 
\draw[thick, postaction={decorate},  decoration={markings, mark=at position 0.4 with {\arrow{<}}}]     (-2,0.6)  [out=0, in=180] to (0,-3); 
\draw[thick, postaction={decorate},  decoration={markings, mark=at position 0.6 with {\arrow{>}}}]       (-2.5,0)  [out=240, in=180] to (-2.9,0.-2); 
\draw[thick, postaction={decorate},  decoration={markings, mark=at position 0.6 with {\arrow{>}}}]      (-2.9 ,-2)  [out=0, in=180] to (2.9, 2); 
\draw[thick, postaction={decorate},  decoration={markings, mark=at position 0.4 with {\arrow{>}}}]       (2.9 ,2)  [out=0, in=60] to (2.5,0); 
\draw[thick, postaction={decorate},  decoration={markings, mark=at position 0.4 with {\arrow{<}}}]       (0 ,-3)  [out=0, in=180] to (2, 0.6); 

\begin{scope}[xscale=-1]
\draw[thick]       (-2.5,0)  [out=60, in=180] to (-2,0.6); 
\draw[thick]     (-2,0.6)  [out=0, in=180] to (0,-3); 
\draw[thick, postaction={decorate},  decoration={markings, mark=at position 0.4 with {\arrow{<}}}]       (-2.5,0)  [out=240, in=180] to (-2.9,0.-2); 
\draw[thick, postaction={decorate},  decoration={markings, mark=at position 0.4 with {\arrow{<}}}]      (-2.9 ,-2)  [out=0, in=180] to (2.9, 2); 
\draw[thick, postaction={decorate},  decoration={markings, mark=at position 0.4 with {\arrow{<}}}]       (2.9 ,2)  [out=0, in=60] to (2.5,0); 
\draw[thick]       (0 ,-3)  [out=0, in=180] to (2, 0.6); \end{scope}

\begin{scope}[yscale=-1]

\draw[thick]       (-2.5,0)  [out=60, in=180] to (-2,0.6); 
\draw[thick, postaction={decorate},  decoration={markings, mark=at position 0.4 with {\arrow{<}}}]     (-2,0.6)  [out=0, in=180] to (0,-3); 
\draw[thick]       (-2.5,0)  [out=240, in=180] to (-2.9,0.-2); 
\draw[thick]      (-2.9 ,-2)  [out=0, in=180] to (2.9, 2); 
\draw[thick]       (2.9,2)  [out=0, in=60] to (2.5,0); 
\draw[thick, postaction={decorate},  decoration={markings, mark=at position 0.4 with {\arrow{<}}}]       (0 ,-3)  [out=0, in=180] to (2, 0.6); 
\begin{scope}[xscale=-1]
\draw[thick]       (-2.5,0)  [out=60, in=180] to (-2,0.6); 
\draw[thick]     (-2,0.6)  [out=0, in=180] to (0,-3); 
\draw[thick]       (-2.5,0)  [out=240, in=180] to (-2.9,-2); 
\draw[thick]      (-2.9 ,-2)  [out=0, in=180] to (2.9, 2); 
\draw[thick]       (2.9 ,2)  [out=0, in=60] to (2.5,0); 
\draw[thick]       (0 ,-3)  [out=0, in=180] to (2, 0.6); 
\end{scope}
\end{scope}

       \filldraw[draw=black, fill=gray] (2.5,0) circle (0.1cm);
       \filldraw[draw=black, fill=gray] (-2.5,0) circle (0.1cm);
       \filldraw[draw=black, fill=gray] (0,0) circle (0.1cm);
       \filldraw[draw=black, fill=gray] (1.2,0) circle (0.1cm);
       \filldraw[draw=black, fill=gray] (-1.2,0) circle (0.1cm);
       \filldraw[draw=black, fill=gray] (1,1.2) circle (0.1cm);
       \filldraw[draw=black, fill=gray] (-1, 1.2) circle (0.1cm);
       \filldraw[draw=black, fill=gray] (1, -1.2) circle (0.1cm);
       \filldraw[draw=black, fill=gray] (-1,  -1.2) circle (0.1cm);

\node at (0, 3.5) {  $K$};

\end{scope}

\node[blue] at (-1.8 , -0.9) {\small $1$};
\node[blue] at (-1.8 , 0.9) {\small $-1$};
\node[blue] at (1.8 , -0.9) {\small $1$};
\node[blue] at ( 1.8 , 0.9) {\small $-1$};
 \node[blue] at (0 , 1.7) {\small $1$};
\node[blue] at (0 , -1.7) {\small $-1$};
 \node[blue] at (-0.5 ,0 ) {\small $0$};
 \node[blue] at (1.7 , -0.25 ) {\small $0$};
   \node[blue] at (0.5 ,0 ) {\small $0$};
 \node[blue] at (-1.7 , -0.25 ) {\small $0$};
 \node[blue] at (1.4 ,2 ) {\small $0$};

  \draw[thick, ->] (3.7,1.5) [out=60, in=200] to (6,3); 
\node[below] at (4.9,2.3) {\small $L_M^{-1}$};
\node[above] at (4.6,2.6) {\small $L_E^{-1}$};

\begin{scope}[yscale=-1]
  \draw[thick, ->] (3.7,1.5) [out=60, in=200] to (6 ,3); 
\node[below] at ( 4.9,1.4) {\small $B^{-1}$};
\end{scope}


\begin{scope}[xshift=10cm, yshift=3cm]
\draw[gray] (-3,0) to (3,0);

\draw[thick, postaction={decorate},  decoration={markings, mark=at position 0.4 with {\arrow{<}}}]       (-2.5 ,0)  [out=60, in=180] to (-2, 0.6); 
\draw[thick]       (-2  ,0.6)  [out=0, in=140] to (-1.3 , -1.3); 
\draw[thick, postaction={decorate},  decoration={markings, mark=at position 0.4 with {\arrow{<}}}]       (-1.3 , -1.3)  [out=320, in=240] to (0.3, 2); 
\draw[thick, postaction={decorate},  decoration={markings, mark=at position 0.4 with {\arrow{<}}}]       (0.3 ,2)  [out=60, in=150] to (1.8, -1);
\draw[thick ]       (1.8 , -1)  [out=330, in=180] to (2.7, -1.3);
\draw[thick, postaction={decorate},  decoration={markings, mark=at position 0.4 with {\arrow{<}}}]       (2.7  , -1.3)  [out=0, in=300] to (2.5, 0);

\begin{scope}[xscale=-1]
\draw[thick]       (-2.5 ,0)  [out=60, in=180] to (-2, 0.6); 
\draw[thick]       (-2  ,0.6)  [out=0, in=140] to (-1.3 , -1.3); 
\draw[thick]       (-1.3 , -1.3)  [out=320, in=240] to (0.3, 2); 
\draw[thick, postaction={decorate},  decoration={markings, mark=at position 0.4 with {\arrow{<}}}]       (0.3 ,2)  [out=60, in=150] to (1.8, -1);
\draw[thick ]       (1.8 , -1)  [out=330, in=180] to (2.7, -1.3);
\draw[thick ]       (2.7  , -1.3)  [out=0, in=300] to (2.5, 0); 
\end{scope}

    \filldraw[draw=black, fill=gray] (-1.55,-0.75) circle (0.1cm);
       \filldraw[draw=black, fill=gray] (1.55,-0.75) circle (0.1cm);
       \filldraw[draw=black, fill=gray] (0,1.2) circle (0.1cm);

\node[blue] at (-2.3, -0.8) {\small $1$};
\node[blue] at (-0.8 ,-0.1) {\small $-1$};
\node[blue] at (0.8 , -0.1) {\small $1$};
\node[blue] at (2.3, -0.8) {\small $-1$};

 \node[blue] at (-1.5 ,1.5 ) {\small $0$};

   \filldraw[draw=red, fill=red] (-2,0) circle (0.1cm);
       \filldraw[draw=red, fill=red] (2,0) circle (0.1cm);

\end{scope}


\begin{scope}[xshift=10cm, yshift=-3cm]

\draw[gray] (-3,0) to (3,0);

\begin{scope}[scale=0.9]

\begin{scope}[yshift=0.2cm]

\draw[thick, postaction={decorate},  decoration={markings, mark=at position 0.99 with {\arrow{>}}}]       (-2.7 ,0)  [out=60, in=180] to (0,1);
\draw[thick]       (0 ,1 )  [out=0, in=120] to (2.7,0);
\draw[thick]       (2.7  ,0)  [out=300, in=0] to (2.5,-1.5 );
\draw[thick, postaction={decorate},  decoration={markings, mark=at position 0.4 with {\arrow{>}}}]       ( 2.5 ,-1.5)  [out=180, in=270] to (-1,1.5);
\draw[thick, postaction={decorate},  decoration={markings, mark=at position 0.4 with {\arrow{>}}}]       (-1 ,1.5)  [out=90, in=180] to (0,3.2);

\begin{scope}[xscale=-1]

\draw[thick]     (-2.7 ,0)  [out=60, in=180] to (0,1);
\draw[thick]       (0 ,1)  [out=0, in=120] to (2.7,0);
\draw[thick, postaction={decorate},  decoration={markings, mark=at position 0.4 with {\arrow{<}}}]       (2.7  ,0)  [out=300, in=0] to (2.5,-1.5 );
\draw[thick]     ( 2.5 ,-1.5)  [out=180, in=270] to (-1,1.5);
\draw[thick]       (-1 ,1.5)  [out=90, in=180] to (0,3.2);
\end{scope}
\end{scope}
\end{scope}

 \node[blue] at (-1.8 ,2 ) {\small $0$};

\node[above, gray] at (1.25, 1) {\small $-1$};
\node[above, gray] at (-1.3, 1 ) {\small $-1$};
\node[below, gray] at (0, -0.5) {\small $-1$};

\node[blue] at (0,0.6) {\small $-2$};
\node[blue] at (-1.2,-0.4) {\small $-1$};
\node[blue] at (1.2, -0.4) {\small $-1$};
\node[blue] at (0,2) {\small $-1$};

    \filldraw[draw=black, fill=gray] (0.85, 1.05) circle (0.1cm);
       \filldraw[draw=black, fill=gray] (-0.85, 1.05) circle (0.1cm);
       \filldraw[draw=black, fill=gray] (0,-0.35) circle (0.1cm);

   \filldraw[draw=red, fill=red] (-2,0) circle (0.1cm);
       \filldraw[draw=red, fill=red] (2,0) circle (0.1cm);

\end{scope}

\end{scope}

\draw[gray] (-3,0) to (3,0);

\draw[thick, postaction={decorate},  decoration={markings, mark=at position 0.5 with {\arrow{<}}}]         (-2.5,0)  [out=90, in=180] to (-1.2,1.5); 
\draw[thick, postaction={decorate},  decoration={markings, mark=at position 0.3 with {\arrow{<}}}]        (-1.2, 1.5)  [out= 0, in=120] to (0.5,0); 
\draw[thick]        (0.5,0)  [out=300, in=180] to ( 1 , -0.5); 
\draw[thick]       (1, -0.5)  [out= 0, in=270] to ( 1.5 ,0.5); 
\draw[thick, postaction={decorate},  decoration={markings, mark=at position 0.8 with {\arrow{<}}}]       (1.5, 0.5)  [out= 90, in=0] to (0 ,2.5);

\begin{scope}[xscale=-1]
\draw[thick]      (-2.5,0)  [out=90, in=180] to (-1.2,1.5); 
\draw[thick, postaction={decorate},  decoration={markings, mark=at position 0.3 with {\arrow{>}}}]         (-1.2, 1.5)  [out= 0, in=120] to (0.5,0); 
\draw[thick]    (0.5,0)  [out=300, in=180] to ( 1 , -0.5); 
\draw[thick]       (1, -0.5)  [out= 0, in=270] to ( 1.5 ,0.5); 
\draw[thick]      (1.5, 0.5)  [out= 90, in=0] to (0 ,2.5); 
\end{scope} 

\begin{scope}[yscale=-1]
\draw[thick]         (-2.5,0)  [out=90, in=180] to (-1.2,1.5); 
\draw[thick, postaction={decorate},  decoration={markings, mark=at position 0.3 with {\arrow{>}}}]        (-1.2, 1.5)  [out= 0, in=120] to (0.5,0); 
\draw[thick]        (0.5,0)  [out=300, in=180] to ( 1 , -0.5); 
\draw[thick]       (1, -0.5)  [out= 0, in=270] to ( 1.5 ,0.5); 
\draw[thick, postaction={decorate},  decoration={markings, mark=at position 0.2 with {\arrow{>}}}]       (1.5, 0.5)  [out= 90, in=0] to (0 ,2.5);

\begin{scope}[xscale=-1]
\draw[thick, postaction={decorate},  decoration={markings, mark=at position 0.5 with {\arrow{<}}}]         (-2.5,0)  [out=90, in=180] to (-1.2,1.5); 
\draw[thick, postaction={decorate},  decoration={markings, mark=at position 0.3 with {\arrow{<}}}]        (-1.2, 1.5)  [out= 0, in=120] to (0.5,0); 
\draw[thick]        (0.5,0)  [out=300, in=180] to ( 1 , -0.5); 
\draw[thick]       (1, -0.5)  [out= 0, in=270] to ( 1.5 ,0.5); 
\draw[thick, postaction={decorate},  decoration={markings, mark=at position 0.9 with {\arrow{<}}}]       (1.5, 0.5)  [out= 90, in=0] to (0 ,2.5); 
\end{scope} \end{scope}

       \filldraw[draw=red, fill=red] (-1 ,0) circle (0.1cm);
       \filldraw[draw=red, fill=red] (1 ,0) circle (0.1cm);

       \filldraw[draw=black, fill=gray] (1.45,0) circle (0.1cm);
       \filldraw[draw=black, fill=gray] (-1.45,0) circle (0.1cm);
       \filldraw[draw=black, fill=gray] (0,0.85) circle (0.1cm);
              \filldraw[draw=black, fill=gray] (0,-0.85) circle (0.1cm);

       \filldraw[draw=black, fill=gray] (0.5,0) circle (0.1cm);
       \filldraw[draw=black, fill=gray] (-0.5,0) circle (0.1cm);
        \filldraw[draw=black, fill=gray] (-1.3, 1.5) circle (0.1cm);
       \filldraw[draw=black, fill=gray] (1.3, -1.5) circle (0.1cm);
       \filldraw[draw=black, fill=gray] (-1.3, -1.5) circle (0.1cm);
       \filldraw[draw=black, fill=gray] (1.3, 1.5) circle (0.1cm);

\node at (0, 3) {  $K$};

  \node[blue] at (0 , 1.7) {\small $-1$};
\node[blue] at (0 , -1.8) {\small $-1$};
\node[blue] at (2 , -0.5) {\small $1$};
\node[blue] at (-2 , -0.5) {\small $1$};
\node[blue] at (0,0) {\small $1$};

\node[blue] at (0.8 , 0.8) {\small $0$};
\node[blue] at (-0.8 , 0.8) {\small $0$};
\node[blue] at (0.8 , -0.8) {\small $0$};
\node[blue] at (-0.8 , -0.8) {\small $0$};

\node[blue] at (-1 , -0.2) {\small $-1$};
\node[blue] at (1 , -0.2) {\small $-1$};

 \node[blue] at (-2 ,2.3 ) {\small $0$};

  \draw[thick, ->] (3.7,1.5) [out=60, in=200] to (6,3); 
\node[below] at (4.9,2.3) {\small $L_M^{-1}$};
\node[above] at (4.6,2.6) {\small $L_E^{-1}$};
\begin{scope}[yscale=-1]
  \draw[thick, ->] (3.7,1.5) [out=60, in=200] to (6 ,3); 
\node[below] at ( 4.9,1.4) {\small $B^{-1}$};
\end{scope}


\begin{scope}[xshift=10cm, yshift=3cm]
\draw[gray] (-3.2,0) to (3.2,0);

\begin{scope}[scale=1.1]
\draw[thick]       (-2  ,0)  [out=60, in=180] to (-1.5, 0.5); 
\draw[thick]       (-1.5 ,0.5)  [out= 0, in=120] to (-1 , 0 ); 
 \draw[thick, postaction={decorate},  decoration={markings, mark=at position 0.5 with {\arrow{>}}}]       (-1  ,0)  [out=300, in=180] to (-0, -2.5); 
\draw[thick, postaction={decorate},  decoration={markings, mark=at position 0.6 with {\arrow{>}}}]       (0  , -2.5)  [out=0, in=240] to (1, 0 ); 
\draw[thick]       (1  ,0)  [out=60, in=180] to ( 1.5, 0.5); 
\draw[thick]       (1.5  ,0.5)  [out= 0, in=120] to (2,0); 
\draw[thick, postaction={decorate},  decoration={markings, mark=at position 0.7 with {\arrow{>}}}]       ( 2  ,0)  [out=300, in=0] to ( 1.5, -2); 
\draw[thick, postaction={decorate},  decoration={markings, mark=at position 0.6 with {\arrow{>}}}]       (1.5  , -2)  [out=180, in=300] to (-0.5, 0 ); 
\draw[thick]     (-0.5  ,0)  [out=120, in=0] to (-1.5, 1); 
\draw[thick, postaction={decorate},  decoration={markings, mark=at position 0.1 with {\arrow{>}}}]       (-1.5  , 1)  [out=180, in=90] to (-2.8, 0); 
\draw[thick]      (-2 .8 ,0)  [out=270, in=180] to (-1.5, -1); 
\draw[thick]       (-1.5  , -1)  [out=0, in=240] to (-0.5, 0 ); 
\draw[thick, postaction={decorate},  decoration={markings, mark=at position 0.25 with {\arrow{>}}}]       (-0.5  ,0)  [out=60, in=180] to ( 1.5,2); 
\draw[thick]      (1.5   , 2)  [out= 0, in=60] to (2, 0 ); 
\draw[thick]      ( 2  ,0)  [out=240, in=0] to ( 1.5, -0.5); 
\draw[thick ]      (1.5  , -0.5)  [out=180, in=300] to (1, 0); 

\draw[thick, postaction={decorate},  decoration={markings, mark=at position 0.4 with {\arrow{>}}}]       (1  ,0)  [out=120, in=0] to (0, 2.5);

\draw[thick]     (0  , 2.5)  [out=180, in=60] to (-1 , 0 ); 
\draw[thick]      (-1  ,0)  [out=240, in=0] to (-1.5, -0.5); 
\draw[thick]       (-1.5  ,  -0.5)  [out=180, in=300] to (-2 , 0 ); 
\draw[thick, postaction={decorate},  decoration={markings, mark=at position 0.7 with {\arrow{>}}}]       (-2  , 0)  [out=120, in=180] to (-1.5 ,  2 ); 
\draw[thick]      (-1.5  ,2)  [out= 0, in=120] to (0.5 , 0 ); 
\draw[thick]      (0.5  , 0)  [out=300, in=180] to (1.5 ,   -1 ); 
\draw[thick]      (1.5  ,  -1)  [out= 0, in=270] to (2.8 ,  0 ); 
\draw[thick]      (2.8  , 0)  [out=90, in=0] to ( 1.5 ,  1 ); 
\draw[thick]    (1.5  , 1)  [out=180, in=60] to (0.5 ,  0 ); 
\draw[thick]   (0.5  , 0)  [out=240, in=0] to (-1.5 , - 2 ); 
\draw[thick, postaction={decorate},  decoration={markings, mark=at position 0.5 with {\arrow{>}}}]       (-1.5  ,  -2)  [out=180, in=240] to (-2 ,  0 ); 
\end{scope}

    \filldraw[draw=black, fill=gray] (-2.2 , 0) circle (0.1cm);
    \filldraw[draw=black, fill=gray] (-1.1 , 0) circle (0.1cm);
    \filldraw[draw=black, fill=gray] (-0.55 , 0) circle (0.1cm);
    \filldraw[draw=black, fill=gray] (2.2 , 0) circle (0.1cm);
    \filldraw[draw=black, fill=gray] (1.1 , 0) circle (0.1cm);
    \filldraw[draw=black, fill=gray] (0.55 , 0) circle (0.1cm);
    
     \filldraw[draw=black, fill=gray] (0 , 1) circle (0.1cm);
     \filldraw[draw=black, fill=gray] (0 ,- 1) circle (0.1cm);
     \filldraw[draw=black, fill=gray] (-0.75 , 1.9) circle (0.1cm);
     \filldraw[draw=black, fill=gray] (0.75 , 1.9) circle (0.1cm);
     \filldraw[draw=black, fill=gray] (-0.75 ,- 1.9) circle (0.1cm);
     \filldraw[draw=black, fill=gray] (0.75 ,- 1.9) circle (0.1cm);

     \filldraw[draw=black, fill=gray] (-2.45 , 0.9) circle (0.1cm);
     \filldraw[draw=black, fill=gray] (-2.45 , -0.9) circle (0.1cm);
     \filldraw[draw=black, fill=gray] (2.45 , 0.9) circle (0.1cm);
     \filldraw[draw=black, fill=gray] ( 2.45 , -0.9) circle (0.1cm);

     \filldraw[draw=black, fill=gray] (-0.9 , 0.6) circle (0.1cm);
     \filldraw[draw=black, fill=gray] (-0.9 , -0.6) circle (0.1cm);
     \filldraw[draw=black, fill=gray] ( 0.9 , 0.6) circle (0.1cm);
     \filldraw[draw=black, fill=gray] ( 0.9 ,- 0.6) circle (0.1cm);

 \node[blue] at (-2 .4,2.55 ) {\small $0$};

\node[blue] at (0,0.2) {\small $-1$};
\node[blue] at (0,2) {\small $ 1$};
\node[blue] at (0,-2) {\small $ 1$};
\node[blue] at (-1.7,1.6) {\small $-1$};
\node[blue] at (-1.7,-1.6) {\small $-1$};
\node[blue] at (1.7,1.6) {\small $-1$};
\node[blue] at (1.7,-1.6) {\small $-1$};

\node[blue] at (-2.65,0) {\small $1$};
\node[blue] at ( 2.65,0) {\small $ 1$};
\node[blue] at (-0.8,0) {\small $ 1$};
\node[blue] at ( 0.8,0) {\small $ 1$};
\node[blue] at (1.6, -0.25) {\small $-1$};
\node[blue] at (-1.6, -0.25) {\small $-1$};

 \node[blue] at (-0.5, -1) {\small $0$};
 \node[blue] at ( 0.5, -1) {\small $0$};
 \node[blue] at ( 0.5,  1) {\small $0$};
 \node[blue] at (-0.5,  1) {\small $0$};

 \node[blue] at (-1.75, -0.8) {\small $0$};
 \node[blue] at (1.75, -0.8) {\small $0$};
 \node[blue] at (-1.75,  0.8) {\small $0$};
 \node[blue] at ( 1.75,  0.8) {\small $0$};

   \filldraw[draw=red, fill=red] (-1.75 ,0) circle (0.1cm);
       \filldraw[draw=red, fill=red] (1.75,0) circle (0.1cm);

\end{scope}


\begin{scope}[xshift=10cm, yshift=-4cm]

\draw[gray] (-3,0) to (3,0);

\begin{scope}[scale=0.8]

\draw[thick, postaction={decorate},  decoration={markings, mark=at position 0.99 with {\arrow{>}}}]       (-2.5 ,0)  [out=60, in=180] to (-1,1);
\draw[thick]     (-1 ,1)  [out= 0, in=90] to (0.5, 0);
 \draw[thick, postaction={decorate},  decoration={markings, mark=at position 0.99 with {\arrow{>}}}]       (0.5 ,0)  [out=270, in=0] to (-1,-1);
 \draw[thick]     (-1, -1)  [out=180, in=270] to (-2.8, 1.5);
 \draw[thick, postaction={decorate},  decoration={markings, mark=at position 0.99 with {\arrow{>}}}]       (-2.8 , 1.5)  [out=90, in=180] to (0,3.5);

 \begin{scope}[xscale=-1]
 
\draw[thick, postaction={decorate},  decoration={markings, mark=at position 0.99 with {\arrow{<}}}]       (-2.5 ,0)  [out=60, in=180] to (-1,1);
\draw[thick]     (-1 ,1)  [out= 0, in=90] to (0.5, 0);
 \draw[thick, postaction={decorate},  decoration={markings, mark=at position 0.05 with {\arrow{<}}}]       (0.5 ,0)  [out=270, in=0] to (-1,-1);
 \draw[thick]     (-1, -1)  [out=180, in=270] to (-2.8, 1.5);
 \draw[thick]      (-2.8 , 1.5)  [out=90, in=180] to (0,3.5);

\end{scope}
 
 \begin{scope}[yscale=-1]
   \draw[thick]     (-2.5,0)  [out=120, in=270] to (-2.8, 1.5);
 \draw[thick, postaction={decorate},  decoration={markings, mark=at position 0.99 with {\arrow{<}}}]       (-2.8 , 1.5)  [out=90, in=180] to (0,3.5);

 \begin{scope}[xscale=-1]
 
   \draw[thick]     (-2.5,0)  [out=120, in=270] to (-2.8, 1.5);
 \draw[thick]      (-2.8 , 1.5)  [out=90, in=180] to (0,3.5);

\end{scope}
\end{scope}

\end{scope}

\node[right, gray] at (2 ,0.15) {\small $-1$};
\node[left, gray] at (-2 ,0 .15) {\small $-1$};
\node[above, gray] at (0, 0.65) {\small $-2$};
\node[below, gray] at (0, -0.65) {\small $-2$};

 \node[blue] at (-2 ,2.5 ) {\small $0$};

\node[blue] at (0,0.1 ) {\small $-3$};
\node[blue] at (-1 ,-0.3) {\small $-2$};
\node[blue] at (1 , -0.3) {\small $-2$};
\node[blue] at (0,1.8) {\small $-1$};
\node[blue] at (0,-1.8) {\small $-1$};

    \filldraw[draw=black, fill=gray] (-2  , 0) circle (0.1cm);
       \filldraw[draw=black, fill=gray] (2  ,0) circle (0.1cm);
       \filldraw[draw=black, fill=gray] (0,-0.6 ) circle (0.1cm);
       \filldraw[draw=black, fill=gray] (0,0.6 ) circle (0.1cm);

   \filldraw[draw=red, fill=red] (-1.5,0) circle (0.1cm);
       \filldraw[draw=red, fill=red] (1.5,0) circle (0.1cm);

\end{scope}

      \end{tikzpicture}
 \caption{Examples of a $T_{3,2}$-type orbit (above) and a $T_{2,3}$-type orbit (below) with    its regularized orbits.  
The blue numbers indicate    the winding numbers of the corresponding connected components.  
The gray numbers correspond to   the winding numbers of the associated   double points, which are also  shaded in   gray.  Any double point  not labeled with a number has  winding number $0$}
 \label{fig:ex1-1}
\end{figure}

Since $w_E(K)+w_M(K)=0$ is even, the preimage $B^{-1}(K)$ also has two connected components, and we choose one component as in Figure \ref{fig:ex1-1} from which we immediately obtain 
\[
n(K) = \lvert w_0(\tilde K) \rvert = 2.
\]
Using Viro's formula we further find
\[
\mathcal{J}_{E,M}(K) = J^+(\tilde{K})  = 1  + 3 - 7 + 3 = 0.
\] 
We observe that the results obtained thus far are in agreement with those presented in \eqref{eq:exak3l2}.

We then move to the $T_{2,3}$-type orbit $K$ given in Figure \ref{fig:ex1-1}.  Arguing as in the previous case we find
\begin{align*}
     \mathcal{J}_0 (K) &= \displaystyle J^+(K) + \frac{ w_E(K)^2}{2} + \frac{ w_M(K)^2}{2} = \left( 1 + 10 -  7   +  0   \right) + \frac{1}{2} + \frac{1}{2}= 5,\\
\mathcal{J}_E(K) &=  \displaystyle J^+(\tilde{K}_E) + \frac{w_{M_1}(\tilde{K}_E)^2}{2} +\frac{w_{M_2}(\tilde{K}_E)^2}{2}    = \left( 1 + 20   -  13  + 0    \right) + \frac{1}{2} + \frac{1}{2} = 9,    \\
n(K) &= \lvert w_0(\tilde K) \rvert =   3, \\
\mathcal{J}_{E,M}(K) &= J^+(\tilde{K})  = 1+4- 19 + 10 = -4 =2({\rm mod 6}),
\end{align*}
which precisely match the  corresponding results given in  \eqref{eq:exak2l3}.

\end{example}

 \noindent
 {\bf Acknowledgments.} HH would like to express her sincere gratitude to her supervisor Urs Frauenfelder for his encouragement and continued support. She also thanks the Institute for Mathematics at the University of Augsburg for providing an excellent research environment. Part of this work was carried out during her visit to Kongju National University. She cordially thanks its warm hospitality. HH was supported by a travel scholarship given by the B\"uro f\"ur Chancengleichheit at the University of Augsburg and partially supported by Deutsche Forschungsgemeinschaft grant FR 2637/5-1.

 \section{$J^+$-invariants for two-center Stark-Zeeman systems}
 
This section reviews certain results on periodic orbits in planar two-center Stark-Zeeman systems from \cite{CFZ23} and presents a mild generalization inspired by the approach in \cite{KimSZ}.

\subsection{Planar two-center Stark-Zeeman systems}

Let $  E, M \in \mathbb{R}^2 \equiv \mathbb{C} $ be two distinct points, referred to as the  {Earth} and the  {Moon}, respectively. The gravitational potentials induced by  $E$ and $M$ are defined as
\begin{align*}
&V_E \colon \mathbb{R}^2 \setminus \{E\} \to \mathbb{R}, \quad q \mapsto -\frac{\mu_E}{\lvert q - E \rvert}, \\
&V_M \colon \mathbb{R}^2 \setminus \{M\} \to \mathbb{R}, \quad q \mapsto -\frac{\mu_M}{\lvert q - M\rvert},
\end{align*}
  where  $\mu_E$, $\mu_M>0$ denote the masses of $E$ and $M$,  respectively.

Let $V_1 \colon U_0 \to \R $ be a smooth function defined on an open set $U_0 \subset \R^2$ containing $E$ and $M$. We consider
\[
V := V_E + V_M + V_1  \colon U \to \mathbb{R}, \quad \quad U := U_0 \setminus \{E, M\}.
\]

Given a smooth function $\mathcal{B}  \colon U_0 \to \R$, the associated magnetic form is defined by
\[
\sigma_\mathcal{B}   = \mathcal{B}  (q) dq_1 \wedge dq_2 \in \Omega^2(U_0).
\]
The corresponding {twisted symplectic form} on $T^*U_0$ is then given by
\[
\omega_\mathcal{B}   = \omega_0+ \pi^* \sigma_\mathcal{B}   \in \Omega^2(T^*U_0),
\]
where $  \pi : T^*U_0 \cong U_0 \times \R^2 \to U_0 $ denotes the footpoint projection.

Fix a Riemannian metric $g$ on $TU_0$ and denote by $g^*$ its dual metric on $T^*U_0$. {\bf A two-center  Stark-Zeeman system} is a Hamiltonian system associated with the Hamiltonian of the form
\[
  H_{V,g} (q,p) = \frac{1}{2} \| p \|_{g^*_q}^2 + V(q) ,
\]
defined on $( T^*U, \omega_\mathcal{B}  ). $  The corresponding Hamiltonian vector field $X^\mathcal{B}  _{V,g}$ is  implicitly defined by the equation
\[
dH_{V,g} = \omega_\mathcal{B}  (\cdot, X^\mathcal{B}  _{V,g}).
\]
Notice that when $V_1$ and $\mathcal{B}  $ vanish, the two-center Stark-Zeeman system reduces to the Euler problem.

Since $\sigma_\mathcal{B}  $ is closed and the second de Rham cohomology group of $U_0$ is trivial, there exists a one-form $\alpha_A = A_1(q) dq_1 + A_2(q) dq_2$ such that  $d\alpha_A = \sigma_\mathcal{B}  . $ The symplectomorphism $\Phi_A \colon \left( T^*U, \omega_0 \right)  \to \left( T^*U, \omega_\mathcal{B}   \right)$, defined by
\[
\Phi_A(q,p) = (q, p- A(q)), \quad \quad A(q) = (A_1(q),A_2(q)),
\]
pulls back the Hamiltonian flow of $H_{V,g}$ with respect to the twisted symplectic form $\omega_\mathcal{B}  $ to the Hamiltonian flow of the Hamiltonian
\begin{equation}\label{eq:newHam}
H_{A,V,g}  (q,p):=  \Phi_A^* H_{V,g}(q,p) = \frac{1}{2} \| p - A(q) \|^2_{g_q^*} +V(q) 
\end{equation}
with respect to the standard symplectic form $\omega_0$. It follows that the Hamiltonian vector field $X_{V,g}^\mathcal{B}  $ is conjugate to the Hamiltonian vector field $X_{A,V,g}$, implicitly defined as
\[
dH_{A,V,g} = \omega_0 ( \cdot, X_{A,V,g}).
\]
A two-center Stark-Zeeman system with $A \equiv0$ is called a {\bf two-center Stark system}.

\begin{remark}
   The function $V_1$ can be regarded as a non-gravitational, position-dependent perturbation, whereas $\mathcal{B}  $ accounts for velocity-dependent effects, including those arising from the Lorentz or Coriolis force.
    See \eqref{eq:newHam}.
    \end{remark}

\subsection{Periodic orbits in two-center Stark-Zeeman systems}
 The behavior of families of periodic orbits in (single-center) Stark-Zeeman systems has been studied   in  \cite{CFvK17, KimSZ}. Due to the local nature of the analysis,   the corresponding results extend naturally to two-center Stark-Zeeman systems, as briefly described in the following.

The Hamiltonian $H_{A,V,g}$ is time-independent, allowing us to study the dynamics restricted to individual energy levels.  From now on, for the sake of convenience, we shall write $H =H_{A,V,g}.$

Fix an energy value $c \in \R$ and let $\Sigma_c$ be a connected component  of the energy level $H ^{-1}(c)$. We impose the following assumptions: 
\begin{itemize}
\item  $c$ is a regular value of $V$; and 

\item    Hill's region $\mathcal{K}_c :=\pi ( \Sigma_c ) = \{ q \in U \mid V(q) \leq c\} $ is bounded, and the union $\mathcal{K}_c \cup \{E, M\}$ is simply connected.  
\end{itemize}

 Let $q^s \in \mathcal{K}_c, s \in (-\varepsilon, \varepsilon)$ be a family of periodic orbits contained in $\mathcal{K}_c$. Recall that the Hamiltonian $H$ admits singularities at $E$ and $M$, corresponding to collisions. In what follows, we allow these orbits to pass  through the singularities.  See Section \ref{sec:CFZinv} for further details.

Pick a point $q_0 = q^0(t_0)$. The orbit $q^0$   fails to be an immersion at $q_0$ only if   $q_0=E$,  $q_0=M, $ or $q_0 \in \partial \mathcal{K}_c$. We first consider the case where  $\mathcal{B} \neq 0$.
Then, in each of the above situations,   the orbit $q^0$ develops a   cusp at $t=t_0$, through which a loop is either created or annihilated.    In contrast, when  $\mathcal{B}   \equiv 0, $ namely, in a two-center Stark system, the orbit retraces its previous path near $q_0$, and every point on the orbit becomes an inverse self-tangency. This observation leads to the following definition.

\begin{definition}\label{def:threeorbits}
A periodic orbit in a two-center Stark system that exhibits inverse self-tangencies is classified as

    \begin{enumerate}

        \item a brake-brake orbit, if it hits the boundary of Hill's region at two distinct points;

        \item a brake-collision orbit, if it hits the boundary of Hill's region and also collides with either $E$ or $M$; and 

        \item a collision-collision orbit, if it undergoes two collisions--either both with the same singularity or one with each of $E$ and $M$. The former case is referred to as  type I and the latter as type II.

    \end{enumerate}
\end{definition}

Considering all the events that may occur along families of periodic orbits, Cieliebak, Frauenfelder, and Zhao introduce the following notion.

\begin{definition}[{\cite[Definition 5.4]{CFZ23}}] \label{def:two-center-SZhomotopy}
A smooth one-parameter family $(K^s)_{s \in [0,1]}$ of (simple) closed curves in $\C$ is called a {\bf two-center Stark-Zeeman homotopy} if it satisfies the following conditions:
\begin{itemize}

\item for all but finitely many values of $s \in (0,1)$, the  curve $K^s$ is a generic immersion in $\C \setminus \{ E,M\}$, meaning that it is an immersion   with only transverse double points. 

\item  at the finitely many exceptional values $s \in (0,1)$ where $K^s$ is not a generic immersion,   the family undergoes one of the following local events  (see Figures 5-8 in \cite{CFvK17}):
    \begin{itemize}
        \item[$(I_E)$] birth or death of interior loops around $E$ through  cusps at $E$;
        \item[$(I_M)$] birth or death of interior loops around $M$ through  cusps at $M$;
        \item [$(I_\infty)$] birth or death of exterior loops through cusps at the boundary of Hill's region;
        \item [$(II^-)$] crossings through inverse self-tangencies;
        \item [$(III)$] crossings through triple points.
    \end{itemize}
 
\end{itemize}

\end{definition}

As in \cite{KimSZ}, we treat the case $\mathcal{B}   \equiv 0 $ separately.

\begin{definition}[{cf. \cite[Definition 2.10]{KimSZ}}] \label{def:two-center-Shomotopy}
A two-center Stark-Zeeman homotopy is called a {\bf two-center Stark homotopy} if it involves only the events ($II^-$) and ($III$), where ($II^-$) now corresponds to the emergence of the distinguished orbits introduced in Definition \ref{def:threeorbits}.

\end{definition}

\subsection{Four invariants of two-center Stark-Zeeman homotopies}\label{sec:CFZinv}

In \cite{Arnold1994} Arnold introduced the $J^+$-invariant for generic homotopies, which   is invariant under the events ($I_{\infty}$), ($II^-$) and ($III$), but is sensitive to direct self-tangencies: it increases (or decreases) by $2$ when passing through a direct self-tangency that creates (or eliminates) a double point. Moreover, the $J^+$-invariant   satisfies the following normalization condition. To describe this,  we define the standard curves $K_j, j \in \N \cup \{ 0 \}$ as follows: $K_0$ is the figure eight, $K_1 $ is the circle, and for every $j \geq 2$ the curve $K_j$ is the circle with $j-1$ interior loops attached as illustrated in Figure \ref{fig:Kj}.  
    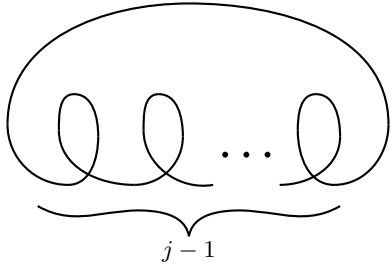
\begin{figure}[ht]
     \centering
\begin{tikzpicture} [scale=0.8]
 
     \draw[thick]  ( -2.5, 0 ) [out=90 ,in=180]  to (0.5,2);
     \draw[thick]  ( 3.8, 0 ) [out=90 ,in=0]  to (0.5,2);
    \draw[thick] (-2.5,0) [out=270, in=180] to (-1.5, -1);
    \draw[thick] (-1.5,-1) [out=0, in=270] to (-1,-0.2); 
    \draw[thick] (-1,-0.2) [out=90, in=0] to (-1.4,0.5);
    \draw[thick] (-1.4, 0.5) [out=180, in=90]  to (-1.65, -0.1);
  \draw[thick] (-1.65  , -0.1) [out=270, in=180] to (-0.4, -1);
 \draw[thick] (-0.4,-1) [out=0, in=270] to (0.4, -0.2);
    \draw[thick] (0.4,-0.2) [out=90, in=0] to (0,0.5);
    \draw[thick] (0, 0.5) [out=180, in=90]  to (-0.25, -0.1);
  \draw[thick] (-0.25  , -0.1) [out=270, in=190] to (0.9, -1);

       \filldraw[draw=black, fill=black] (1.1,-0.5) circle (0.04cm);
       \filldraw[draw=black, fill=black] (1.45,-0.5) circle (0.04cm);
       \filldraw[draw=black, fill=black] (1.8,-0.5) circle (0.04cm);

    \draw[thick] (2,-1) [out=0, in=270] to (3,-0.2); 
    \draw[thick] (3,-0.2) [out=90, in=0] to (2.6,0.5);
    \draw[thick] (2.6, 0.5) [out=180, in=90]  to (2.3, -0.1);
  \draw[thick] (2.3  , -0.1) [out=270, in=180] to (2.9, -1);
 \draw[thick] (2.9,-1) [out=0, in=270] to (3.8,0);

\begin{scope}[yshift=-0.05cm]
\draw[thick] (-2 , -1.3) [out=330, in=170] to (0, -1.4);
\draw[thick] (0, -1.4) [out=350, in=100] to (0.5, -1.8);
\begin{scope}[xscale=-1, xshift=-1cm]
    \draw[thick] (-2 , -1.3) [out=330, in=170] to (0, -1.4);
\draw[thick] (0, -1.4) [out=350, in=100] to (0.5, -1.8);
\end{scope}\end{scope}
\node  at (0.5, -2.1)  {\small $j-1$};
      \end{tikzpicture}
    \caption{Standard curves $K_j$ for $  j \geq 2$}
     \label{fig:Kj}
 \end{figure}  
The  $J^+$-invariant of the standard curve $K_j$ is defined by
\[
J^+(K_j) = \begin{cases} 2-2j & \;  j \geq 1, \\ 0 &  \; j=0. \end{cases}
\]

In the context of two-center Stark-Zeeman systems, families of periodic orbits may  exhibit the additional events ($I_E$) and ($I_M$), which cause  the $J^+$-invariant to change along such families. This observation led Cieliebak, Frauenfelder, and Zhao \cite{CFZ23} to introduce new invariants tailored to two-center Stark-Zeeman systems, based on modifications of the $J^+$-invariant. Throughout this section, the metric is assumed to be conformal to the standard Euclidean metric.

The first of these invariants, denoted by $\mathcal{J}_0$, requires no regularization.  It is defined by correcting the $J^+$-invariant with the winding numbers of the curve $K$ around the singularities $E$ and $M$:
\[
\mathcal{J}_0(K) := J^+(K) + \frac{w_E(K)^2}{2} + \frac{w_M(K)^2}{2},
\]
where  $w_E$ and $w_M$ denote the winding numbers of $K$ around $ E $ and $ M $, respectively.

To handle double collisions at $ E $ or $ M $, we apply Levi-Civita regularization. Recall that the Levi-Civita transformation is given by the cotangent lift of the complex squaring map $L(z) = z^2, \; z \in \C \setminus \{ 0 \}$.   By applying a translation, we may assume without loss of generality that $E $ lies at the origin. After applying Levi-Civita regularization, collisions at $E$ are regularized, but $M$   remains non-regularized and is pulled back to two singularities, denoted by $M_1$ and $M_2$.  Let $K$ be a closed curve that is a generic immersion in $\C \setminus \{ E, M\}$. If $w_E(K)$ is odd,   then the preimage $L^{-1}(K)$ is connected and  satisfies $w_E(L^{-1}(K)) = 2w_E(K)$.   If   $w_E(K)$ is even,   then  $L^{-1}(K)$  consists of two connected components that are related by a rotation of  $180^{\circ}$. In this case, each component  of $L^{-1}(K)$  has the same winding number around $E$ as $K$ does.  Let  $\tilde{K}_E$  denote  a connected component of ${L}^{-1}(K)$. We then define the invariant
\begin{equation*}\label{eq:JE}
\mathcal{J}_E(K) := J^+(\tilde{K}_E) + \frac{w_{M_1}(\tilde{K}_E)^2}{2} + \frac{w_{M_2}(\tilde{K}_E)^2}{2}.
\end{equation*}
By placing $M$ at the origin instead and applying the same procedure, we define 
\[
\mathcal{J}_M(K)  := J^+(\tilde{K}_M) + \frac{w_{E_1}(\tilde{K}_M)^2}{2} + \frac{w_{E_2}(\tilde{K}_M)^2}{2},
\]
where $\tilde{K}_M$ denotes a   component of $L^{-1}(K)$ after regularizing the singularity at $M$.
These invariants are independent of the choice  of the component $\tilde{K}_E$ or $\tilde{K}_M$.

The final invariant appears as a pair,  obtained by simultaneously regularizing   collisions at both $E$ and $M$.  
To this end, we place the singularities  at $E=-1$ and $M=+1$.  
Consider the Birkhoff regularization map  
\begin{equation}\label{eq:Birkhoff}
B \colon    \mathbb{C} \setminus \{0\} \to \mathbb{C}, \quad B(z) = \frac{1}{2} \left( z + \frac{1}{z} \right)
\end{equation}
which is the conjugation of the complex squaring map $L(z)=z^2$ via the M\"obius transformation  $T(z)=(1 - z)/( 1 + z)$.  Let  $K \subset \C \setminus \{ E, M \} $ be a   closed curve that is a generic immersion. Then the preimage $B^{-1}(K) \subset \C \setminus \{ 0 \}$ consists of 
\begin{itemize}
    \item a single connected component if $w_E(K)+w_M(K)$ is odd;
    
    \item   two connected components if $w_E(K)+w_M(K)$ is even. 
\end{itemize}
See \cite[Proposition 4.2]{CFZ23} for a proof.  Select   one connected component of $B^{-1}(K)$, denoted by $\tilde K \subset \C \setminus \{ 0 \}, $ and define the following two invariants:
\[
n(K) := |w_0(\tilde{K})| \in \mathbb{N} \cup \{ 0 \} 
\]
and
\[
\mathcal{J}_{E,M}(K) := 
\begin{cases} 
J^+(\tilde{K})  & \text{if } n(K) = 0, \\ 
J^+(\tilde{K}) \mod 2n(K)  & \text{if } n(K) > 0.
\end{cases}
\]
 The pair $(\mathcal{J}_{E,M}, n)$ is independent of the choice of the connected component $\tilde K$.

\begin{proposition}[{\cite[Propositions 6.2, 6.4 and 6.9]{CFZ23}}]
Each of the four invariants $\mathcal{J}_0, \mathcal{J}_E, \mathcal{J}_M$ and $\left(\mathcal{J}_{E,M}, n\right)$ remains constant under two-center Stark-Zeeman homotopies.
\end{proposition}

Cieliebak, Frauenfelder, and Zhao studied in detail the relationships among the four invariants, see \cite[Section 6.4]{CFZ23}. The following result, which will be of particular relevance to us, is a special case of their results. 

 \begin{lemma}[{\cite[Corollary 6.23]{CFZ23}}]
Let $K \subset \C \setminus \{ E, M\}$ be a closed curve that is a generic immersion. If both $w_E(K)$ and $w_M(K)$ are odd, then
\[
\mathcal{J}_E(K) = \mathcal{J}_M(K) = 2 \mathcal{J}_0(K)-1.
\]
\end{lemma}

We now restrict our attention to the case $\mathcal{B} \equiv 0$. In this setting, we extend the definitions of the four invariants to Stark systems and to show that they remain invariant under two-center Stark homotopies.  Following the approach in \cite{Kimcorr, KimSZ}, it suffices to verify that these invariants are preserved under the creation and annihilation of the distinguished periodic orbits described in Definition \ref{def:threeorbits}.

Let $K$ be a distinguished periodic orbit and denote by $\tilde K_1$ and $\tilde K_2$  two of its perturbations, each of which is a closed curve and a generic immersion in $\C \setminus \{ E, M\}$. We claim that the invariants of $K$ can be consistently defined as the invariants of any of its perturbations.

Suppose first that $K$ is a brake-brake orbit or a brake-collision orbit. Note that $\tilde{K}_1$ is obtained from $\tilde{K}_2$ by a change of orientation, the addition or removal of exterior loops, or a crossing through a triple point. These operations do not affect the $J^+$-invariant, nor do they alter the squared winding numbers that enter into the definitions of the four invariants. Therefore, the Cieliebak-Frauenfelder-Zhao invariants of $\tilde K_1, \tilde K_2$ coincide, which proves the claim in this case.

Now consider the case where $K$ is a collision-collision orbit of type I, i.e.~it collides twice with the same primary.  
Without loss of generality, we may assume that the collisions occur at $E$.  This situation coincides precisely with the one treated in \cite{Kimcorr}, so it  follows that $\mathcal{J}_0 (\tilde K_1)  = \mathcal{J}_0 (\tilde K_2)$. The same argument applies analogously to the other invariants.

 Finally, assume that $K$ is a collision-collision orbit of type II. As in the first case, the perturbations $\tilde K_1$ and $\tilde K_2$ differ only by orientation, the number of exterior loops, and the crossings through triple points. See Figure  \ref{collision-collision-orbit1}  and also \cite[Figure 7]{KimSZ}. Since $w_E(\tilde K_j)^2 = w_M(\tilde K_j)^2=1, j=1,2$, it follows that $\mathcal{J}_0(\tilde K_1) = \mathcal{J}_0 (\tilde K_2)$. We now apply regularization: either of $E$ or $M$ via the Levi-Civita regularization $L$, or of both $E$ and $M$ via the Birkhoff map $B$. In either case, each connected component of $L^{-1}(K)$ and $B^{-1}(K)$ remains a collision-collision  orbit of type II. Therefore,   arguing in a similar way, the invariants $\mathcal{J}_E$, $\mathcal{J}_M$ and $\left( \mathcal{J}_{E,M}, n\right)$ are well-defined. This concludes the proof of the claim.

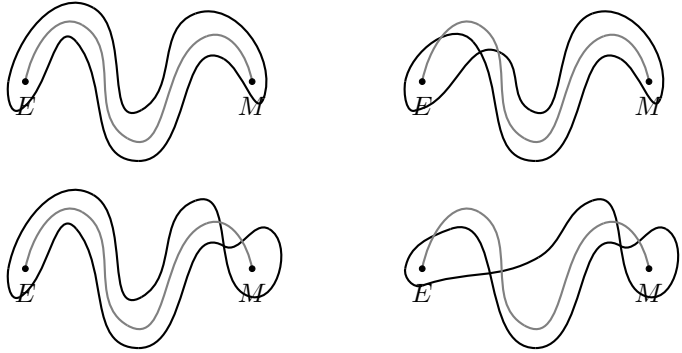
\begin{figure}[ht]
     \centering
\begin{tikzpicture}[scale=0.75]

\draw[  thick] (-2.3, 0) [out=80, in=150] to (-0.8,1.3);
\draw[  thick] ( -0.8   ,1.3  ) [out=330  , in= 210  ] to (0.05    ,  -0.5  );
\draw[  thick] (  0.05  ,  -0.5) [out= 390 , in= 200  ] to ( 1   , 1.2  );
\draw[  thick] (  1  , 1.2 ) [out= 380 , in= 70  ] to ( 2.2   , -0.3  );
\draw[  thick] ( 2.2   , -0.3 ) [out= 250 , in=  330 ] to (  1.5  , 0.4  );
\draw[  thick] (  1.5  , 0.4 ) [out= 150 , in= 0  ] to (  0  , -1.4  );
\draw[  thick] (  0  , -1.4 ) [out= 180 , in= 310  ] to (  -1.1  ,  0.7  );
\draw[  thick] (  -1.1  ,  0.7 ) [out= 130 , in= 30  ] to ( -2.1   , -0.5  );
\draw[  thick] (   -2.1  ,   -0.5   ) [out= 210   , in= 260   ] to (  -2.3     , 0    );

\draw[gray,thick] (-2,0) [out=80, in=150] to (-1 , 1);
\draw[gray,thick] (-1 ,1) [out=330, in=150] to (-0.2, -1);
\draw[gray,thick] (-0.2,-1) [out=330, in=200] to (1.2, 0.8);
\draw[gray,thick] (1.2,0.8) [out=380, in=100] to (2, 0 );
  \filldraw[draw=black, fill=black] (-2,0) circle (0.05cm);
 \filldraw[draw=black, fill=black] (2,0) circle (0.05cm);
  \node  at (-2,-0.1) [below] { $E$};
  \node  at (2,-0.1) [below] { $M$};

\begin{scope}[xshift=7cm]

\draw[  thick] (-2.3, 0) [out=80, in=130] to (-1.1 , 0.7);
\draw[  thick] (   -1.1  ,  0.7    ) [out=  310  , in= 180   ] to ( 0     , -1.4     );

\draw[  thick] (  -0.1   ,  -0.5) [out= 330 , in= 200  ] to ( 1   , 1.2  );
\draw[  thick] (  1  , 1.2 ) [out= 380 , in= 70  ] to ( 2.2   , -0.3  );
\draw[  thick] ( 2.2   , -0.3 ) [out= 250 , in=  330 ] to (  1.5  , 0.4  );
\draw[  thick] (  1.5  , 0.4 ) [out= 150 , in= 0  ] to (  0  , -1.4  );
\draw[   thick] (   -2.1  ,   -0.5   ) [out= 210   , in= 260   ] to (  -2.3     , 0    );
\draw[  thick] (   -2.1, -0.5       ) [out= 10    , in= 150   ] to (   -0.6   , 0.5    );
\draw[ thick] (   -0.6  ,  0.5    ) [out= 330   , in= 150   ] to (   -0.1     ,  -0.5    );

\draw[gray,thick] (-2,0) [out=80, in=150] to (-1 , 1);
\draw[gray,thick] (-1 ,1) [out=330, in=150] to (-0.2, -1);
\draw[gray,thick] (-0.2,-1) [out=330, in=200] to (1.2, 0.8);
\draw[gray,thick] (1.2,0.8) [out=380, in=100] to (2, 0 );
  \filldraw[draw=black, fill=black] (-2,0) circle (0.05cm);
 \filldraw[draw=black, fill=black] (2,0) circle (0.05cm);
  \node  at (-2,-0.1) [below] { $E$};
  \node  at (2,-0.1) [below] { $M$};

\end{scope}

 \begin{scope}[yshift=-3.3cm]
     
\draw[  thick] (-2.3, 0) [out=80, in=150] to (-0.8,1.3);
\draw[  thick] ( -0.8   ,1.3  ) [out=330  , in= 210  ] to (0.05    ,  -0.5  );
\draw[  thick] (  0.05  ,  -0.5) [out= 390 , in= 200  ] to ( 1   , 1.2  );
\draw[  thick] (  1  , 1.2 ) [out= 380 , in= 170  ] to ( 2    , -0.5  );
\draw[  thick] ( 2    ,  -0.5    ) [out=   350 , in= 330   ] to (  2.3    , 0.7    );

 \draw[  thick] ( 2.3   , 0.7 ) [out= 150 , in=  330 ] to (  1.5  , 0.4  );
\draw[ thick] (  1.5  , 0.4 ) [out= 150 , in= 0  ] to (  0  , -1.4  );
\draw[  thick] (  0  , -1.4 ) [out= 180 , in= 310  ] to (  -1.1  ,  0.7  );
\draw[  thick] (  -1.1  ,  0.7 ) [out= 130 , in= 30  ] to ( -2.1   , -0.5  );
\draw[  thick] (   -2.1  ,   -0.5   ) [out= 210   , in= 260   ] to (  -2.3     , 0    );

\draw[gray,thick] (-2,0) [out=80, in=150] to (-1 , 1);
\draw[gray,thick] (-1 ,1) [out=330, in=150] to (-0.2, -1);
\draw[gray,thick] (-0.2,-1) [out=330, in=200] to (1.2, 0.8);
\draw[gray,thick] (1.2,0.8) [out=380, in=100] to (2, 0 );
  \filldraw[draw=black, fill=black] (-2,0) circle (0.05cm);
 \filldraw[draw=black, fill=black] (2,0) circle (0.05cm);
  \node  at (-2,-0.1) [below] { $E$};
  \node  at (2,-0.1) [below] { $M$};

\begin{scope}[xshift=7cm]

\draw[  thick] (-2.3, 0) [out=80, in=200] to (-1.5,0.7);
\draw[  thick] ( -1.5   , 0.7  ) [out=380  , in= 180  ] to (0     ,  -1.4  );
 \draw[  thick] (  0.05  ,  0.2) [out= 390 , in= 200  ] to ( 1   , 1.2  );
 \draw[  thick] (  1  , 1.2 ) [out= 380 , in= 170  ] to ( 2    , -0.5  );
 \draw[  thick] ( 2    ,  -0.5    ) [out=   350 , in= 330   ] to (  2.3    , 0.7    );

 \draw[  thick] ( 2.3   , 0.7 ) [out= 150 , in=  330 ] to (  1.5  , 0.4  );
\draw[ thick] (  1.5  , 0.4 ) [out= 150 , in= 0  ] to (  0  , -1.4  );
 \draw[  thick] (  -2.3   , 0 ) [out= 260 , in= 200  ] to (  -2  ,  -0.3  );
\draw[  thick] (  -2  ,  -0.3 ) [out= 380 , in= 210  ] to ( 0.05  , 0.2  );

\draw[gray, thick] (-2,0) [out=80, in=150] to (-1 , 1);
\draw[gray,thick] (-1 ,1) [out=330, in=150] to (-0.2, -1);
\draw[gray,thick] (-0.2,-1) [out=330, in=200] to (1.2, 0.8);
\draw[gray,thick] (1.2,0.8) [out=380, in=100] to (2, 0 );
  \filldraw[draw=black, fill=black] (-2,0) circle (0.05cm);
 \filldraw[draw=black, fill=black] (2,0) circle (0.05cm);
  \node  at (-2,-0.1) [below] { $E$};
  \node  at (2,-0.1) [below] { $M$};

\end{scope}
 
 \end{scope}

      \end{tikzpicture}
\caption{Perturbations of a collision-collision orbit of type II}
\label{collision-collision-orbit1}
 \end{figure}

\begin{proposition}
    The four invariants $\mathcal{J}_0$, $\mathcal{J}_E$, $\mathcal{J}_M$ and $\mathcal{J}_{E,M}$ are preserved under two-center Stark-homotopies.
\end{proposition}
\begin{proof}
  See a proof of \cite[Proposition 2.13]{Kimcorr}.
\end{proof}

The following proposition, which will be used in the proof of the main theorem in Section \ref{sec:main}, provides explicit formulas for the invariant $ \mathcal{J}_0 $ associated with the distinguished periodic orbits.

 \begin{proposition}\label{prop:formula of collision orbits}
Let $K$ be a distinguished orbit as in Definition \ref{def:threeorbits}, and suppose that it has $N$ self-intersection points. Then the invariant $\mathcal{J}_0(K)$ is given by
   \[
   \mathcal{J}_0(K) = \begin{cases} 
2N & \text{ if $K$ is a brake-brake orbit,} \\
2N +\frac{1}{2} &\text{ if $K$ is a  brake-collision orbit,} \\
   2N+2 & \text{ if $K$ is a  collision-collision orbit of type I,} \\ 2N+1 & \text{  if $K$ is a collision-collision orbit of type II.} \end{cases}
   \]
   \end{proposition}
\begin{proof} 
 In order to determine $J^+(K)$, it is sufficient -- as shown in \cite[Proposition 2.14]{KimSZ} -- to count the number of self-intersection points of $K$.  After a small perturbation each self-intersection point gives rise to four double points of a generic immersion. During a homotopy, this results in one crossing through a direct self-tangency and one crossing through an inverse self-tangency (possibly along with a finite number of crossings through triple points). Since such a perturbed curve is a generically homotoped to a circle, for a curve $K$ with $N$ self-intersection points, we have $J^+(K)=2N$.

It remains to determine the winding numbers. 
Let $\tilde K$ be a small perturbation of $K$. If $K$ is a brake-brake orbit or a brake-collision orbit, then the argument given in \cite[Proposition 2.14]{KimSZ} still applies. In particular, we obtain 
\[ 
w_E(K)^2 + w_M(K)^2 = \begin{cases} 0 & \text{ if $K$ is a brake-brake orbit},\\ 1 & \text{ if $K$ is a brake-collision orbit}. \end{cases}
\]

Suppose that $K$ is a collision-collision orbit of type I. In this case,  the value of $w_E(\tilde K) ^2 + w_M(\tilde K)^2 $ depends on the choice of the perturbation $\tilde K$ of $K$, and is either $0$ or $4$. As explained in \cite[Section 1]{Kimcorr}, the distinction  between these two cases lies in the behavior of the homotopy from $\tilde K$ to a circle: when the sum is $0$, there occurs an additional crossing through a direct self-tangency, which increases the  $J^+$-invariant by $2$. Therefore, $\mathcal{J}_0$ takes the same value in both cases.  
Now assume that $K$ is a collision-collision orbit of type II. It suffices to consider the four perturbations illustrated in Figure \ref{collision-collision-orbit1}. In each case, we find that $w_E^2(\tilde K) = w_M^2(\tilde K)=1$, so that $w_E^2(\tilde K) + w_M^2(\tilde K)=2$. This finishes the proof.    
\end{proof}

\section{Euler's two-center problem} \label{Sec:lemni}

In this section, we review some basic facts of the Euler problem and collect several properties of lemniscate motions that will be used in the next section to calculate the invariants.  See \cite[Section 3]{WDR04} for further details. 

 As explained in the introduction, the Hamiltonian in elliptic coordinates \eqref{eq:ellipticco} takes the form 
\[
H = \frac{H_{\lambda} + H_{\nu}}{\cosh^2\lambda - \cos^2 \nu}
\]
where $H_\lambda = \frac{1}{2}p_\lambda^2 - \cosh \lambda$ and $H_\nu= \frac{1}{2}p_\nu^2 +(1-2\mu)\cos\nu.$ Define the function
\[
G =  \frac{H_{\lambda}\cos^2 \nu + H_{\nu}\cosh^2\lambda}{\cosh^2\lambda - \cos^2 \nu}
\]
A direct computation shows that $\{H, G\}=0$ and the differentials $dH$ and $dG$ are linearly independent almost everywhere, which implies that $G$ is a first integral of the system.

 The image of the energy-momentum map $(\lambda, \nu) \mapsto (G(\lambda, \nu), H(\lambda, \nu))$ consists of four regions (when $\mu \neq \frac{1}{2}$) or three regions (when $\mu=\frac{1}{2}$) in the lower-half $(G,H) =(g,c)$-plane, as illustrated in Figure \ref{fig:orbitstypes}. Each point in these regions corresponds to a Liouville torus. The boundaries of these regions are defined by the curves 
\begin{align*}
     \ell_{1,2} : c = g \pm (1-2\mu), \quad  \ell_3 : c = g-1, \quad \ell_4 : 4gc = -(1-2\mu)^2, \quad  \ell_5 : 4gc =-1 
\end{align*}
 along which the differentials $dH$ and $dG$ become linearly dependent. 
 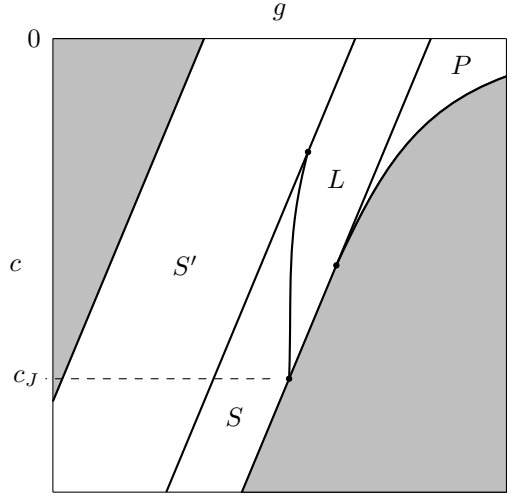
\begin{figure}[ht]
     \centering
\begin{tikzpicture} [scale=0.5]
	\draw [fill=lightgray] plot [  tension=0.7] coordinates { (-6,6)  (-2,6)  (-6, -3.6)};
	\draw [fill=lightgray] plot [  tension=0.7] coordinates {   (6,5) (6,-6) (-1,-6)   (1.5,0)   };	
	\draw[thick, name path=one] (1.5,0) [out=65, in=200] to (6, 5);
	 	\draw [name path=two] plot [ tension=0.7] coordinates {    (6,5) (6,-6) (-1, -6)  (1.5,0)    };
  \tikzfillbetween[   of=one and two,split ] {fill=lightgray};
\draw (-6,-6) to (-6,6);
\draw (-6,-6) to (6,-6);
\draw (6,-6) to (6,6);
\draw (6,6) to (-6,6);
\draw[thick ] (4,6) to (-1,-6);
\draw[thick ] (2,6) to (-3,-6);
\draw[thick ] (-2,6) to (   -6,-3.6);
\draw[thick  ] (0.25, -3) [out=88, in=255] to ( 0.75, 3);
\draw[thick ] (1.5,0) [out=65, in=200] to (6, 5);
      \draw[fill]  (0.25,-3) circle [radius=0.07];
     \draw[fill]  (0.75,3) circle [radius=0.07];
     \draw[fill]  (1.5,0) circle [radius=0.07];
\node at (0,6.7) {$g$};
\node at (-7,0) {  $c$};
 \draw[dashed] (-0.25,-3) to (-6.2,-3);
\node at (-6.7,-3) {  $c_{J}$};
\node at (-6.5, 6) {$0$};
\node at (-2.5,0) {  $S'$};
\node at (-1.2,-4) {  $S$};
\node at (1.5,2.3) {   $L$};
\node at (4.8,5.3) {  $P$};

 
  \end{tikzpicture}
    \caption{Non-symmetric case $\mu \neq \frac{1}{2}$}
    \label{fig:orbitstypes}
 \end{figure} 
We label the regions by $S',S$(satellite), $L$(lemniscate), and $P$(planetary). 
Note that the curves $\ell_3$ and $\ell_4$ intersects precisely at the critical value $c = c_J, $  and the regions $P$ and $L$ exist only for $c>c_J.$ The nature of the motion in each region is as follows:  
\begin{itemize}

\item Satellite motions ($S$ and $S'$-regions):   \\
In the $S$-region, the variables   take values in  intervals $\lambda \in [-\lambda_0,\lambda_0]$ and $\nu \in [-\nu_0,\nu_0]\cup[\pi-\nu_1,\pi+\nu_1] $  for some  $\lambda_0>0$ and $0<\nu_0, \nu_1<\pi ,$ showing that  the motion is localized near either primary.  In contrast, the $S'$-region corresponds to motion localized only near the   heavier primary.

\item Lemniscate motions ($L$-region):  \\
The variable  $\lambda$ is again bounded by $\lambda \in [-\lambda_0, \lambda_0],$ but $\nu$ is unrestricted.  Thus, the motion occurs within a single ellipse defined by  $\lambda = \lambda_0,$ and the trajectory traverses both primaries.

\item Planetary motions ($P$-region):  \\
There exist $0<\lambda_0<\lambda_1$ such that $\lambda \in [-\lambda_1,-\lambda_0]\cup[\lambda_0,\lambda_1],$ with $\nu$ again unrestricted. Therefore, the motion is confined between two ellipses   $\lambda = \lambda_0$ and $\lambda =\lambda_1,$ circling around both primaries.

\end{itemize}

From now on we pay attention to lemniscate motions.
Fix $\mu \in (0,1)$ and $c \in (c_J, 0)$.
Let $\gamma(t)= (\lambda(t), \nu(t), p_{\lambda}(t), p_{\nu}(t))$ be a Hamiltonian trajectory lying on the regularized energy level $F_{\mu,c}^{-1}(0)$, see \eqref{eq:FofEuler}. We denote by 
\[
\gamma_1(t) = (\lambda(t), p_{\lambda}(t)) \quad \text{ and } \quad  \gamma_2(t) = (\nu(t), p_{\nu}(t))
\]
the projections of $\gamma(t)$ onto the $(\lambda, p_\lambda)$-plane and the $(\nu, p_\nu)$-plane, respectively.

\begin{lemma}\label{lem:phaseportrait}  Let $\gamma(t)  = (\lambda(t), \nu(t), p_{\lambda}(t), p_{\nu}(t))\in F_{\mu,c}^{-1}(0)$ be a Hamiltonian trajectory such that $(\lambda(t), \nu(t))$ describes a lemniscate motion and $\gamma_1(t),\gamma_2(t)$ be as defined above.   Then the following assertions hold:
\begin{enumerate}
    \item The projection $\gamma_1(t)$ is a non-trivial simple closed curve that winds around the origin in the clockwise direction. It is symmetric with respect to both the $\lambda$- and $p_{\lambda}$-axes.

    \item  Along $\gamma_2(t)$ the sign of $p_{\nu}(t)$ is constant, i.e.~either    $p_\nu(t)>0$ or $p_\nu(t)<0$ for all $t$. Denote by $\gamma_2^\pm(t)$ the trajectory for which $\pm p_\nu(t)>0$. Along $\gamma_2^+(t)$, the function $v(t)$ is strictly increasing in $t$, while along $\gamma_2^-(t)$, it is strictly decreasing. The curves $\gamma_2^\pm(\R)$ are related by reflection with respect to  the $p_\nu$-axis.  Moreover, the restriction $\gamma_2^\pm(t)$ to $ t \in [0,2\pi] $ is symmetric with respect to $\nu =\pi$, and its restriction to $t \in [-\pi,\pi]$ is symmetric with respect to $\nu=0$, see Figure \ref{fig:phase}.
    \begin{figure}[h]
  \centering
  \includegraphics[width=0.6\linewidth]{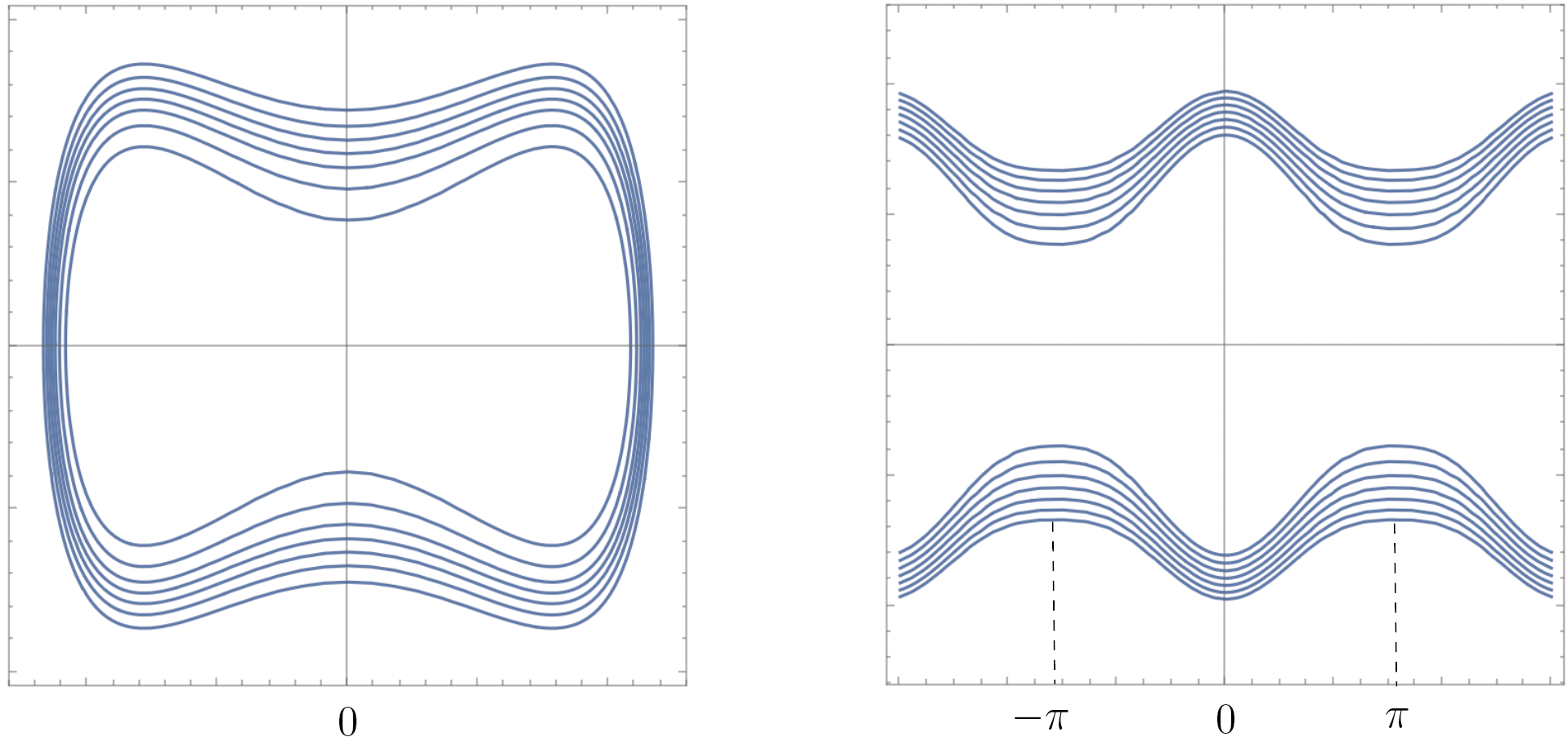}
 \caption{The projections $\gamma_1$ (left) and $\gamma_2$ (right) of $\gamma(t)$}
 \label{fig:phase}
\end{figure}
\end{enumerate}
\end{lemma}
\begin{proof}  See, for instance, \cite[Section 4]{Ver14} or \cite[Section 3]{WDR04}.
\end{proof}

If $\gamma$ is a brake-brake orbit or a brake-collision orbit, then there exists $t_0$ such that $(\lambda(t_0), \nu(t_0) )  $ lies on the boundary of Hill's region.  Equivalently, this means that $p_\lambda( t_0) = p_\nu( t_0)=0$. However, the previous lemma tells us that for lemniscate motions $p_\nu$ never vanishes. Therefore, we obtain the following assertion. 

\begin{corollary}\label{cor:nobrake}
A lemniscate motion does not admit braking points. Hence,   brake-brake orbits and brake-collision orbits cannot be realized as lemniscate motions. 
\end{corollary}

We now turn to study of collision-collision orbits that arise as lemniscate motions.

\begin{lemma}\label{lem:Lregion}
     On each $T_{k,l}$-torus corresponding to lemniscate motions, there exist precisely two collision-collision orbits, denoted by $K_1, K_2$. If $l$ is odd, then both orbits are of type II and related by reflection  with respect to  the $q_1$-axis.   
 If $l$ is even, then each orbit involves collisions with a single primary and is symmetric under reflection  with respect to  the $q_1$-axis.  See Figure \ref{fig:cc}.
 \begin{figure}[h]
  \centering
  \includegraphics[width=0.75\linewidth]{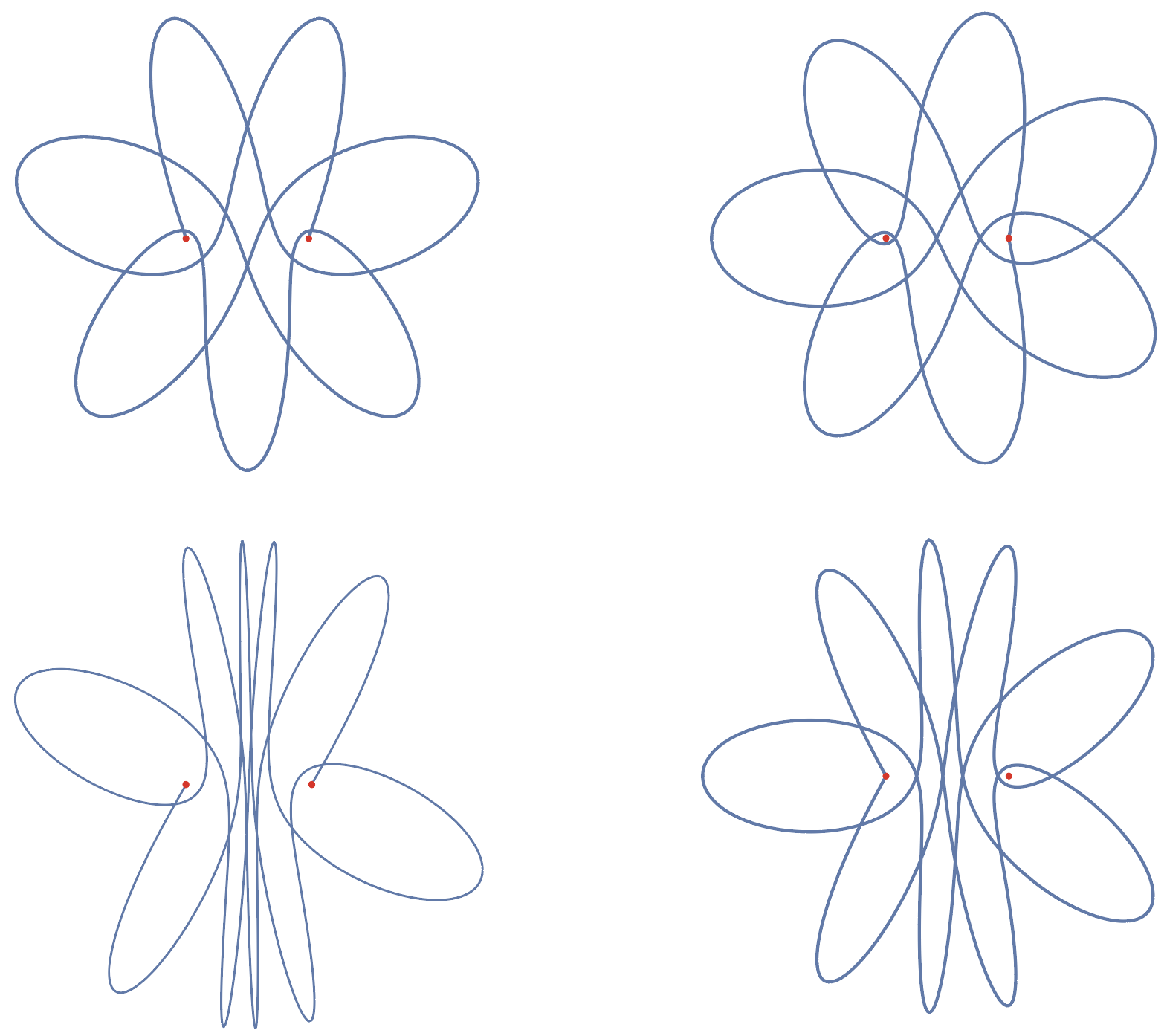}
 \caption{Examples of  collision-collision orbits in the Euler problem. (top left) $k=7, \; l=5$; (top right) $k=7,\;  l=6$; (bottom left) $k=10, \;  l=3$; (bottom right) $k=9,\; l=4$}
 \label{fig:cc}
\end{figure}
    
\end{lemma}
\begin{proof} 

We follow the arguments given in \cite[Proposition 3.5]{KimSZ} and \cite[Section 5]{Ver14}.

  Let $\gamma(t)$ be a $T$-periodic Hamiltonian trajectory corresponding to a lemniscate motion, lying on a $T_{k,l}$-torus.   Since the rotation number satisfies $R=k/l  $ (see the introduction for the definition of $R$), the projections $\gamma_1$ and $\gamma_2$ complete $k$ and $l$ full cycles, respectively. Consequently, the period $T$ satisfies $T = k T_\lambda= l T_\nu$.

Suppose that $\gamma(t)$ admits a collision. Without loss of generality, we may assume that the collision occurs at  $E$,  and reparametrize $\gamma(t)$   so that $\gamma(0)=  (0, -\pi, p_\lambda^0, p_\nu^0) $ with  $p_\lambda^0, p_\nu^0>0$. Because of Corollary \ref{cor:nobrake}, $\gamma$ is a collision-collision orbit. The  second collision then occurs at time $t=T/2$, and  there are  four possible configurations, illustrated   in Figure \ref{fig:four}.
\begin{figure}[h]
  \centering
  \includegraphics[width=0.9\linewidth]{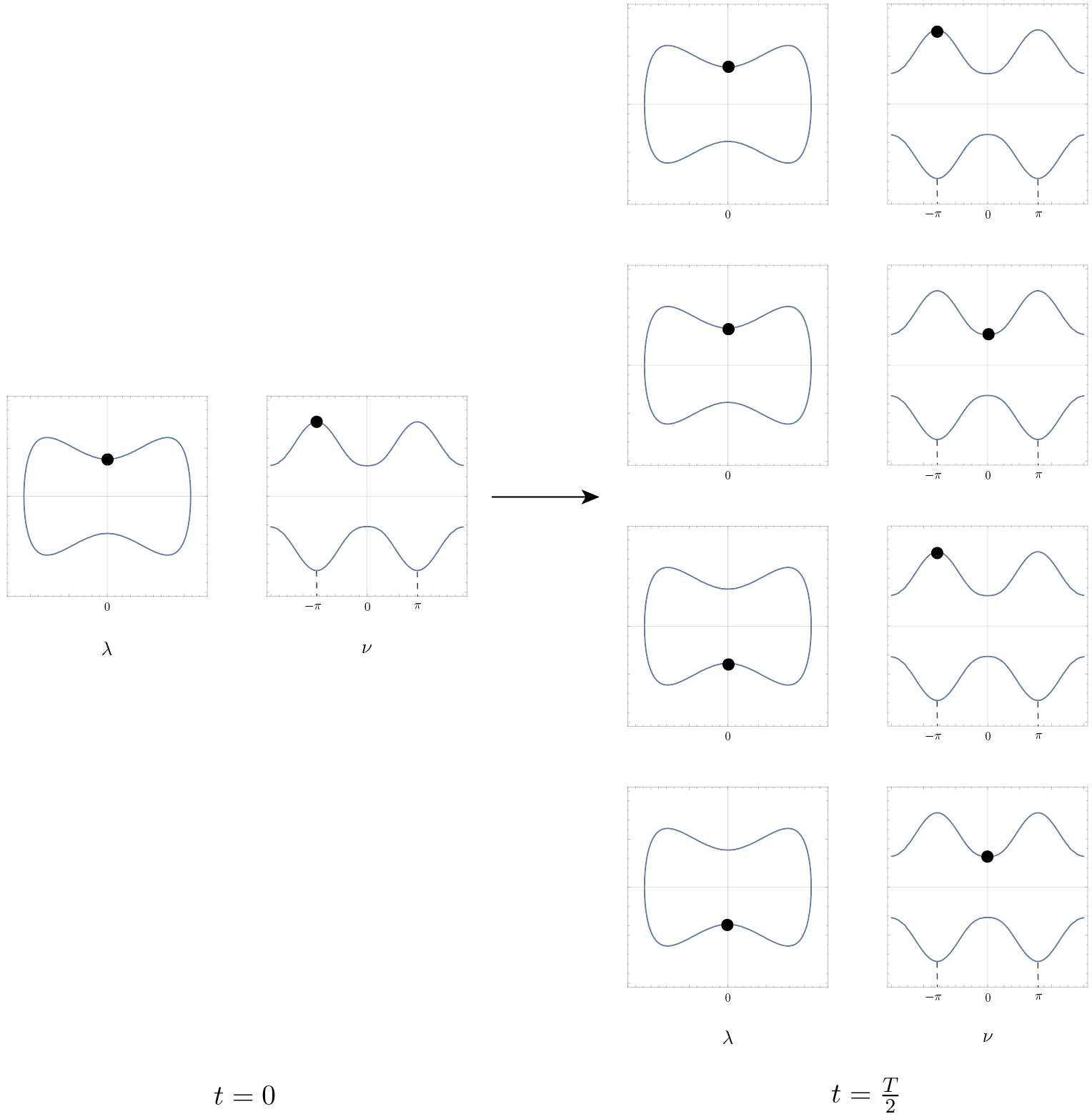}
 \caption{Four possible configurations for the second collision. The first case is excluded, as it contradicts the assumption gcd$(k,l)=1$.  The second and last  cases  correspond to simple covered collision-collision orbits of type II. The third case indicates a double covered collision-collision orbit of type I}
 \label{fig:four}
\end{figure}
The first case is excluded since $k$ and $l$ are assumed to be relatively prime. 

In the second case, we have $k$ is even and $l$ is odd, and the trajectory $\gamma$ collides with $M$ at time $t =T/2$. Hence, $\gamma$ is a simple covered $T$-periodic collision-collision orbit of type II. Since the Hamiltonian $F_{\mu,c}$ is invariant under the anti-symplectic involution $\mathcal{R} \colon (\lambda, \nu, p_\lambda, p_\nu) \mapsto (-\lambda, \nu, p_\lambda, -p_\nu)$, which corresponds to reflection in the  $q_1$-axis,  it follows that $\mathcal{R}(\gamma)$ is also a simple covered $T$-periodic collision-collision orbit of type II.   

In the third case, $k$ is odd and $l$ is even, and $\gamma$ collides again with $E$ at time $t =T/2$. This implies that $\gamma$ is a double-covered $T$-periodic collision-collision orbit of type I. Moreover, this orbit is invariant under the involution $\mathcal{R}$. A similar argument applies if the initial condition is assumed to occur at $M$ instead of $E$.  

In the last case, where both $k$ and $l$ are odd, the situation is equivalent to the second case. 
This completes the proof.  
 \end{proof}

   \section{Calculations of the invariants}\label{sec:main}

   In this section, we compute the four Cieliebak-Frauenfelder-Zhao invariants  for the lemniscate-type  periodic motions in Euler's two-center problem.
   Since every $T_{k,l}$-torus is a Stark homotopy and contains a collision-collision orbit, see Lemma \ref{lem:Lregion}, it suffices to determine the invariants of these orbits. See Figures \ref{fig:exampleofc-cwithk4l3} and \ref{fig:exampleofc-cwithk5l4}.

\begin{figure}[ht]
\centering
\begin{tikzpicture} [scale=0.7]

\draw[gray] (-3,0) to (3,0);

 \begin{scope}[scale=0.9]

 \node[below, red] at (-1.5,-0.05) {\small $E$};
  \node[above, red] at (1.5, 0.05) {\small $M$};


\draw[thick]       (-1.5,0)  [out=100, in=180] to (-1.5,3); 
\draw[thick]       (-1.5, 3)  [out=0, in=180] to ( 1.5,-0.5); 
 \draw[thick]       (1.5, -0.5)  [out= 0, in=0] to (2.8,2); 
\draw[thick]       (2.8, 2)  [out=180, in= 0] to (-2.8,-2); 
\draw[thick]       (-2.8, -2)  [out=180, in=180] to (-1.5,0.5); 
\draw[thick]       (-1.5,0.5)  [out=0, in=180] to ( 1.5,-3); 
\draw[thick]       ( 1.5, -3)  [out= 0, in=280] to (1.5,0);

\end{scope}

\node at (0.7, 2) {  $K$};
       \filldraw[draw=red, fill=red] (-1.35,0) circle (0.1cm);
       \filldraw[draw=red, fill=red] (1.35,0) circle (0.1cm);

  \draw[thick, ->] (3.7,1.5) [out=60, in=200] to (6,3); 
\node[below] at (4.9,2.3) {\small $L_M^{-1}$};
\node[below] at (4.5,3.4) {\small $L_E^{-1}$};

\begin{scope}[yscale=-1]
  \draw[thick, ->] (3.7,1.5) [out=60, in=200] to (6 ,3); 
\node[below] at ( 4.9,1.4) {\small $B^{-1}$};
\end{scope}


\begin{scope}[xshift=10cm, yshift=3cm]
\draw[gray] (-3.2,0) to (3.2,0);

\begin{scope}[scale=1.1]
  \draw[thick ]       (-2  ,0)  [out=110, in=180] to (-1.8,  2); 
   \draw[thick ]       (-1.8 ,2)  [out= 0, in=180] to ( 0.3, -2.2); 
  \draw[thick ]       (0.3 , -2.2)  [out= 0, in=180] to (1.7, 0.7); 
  \draw[thick ]       (1.7  ,0.7)  [out= 0, in=90] to (2.8, -0.1); 
    \draw[thick ]       (2.8  ,-0.1)  [out=270, in=0] to (2.2, -0.7); 
  \draw[thick ]       ( 2.2 ,-0.7)  [out=180, in=0] to (1.3, 2  ); 
  \draw[thick ]       (1.3  , 2 )  [out=180, in=0] to (-1.3, -2 ); 
  \draw[thick ]       (-1.3  ,-2)  [out=180, in=0] to (-2.2, 0.7); 
  \draw[thick ]       (-2.2  ,0.7)  [out=180, in=90] to (-2.8,  0.1); 
  \draw[thick ]       (-2.8  ,0.1)  [out=270, in=180] to (-1.7, -0.7); 
  \draw[thick ]       (-1.7  , -0.7)  [out=0, in=180] to (-0.3,  2.2); 
  \draw[thick ]       (-0.3  ,  2.2)  [out= 0, in=180] to (1.8, -2); 
  \draw[thick ]       ( 1.8  ,-2)  [out= 0, in=290] to (2,0); 

\end{scope}


   \filldraw[draw=red, fill=red] (-2.2 ,0) circle (0.1cm);
       \filldraw[draw=red, fill=red] (2.2,0) circle (0.1cm);

\end{scope}


\begin{scope}[xshift=10cm, yshift=-4cm]

\draw[gray] (-3,0) to (3,0);

   \begin{scope}[scale=0.9]
        \draw[thick ]       (  1.5  ,0 )  [out= 300, in=0] to ( 1.2 ,-3); 
  \draw[thick ]       (  1  .2 , -3)  [out= 180, in=270] to (- 1,-0.5); 
  \draw[thick ]       ( - 1, -0.5)  [out=  90, in=180] to ( 2, 1.5); 
  \draw[thick ]       (  2,  1.5)  [out= 0, in=90] to ( 3, 0.5); 
     \draw[thick ]       (  3,  0.5)  [out= 270, in=300] to (-1.5,0);

     \draw[thick ]       ( -1.5  ,0 )  [out= 120, in=180] to (-1.2 ,3); 
  \draw[thick ]       ( -1  .2 , 3)  [out= 0, in=90] to (1,0.5); 
  \draw[thick ]       (1 , 0.5)  [out= 270, in=0] to (-2,-1.5);

  \draw[thick ]       ( -2, -1.5)  [out= 180, in=270] to (-3,-0.5); 
     \draw[thick ]       ( -3, -0.5)  [out= 90, in=120] to (1.5,0); 

   \end{scope}


   \filldraw[draw=red, fill=red] (-1.35,0) circle (0.1cm);
       \filldraw[draw=red, fill=red] (1.35,0) circle (0.1cm);

\end{scope}


      \end{tikzpicture}
      \caption{An example of a  collision-collision orbit of type II lying on a $T_{4,3}$-torus (left), together  with its regulariazions via the Levi-Civita map (top right) and  the Birkhoff map (bottom right) }
      \label{fig:exampleofc-cwithk4l3}
 \end{figure}
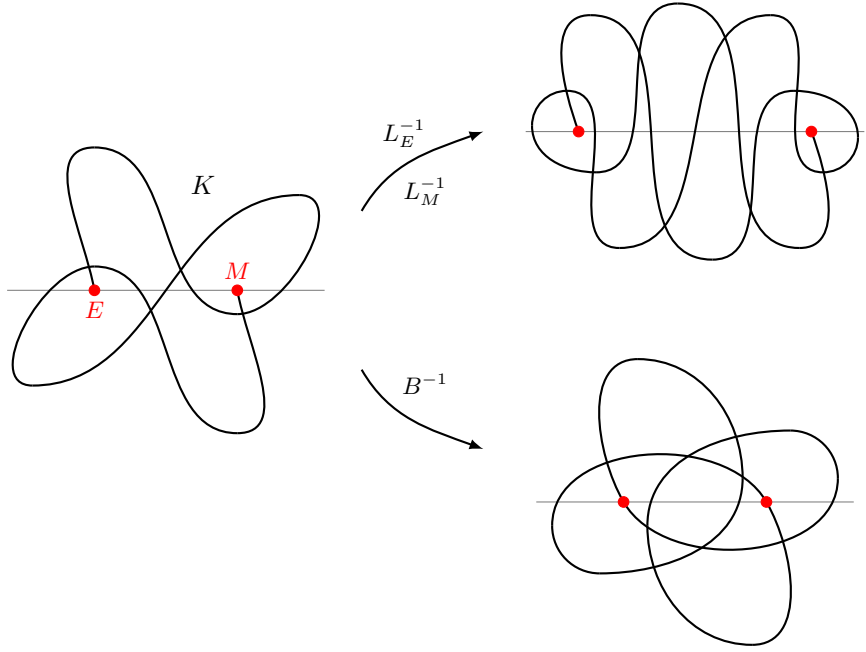




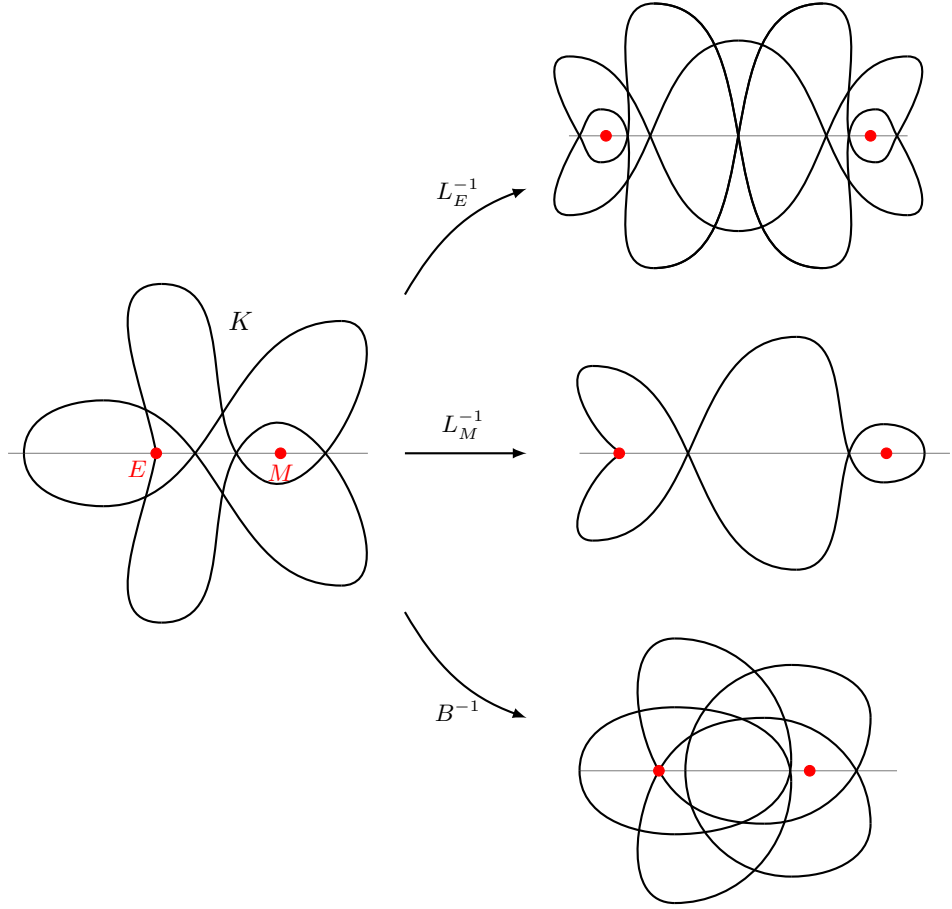
\begin{figure}[ht]
\centering
\begin{tikzpicture} [scale=0.7]

\draw[gray] (-3.8,0) to (3,0);



\draw[thick]       (-1,0)  [out=100, in=180] to (-0.9 ,3.2); 
 \draw[thick]       (-0.9,3.2)  [out= 0, in=150] to ( 1, -0.5); 
\draw[thick]       ( 1,-0.5)  [out=330, in=0] to (2.5 ,2.5); 
\draw[thick]       (2.5,2.5)  [out=180, in=0] to (-2 ,-1); 
\draw[thick]       (-2, -1)  [out=180, in=270] to (-3 .5,0); 

\begin{scope}[yscale=-1]

\draw[thick]       (-1,0)  [out=100, in=180] to (-0.9 ,3.2); 
 \draw[thick]       (-0.9,3.2)  [out= 0, in=150] to ( 1, -0.5); 
\draw[thick]       ( 1,-0.5)  [out=330, in=0] to (2.5 ,2.5); 
\draw[thick]       (2.5,2.5)  [out=180, in=0] to (-2 ,-1); 
\draw[thick]       (-2, -1)  [out=180, in=270] to (-3 .5,0);

\end{scope}


\node at (0.6, 2.5) {  $K$};
       \filldraw[draw=red, fill=red] (-1,0) circle (0.1cm);
       \filldraw[draw=red, fill=red] (1.35,0) circle (0.1cm);
 \node[left, red] at (-1,-0.3) {\small $E$};
  \node[below, red] at (1.35, -0.03) {\small $M$};

  \draw[thick, ->] (3.7,3) [out=60, in=200] to (6,5); 
  \node[above] at (4.7,4.5) {\small $L_E^{-1}$};

  \draw[thick, ->] (3.7,0)  to (6,0); 
\node[above] at (4.8,0.1) {\small $L_M^{-1}$};

\begin{scope}[yscale=-1]
  \draw[thick, ->] (3.7,3) [out=60, in=200] to (6,5); 
\node[below] at ( 4.7,4.5) {\small $B^{-1}$};
\end{scope}


\begin{scope}[xshift=10cm, yshift=6cm]
\draw[gray] (-3.2,0) to (3.2,0);

 
  \draw[thick ]       (-3  ,0)  [out=120, in=180] to (-3.2,  1.5);  
  \draw[thick ]       (-3.2  ,1.5)  [out=0, in=180] to (0,  -1.8); 
   \draw[thick ]       (0, -1.8)  [out= 0, in=180] to ( 3.2 ,  1.5); 
  \draw[thick ]       ( 3 .2 , 1.5)  [out=0, in=60] to (3,  0); 
  \draw[thick ]       (3  ,0)  [out=240, in=0] to (2.6 ,  -0.5); 
  \draw[thick ]       (2.6  ,-0.5)  [out=180, in=0] to (1 .6,  2.5); 
  \draw[thick ]       (1.6 , 2.5)  [out=180, in=0] to (-1.6,  -2.5); 
    \draw[thick ]       (-3  ,0)  [out=300, in=180] to (-2.6 ,  -0.5);  
  \draw[thick ]       (-2.6  ,-0.5)  [out=0, in=180] to (-1.6, 2.5);  
  \draw[thick ]       (-1.6, 2.5)  [out=0, in=180] to (1.6,  -2.5);  

\begin{scope}[yscale=-1]
  \draw[thick ]       (-3  ,0)  [out=120, in=180] to (-3.2,  1.5);  
  \draw[thick ]       (-3.2  ,1.5)  [out=0, in=180] to (0,  -1.8); 
   \draw[thick ]       (0, -1.8)  [out= 0, in=180] to ( 3.2 ,  1.5); 
  \draw[thick ]       ( 3 .2 , 1.5)  [out=0, in=60] to (3,  0); 
  \draw[thick ]       (3  ,0)  [out=240, in=0] to (2.6 ,  -0.5); 
  \draw[thick ]       (2.6  ,-0.5)  [out=180, in=0] to (1 .6,  2.5); 
  \draw[thick ]       (1.6 , 2.5)  [out=180, in=0] to (-1.6,  -2.5); 
    \draw[thick ]       (-3  ,0)  [out=300, in=180] to (-2.6 ,  -0.5);  
  \draw[thick ]       (-2.6  ,-0.5)  [out=0, in=180] to (-1.6, 2.5);  
  \draw[thick ]       (-1.6, 2.5)  [out=0, in=180] to (1.6,  -2.5);  
\end{scope}



   \filldraw[draw=red, fill=red] (-2.5  ,0) circle (0.1cm);
       \filldraw[draw=red, fill=red] (2.5 ,0) circle (0.1cm);

\end{scope}


\begin{scope}[xshift=10cm, yshift=0cm]
\draw[gray] (-3 ,0) to (4,0);

\begin{scope}[scale=1.1]
  \draw[thick ]       (-2  ,0)  [out=150, in=180] to (-2.5,  1.5); 
  \draw[thick ]       (-2.5  ,1.5)  [out= 0, in=180] to (1 ,  -2); 
  \draw[thick ]       (1  , -2)  [out= 0, in=180] to (2.5,  0.5); 
  \draw[thick ]       (2.5  , 0.5)  [out= 0, in=90] to (3.2,  0); 

\begin{scope}[yscale=-1]
  \draw[thick ]       (-2  ,0)  [out=150, in=180] to (-2.5,  1.5); 
  \draw[thick ]       (-2.5  ,1.5)  [out= 0, in=180] to (1 ,  -2); 
  \draw[thick ]       (1  , -2)  [out= 0, in=180] to (2.5,  0.5); 
  \draw[thick ]       (2.5  , 0.5)  [out= 0, in=90] to (3.2,  0); 
\end{scope}

\end{scope}


   \filldraw[draw=red, fill=red] (-2.25 ,0) circle (0.1cm);
       \filldraw[draw=red, fill=red] (2.8,0) circle (0.1cm);

\end{scope}


\begin{scope}[xshift=10cm, yshift=-6cm]

\draw[gray] (-3,0) to (3,0);

        \draw[thick ]       (  -1.5  ,0 )  [out= 120, in=180] to (- 1.2 , 2.5); 
        \draw[thick ]       (  -1.2  ,2.5)  [out=  0, in=90] to (  1  ,  0.2); 
        \draw[thick ]       ( 1, 0.2 )  [out=  270, in=0] to (  -1.2  ,  -1.2); 

        \draw[thick ]       (  -1.2  , -1.2 )  [out=  180, in=270] to (  -3  ,  0 ); 
        \draw[thick ]       (  -3, 0  )  [out=  90, in=180] to ( -1.2  ,  1.2); 
        \draw[thick ]       (  -1.2  , 1.2 )  [out=  0, in=90] to (  1  ,  -0.2); 
        \draw[thick ]       (  1  , -0.2 )  [out=  270, in=0] to (  -1.2  , - 2.5); 
        \draw[thick ]       (  -1.2  ,-2.5)  [out=  180, in=240] to (  -1.5  ,  0 ); 

        \draw[thick ]       (  -1.5  , 0 )  [out=  60, in=180] to (  0.5  ,  1); 
        \draw[thick ]       (  0.5, 1)  [out=  0, in=90] to (  2.5  , -1  ); 
        \draw[thick ]       (   2.5,-1   )  [out=  270, in= 0] to (  1  , -2); 
        \draw[thick ]       (   1, -2 )  [out=  180, in= 270] to (  -1  ,0); 
        \draw[thick ]       (   -1,0 )  [out=  90, in= 180] to (  1  ,  2); 
        \draw[thick ]       (   1, 2 )  [out=  0, in= 90] to (  2.5,1 ); 
        \draw[thick ]       (   2.5,1 )  [out=  270, in= 0] to (  0.5  , -1); 
        \draw[thick ]       (   0.5, -1 )  [out=  180, in= 300] to (   -1.5,0);




   \filldraw[draw=red, fill=red] (-1.5,0) circle (0.1cm);
       \filldraw[draw=red, fill=red] (1.35,0) circle (0.1cm);
\end{scope}


      \end{tikzpicture}
            \caption{An example of a  collision-collision orbit of type I lying on a $T_{5,4}$-torus (left), together with its regularizations via the Levi-Civita map around $E$ (top right), the Levi-Civita map around $M$ (middle right), and the Birkhoff map (bottom right)}
      \label{fig:exampleofc-cwithk5l4}
 \end{figure}

   \subsection{The invariant $\mathcal{J}_0$}

   Let $K$ be a  collision-collision orbit lying on a $T_{k,l}$-torus. In view of Proposition \ref{prop:formula of collision orbits}, in order to determine $\mathcal{J}_0(K)$ it suffices to count self-intersection points along $K$.

   \begin{proposition}\label{prop:quadruple}
   Every collision-collision orbit $K$ lying on a $T_{k,l}$-torus has $N$ self-intersection points, where
   \[
   N= \begin{cases}\displaystyle  \frac{k(l-1) }{2} - \frac{1}{2}   & \text{ if $l$ is even,} \\ \\ \displaystyle \frac{k(l-1)}{2}   & \text{ if $l$ is odd.}\end{cases}
   \]
   \end{proposition}
\noindent
\begin{proof}  We proceed following the approach of the proofs in \cite[Theorem 4.8]{Kimcorr} and \cite[Proposition 4.4]{KimSZ}.  

\smallskip

 \noindent \textit{Case 1.} $l$ is even.

We first assume that $l$ is even, and let $K$ denote the unique $q_1$-symmetric collision-collision orbit of type I which collides with $E$ and lies on a $T_{k,l}$-torus. Let $\gamma$ be a $T$-periodic parametrization of $K$ such that $\gamma(0) = (0,-\pi, p_\lambda^0, p_\nu^0)$ with $p_\lambda^0, p_\nu^0>0$, as in the proof of Lemma \ref{lem:Lregion}. From that proof, it follows that $K$ is a double covered $T$-periodic orbit, and thus it suffices to consider the restriction $\gamma|_{[0,T/2]}$. By abuse of notation, we shall use the same symbol $\gamma$ for the restriction and identify it with its image on the $q$-plane. Note that every self-intersection point of $\gamma$ corresponds to a double point.  

Since $\gamma$ intersects the $q_1$-axis if and only if $\lambda=0$, $\nu=0$ or $\nu=-\pi$ (mod $2\pi$), and since the variables $\lambda$ and $\nu$ complete $k/2$ and $l/2$ cycles along $\gamma$, respectively, it follows that $\lambda=0$ is attained precisely $k$ times, and $\nu=0$ and $\nu=-\pi$ are each attained precisely $l/2$ times, see Lemma \ref{lem:phaseportrait}. Therefore, in the interval $(0,T/2]$, the trajectory $\gamma$ intersects the $q_1$-axis exactly $k+l -1$ times. The subtraction of $1$ accounts for the point $(\lambda(T/2), \nu(T/2))=(0,-\pi) $, corresponding to the second collision at $E$; the remaining intersections are non-collisional. 
By examining the definition of the elliptic coordinates \eqref{eq:ellipticco}, it is easy to show that $\gamma$  intersects the $q_1$-axis perpendicularly if and only if, at the intersection point, either  $\lambda$ or $\nu$ attains its maximum or minimum. Since $\gamma$ is symmetric with respect to the $q_1$-axis, it follows that $\gamma$ has a single point on the $q_1$-axis at $t=T/4$  and $(k+l-3)/2$ double points. At the single point, we have
\[
\lambda = \begin{cases} \lambda_{\max} & \text{ if }   k \in 4 \N_0+1, \\ - \lambda_{\max} & \text{ if } k \in 4 \N_0+3,\end{cases} \quad \quad \nu  = \begin{cases}  - \pi & \text{ if } l \in 4 \N , \\ 0 & \text{ if } l \in 4 \N_0+2,                    \end{cases} 
\]
where $\lambda_{\max}>0$ denotes the unique value of $\lambda$ for which   $p_\lambda=0$.
Therefore, the single point at $t =T/4$ lies on the side opposite to $E$ or $M$, depending on whether $l \in 4 \N$ or if $l \in 4 \N_0+2$, respectively.
See Figure \ref{fig:ex1}.

\begin{figure}[h]
  \centering
  \includegraphics[width=0.6\linewidth]{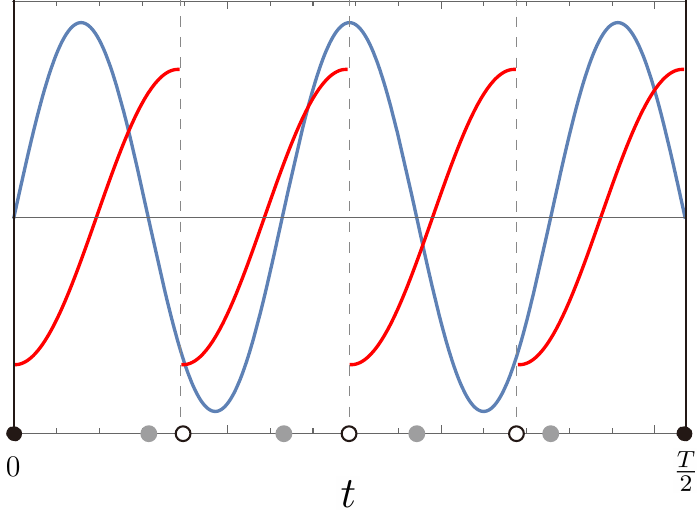}
 \caption{An example of the functions $\lambda$ (blue) and $\nu$ (red) for $(k,l)=(5,8)$. The black, gray, and white dots indicate, respectively, the collisions with $E$, the points where $\lambda=0$, and the points where $\nu=-\pi$ (mod $2\pi$). The middle white dot represents the single point}
 \label{fig:ex1}
\end{figure}

Let $\gamma^\pm  = \gamma \cap \{ \pm q_2 \geq 0\}$. Note that $\gamma^\pm$ are related by reflection with respect to the $q_1$-axis, and hence contain the same number of double points. Moreover, by construction,  neither $\gamma^+$ nor $\gamma^-$  has any double points lying on the $q_1$-axis. Hence, in order to determine the number of double points along $\gamma$, it is sufficient to count those along $\gamma^+$.
We may view $\gamma^+$ as a $T/4$-periodic trajectory starting at $E$ and ending at the single point.

 Let us now treat  $\lambda$ and $\nu$ as functions of time $t \in [0,T/4]$.   Reflecting the negative parts of their graphs with respect to the $q_1$-axis, we obtain
 \[
 (\lambda(T/4), \nu(T/4))=\begin{cases} (\lambda_{\max}, \pi ) & \text{ if } l \in 4 \N , \\  (\lambda_{\max}, 0 ) & \text{ if  }   l \in 4 \N_0+2. \end{cases}
 \]
    Along $\gamma^+$ the variable $\lambda$ performs $k$  quarter-cycles, where each   quarter-cycle corresponds to increase or decrease of $\lambda$ in $[0,\lambda_{\max}]$. Similarly, the variable $\nu$ undergoes $l/2$ half-cycles,  each representing increase or decrease of $\nu$ in $[0,\pi]$.

\medskip

\noindent
{\it {\bf Claim1.} If $t_0 \in (0,T/4)$ corresponds to a double point along $\gamma^+$, then there exist integers $\alpha, \beta \in \Z$ such that $t_0 = (2\alpha k + \beta l )T/4kl$.}

To prove the claim, let     $(\lambda_0 , \nu_0)= (\lambda(t_0), \nu ( t_0))$  for some $t_0 \in (0,T/4)$ be a double point of $\gamma^+$. Let $t_1 \in (0,T/4k)$ be the minimal positive time such that  $\lambda(t_1) = \lambda_0$. Then   $\lambda(t) = \lambda_0$ occurs at the following times:
   \[
  t= t_1, \frac{2T}{4k} \pm t_1, \frac{4T}{4k} \pm t_1, \ldots, \frac{(k-3)T}{4k} \pm t_1, \frac{(k-1)T}{4k} \pm t_1.
   \]
Similarly, let  $t_2 \in (0,T/2l)$ be the minimal positive time such that   $\nu(t_2) = \nu_0$. Then   $\nu(t) = \nu_0$ holds at:
\[
 t = \begin{cases} \displaystyle t_2 ,  \frac{4T}{4l} \pm t_2, \frac{8T}{4l} \pm t_2, \ldots, \frac{ (l-4)T}{4l}\pm t_2,  \frac{lT}{4l}- t_2 & \text{ if }  l \in 4\N, \\ \\  \displaystyle  t_2 ,  \frac{4T}{4l} \pm t_2, \frac{8T}{4l} \pm t_2, \ldots, \frac{(l-6)T}{4l}\pm t_2  ,\frac{(l-2)T}{4l}\pm t_2   & \text{ if }  l \in 4\N_0+2. \end{cases}
 \]
 See Figure \ref{fig:ex22}.

\begin{figure}[h]
  \centering
  \includegraphics[width=0.9\linewidth]{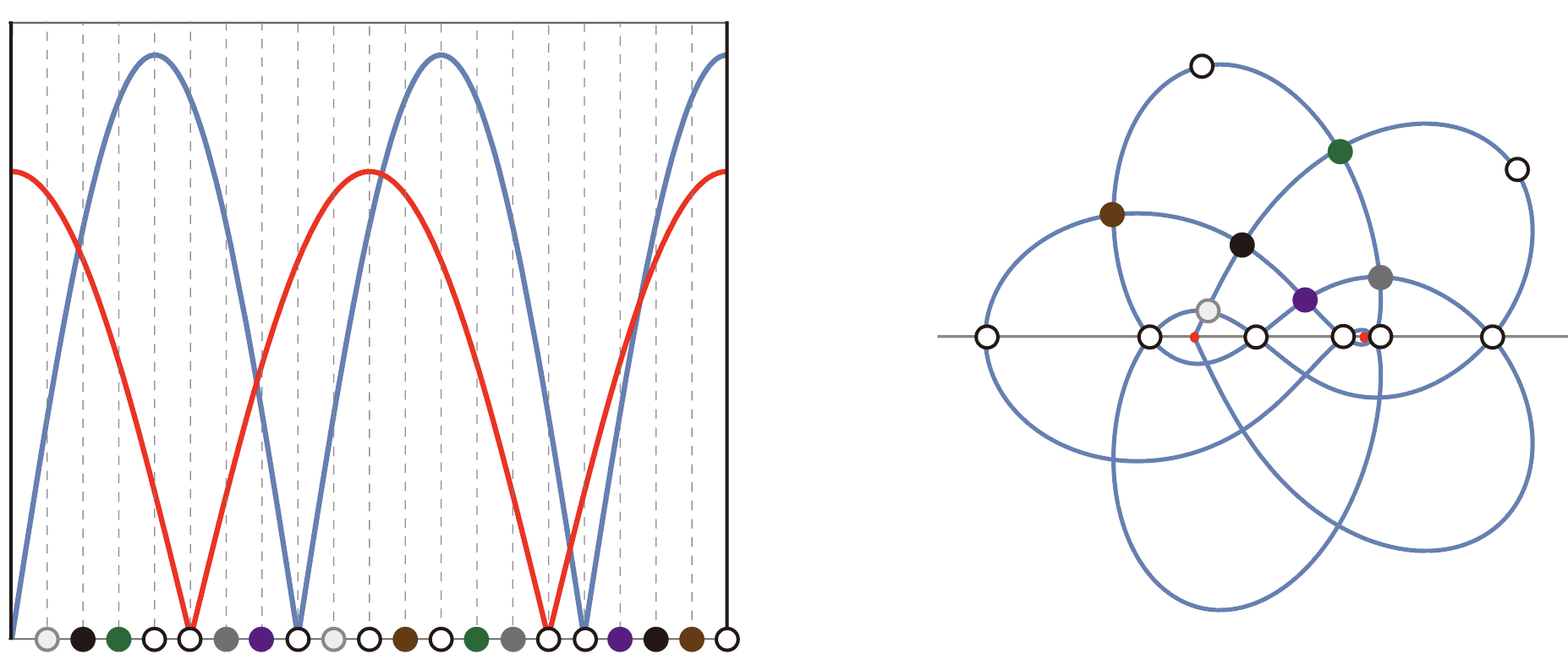}
 \caption{The case $(k,l)=(5,8)$.  Left: The graphs of $\lambda= \lambda(t)$ and $\nu=\nu(t)$ are shown in blue and red,  respectively. The interval $[0,T/4]$ is divided into $20$ subintervals by dashed vertical lines. The white dots indicate the maxima and minima of $\lambda$ or $\nu$, and each pair of the colored dots represents a double point of $\gamma^+$;  Right:  A collision-collision orbit on a $T_{5,8}$-torus }
 \label{fig:ex22}
\end{figure}

We first consider the case $l \in 4 \N$. Suppose   that  for some $0\leq j \leq (k-1)/2$ and $0\leq i \leq l/4-1$, the time $t_0 \in (0, T/4)$ satisfies 
\begin{equation}\label{eq:t0}
t_0 = \frac{ 2j}{4k}T + t_1 = \frac{  4i}{4l}T + t_2,
\end{equation}
where $t_1 \in (0, T/4k)$ and $t_2 \in (0, T/2l)$ are the minimal times such that $\lambda(t_1)=\lambda_0$ and $\nu(t_2) = \nu_0$, respectively, as above. Then, rearranging the equation we obtain
\[
t_1 = \frac{   2( 2ik - jl)}{4kl}T + t_2.
\]

  

Since $(\lambda_0, \nu_0)$ is a double point, there exist integers $m \neq i$ and $n \neq j$ such that one of the following equalities holds: (i) $(2n/4k)T +     t_1 = (4m/4l)T +t_2$, (ii) $(2n/4k)T - t_1 = (4m/4l)T+t_2$, (iii) $(2n/4k)T + t_1 = (4m/4l)T - t_2$, or (iv) $(2n/4k)T - t_1  =  (4m/4l)T - t_2$.

Assume the first case, so that 
\[
t_1 = \frac{ 2 ( 2mk-nl)}{4kl}T + t_2.
\]
Substituting into \eqref{eq:t0}, we obtain $2(i-m)k = (j-n)l$. Since $\lvert j-n \rvert <k$ and $2 \lvert i-m \rvert < l$, this would imply $k=l$, which contradicts the assumption that $k$ and $l$ are relatively prime. A similar contradiction arises in the last case, where the same reasoning leads to  $n=j=k/2$ and $m=i=l/4$, contradicting   $m \neq i $ and $n \neq j$. In the second case we find 
\[
t_1 = \frac{ (n-j)l + 2(i-m)k}{4kl}T \quad \Longrightarrow \quad t_0=   \frac{  2(i-m)k + (n+j)l }{4kl}T .
\]
The third case can be treated analogously, from which the claim follows in the case $l \in 4 \N$. A similar argument applies to the case $l \in 4 \N_0+2$, from which the claim is proved.

\medskip

Consider the set
\[
S= \left\{  \frac{ 2 \alpha k + \beta l }{4kl}T  \in \bigg[0, \frac{T}{4}\bigg] \relmiddle| \alpha, \beta \in \Z \right\} .
\]
 By Claim1, for each double point $(\lambda_0, \nu_0)$ of $\gamma^+$ there  exist distinct times $t_0 \neq t_1 \in S$ such that $(\lambda_0, \nu_0) = ( \lambda( t_0) , \nu(t_0)) = (\lambda( t_1), \nu( t_1))$.  Among 
the $\# S = kl/2 +1$ points in $S$, the following are excluded:
\begin{itemize}
    \item the $k+1$ points   $t = jT/4k, j=0,1, \ldots, k$, at which   $\lambda = 0 $ or $\lambda =  \lambda_{\max}$;  

    \item  the  $ l/2+1 $ points $t = 2iT/4l, i=0, 1,\ldots, l/2$, at which   $\nu = 0$ or $\nu=  \pi$.

\end{itemize}
These points do not represent double points of $\gamma^+$, and then by  
\begin{equation}\label{eq:twosets}
\left\{ \frac{ jT}{4k} \relmiddle| j=0, 1, \ldots, k \right\} \cap \left\{ \frac{ 2iT}{4l} \relmiddle| i=0,1,\ldots, \frac{l}{2} \right\} = \left\{0,\frac{T}{4}\right\},
\end{equation}
the number of points in $S$ that do not correspond to maxima or minima of either variable is
\[
 \frac{kl}{2} +1 - \left(  k+1  +  \frac{l}{2} +1  - 2 \right) = (k-1)\left( \frac{l}{2}-1 \right).
 \]

\medskip

\noindent
{\it {\bf Claim2.} The $(k-1)(l/2-1)$ points represent double points of $\gamma^+$.}

 Let $t_0 \in S$ be a time that does not belong to either of the two sets on  the left-hand side of \eqref{eq:twosets}. Pick integers $\alpha, \beta, \alpha', \beta' \in \Z$ such that
 \[
0< \frac{ 2\alpha k + \beta l}{4kl}T <\frac{T}{4k}  \quad \text{ and } \quad  
0<  \frac{ 2\alpha' k + \beta' l}{4kl}T <\frac{T}{2l} 
 \]
and define $m = 2\alpha k + \beta l $ and $n = 2 \alpha' k + \beta' l$.  

 As in the proof of the previous claim,  the times at which $\lambda = \lambda_0 := \lambda(t_0)$ occurs in the set 
   \[
A = \left\{ \frac{mT}{4kl}, \frac{2lT}{4kl} \pm \frac{mT}{4kl}, \frac{4lT}{4kl} \pm \frac{mT}{4kl}, \ldots,   \frac{(k-1)lT}{4kl} \pm \frac{mT}{4kl} \right\} 
   \]
Similary, the times at which    $\nu = \nu_0 := \nu(t_0)$ occur belong to the set:
\begin{itemize}
    \item if $l \in 4\N$, then 
\[
B =   \left\{ \frac{nT}{4kl} ,  \frac{4kT}{4kl} \pm \frac{nT}{4kl}, \frac{8kT}{4kl} \pm \frac{nT}{4kl}, \ldots, \frac{ (l-4)kT}{4kl}\pm \frac{nT}{4kl},  \frac{klT}{4kl}- \frac{nT}{4kl}  \right\} 
\]
\item if $l \in 4\N_0+2$, then  

\[
B=  \left\{ \frac{nT}{4kl},   \frac{4kT}{4kl} \pm \frac{nT}{4kl}, \frac{8kT}{4kl} \pm \frac{nT}{4kl}, \ldots, \frac{(l-6)kT}{4kl}\pm \frac{nT}{4kl}  ,\frac{(l-2)kT}{4kl}\pm \frac{nT}{4kl}  \right\} 
\]
\end{itemize}
In order to prove the claim we have to show that $\# A \cap B =2$.
Since $t_0 \in A \cap B$ by construction,  it suffices to show that there exists exactly one additional element in $A \cap B$  other than $t_0$.

We first consider the case $l \in 4 \N_0$. Suppose   that the following identity holds for some  $0\leq a \leq (k-1)/2$ and $0 \leq b \leq (l-4)/4  $: 
\begin{equation}\label{eq:mnn}
\frac{ 2al}{4kl}T +\frac{m}{4kl}T = \frac{ 4bk}{4kl}T + \frac{n }{4kl}T.
\end{equation}
This is equivalent to $2al+m = 4bk -n$.
Since $1 \leq m \leq l-1$ and $ k$ and $l$ are relatively prime, there exist $P,Q \in \Z$ such that $m=kP - lQ$. We define $i,j$ as
\[
i = \begin{cases} \frac{1}{2}(P-2b) &  P>2b, \\ \frac{1}{2}(2b-P) & P < 2b, \\ 0 & P = 2b \end{cases} \quad \quad j = \begin{cases} Q-a &  Q>a, \\ a-Q & Q < a, \\ 0 & Q = a \end{cases}
\]

Consider the case where $i=\frac{1}{2}(P-2b)$ and $j=Q-a$, so that $m=   k (2i+2b) - l(j+a) $.   Using \eqref{eq:mnn} we find 
\begin{equation}\label{eq:mn}
 2jl+m = 4ik -n,
\end{equation}
implying  that  
\[
\frac{ 2jl+m}{4kl}T = \frac{ 4ik-n}{4kl}T \in A \cap B.
\]
That is,
\[
 \lambda\left( \frac{ 2al+m}{4kl}T \right) = \lambda\left( \frac{ 2jl+m}{4kl}T \right) \quad \text{ and } \quad 
  \nu\left( \frac{ 2al+m}{4kl}T \right) = \nu\left( \frac{ 2jl+m}{4kl}T \right) .
\]
It remains to show that $ a \neq j$. Assume by contradiction that $a=j$. Then combining \eqref{eq:mnn} and \eqref{eq:mn}, we obtain $4bk + n = 4ik - n$, and hence $  n = 2(i - b)k$. However, this contradicts the assumption that   $1 \leq n \leq 2k-1$,   showing that $a \neq j$ and  $(\lambda_0, \nu_0)$ is indeed a double point.

In the case $i=\frac{1}{2}(P-2b)$ and $j=a-Q$, an argument analogous to the previous case yields  $2jl-m = -4ik +n$.
However, $(-4ik+n)T/4kl \notin S$, 
implying that this case cannot happen. The remaining seven cases can be treated in a similar way.  

The above reasoning also applies to case $l \in 4\N_0+2$, from which the claim is proved.  

\medskip

Claim2 shows that the positive part $\gamma^+$ contains exactly $ (k-1)(l/2-1)/2$ double points. Taking into account the symmetric contribution from the negative part $\gamma^-$, as well as the $(k+l-3)/2$ double points lying on the $q_1$-axis (see the previous discussion), we conclude that when $l$ is even, the unique $q_1$-symmetric collision-collision orbit contains precisely
\[
(k-1)\left( \frac{l}{2}-1\right)   + \frac{k+l-3}{2}  =  \frac{kl-k-1}{2}
\]
double points. This completes the proof of the proposition in the case $l$ is even.

 \medskip

 \noindent \textit{Case 2.} $l$ is odd.

Now suppose that $l$ is odd. In this case, there are two collision-collision orbits on each $T_{k,l}$-torus, both of which collide with both primaries and are related by the $q_1$-axis reflection, see Lemma \ref{lem:Lregion}.  We fix one of these orbits and denote it by $K$. 
Since $K$ is not $q_1$-symmetric, the argument used in the case $l$ is even cannot be applied directly. To analyze this case, we require a different approach, for which we begin with some preliminary observations.

\medskip

\noindent
{\it {\bf Claim3.}   The collision-collision orbit $K$ does not have double points on the $q_1$-axis.}


Let $\gamma$ be a $T$-periodic parametrization of $K$, chosen such that $\gamma(0)=(0,-\pi, p_\lambda^0, p_\nu^0)$, corresponding to a collision at $E$. It suffices to consider its restriction to $[0,T/2]$, and again by abuse of notation we continue to denote the restriction by  $\gamma$. According to the proof of    Lemma \ref{lem:Lregion} we have 
\[
 \gamma  ( T/2) = \begin{cases}  (0,0, p_{\lambda}^0, p_{\nu}') & {\text{if $k$ is even,}} \\   (0,0, -p_{\lambda}^0 , p_{\nu}') & {\text{if $k$ is odd,}}
\end{cases}
\]
for some $p_\nu'>0$, corresponding a collision at $M$. 



 We now reflect the negative part of $\gamma$ with respect to the $q_1$-axis and regard the resulting trajectory as one that does not pass through points with $q_1<0$. 
As before, we treat $\lambda$ and $\nu$ as functions of time $t \in [0,T/2]$ and reflect the negative parts of their graphs with respect to the $q_1$-axis.  
Note that $\lambda$ completes $2k$ quarter-cycles, each corresponding to increase or decrease of $ \lambda $ in $[0, \lambda_{\max}]$, and $\nu$ completes $l$ half-cycles, each corresponding to increase or decrease of $ \nu $ in $[0,\pi]$.

  Suppose that $(\lambda(t_0), \nu(t_0))=(\lambda_0, \nu_0) $ for some $t_0 \in (0,T/2)$ is a double point on the $q_1$-axis. Then either   $\lambda_0=0$, $\nu_0=0$ or $\nu_0=\pi$. We first consider the case $\lambda_0=0$, in which case
\[
t_0 \in  \left\{ \frac{2jl}{4kl}T \relmiddle| j=1, \ldots, k-1 \right\}.
\]
Note  that $\nu_0 \neq 0, \pi$, since otherwise $(\lambda_0, \nu_0)$ represents a collision. 

Let $t_2 \in (0, T/2l)$ be the minimal positive time such that $\nu(t_2)=\nu_0$. Then $\nu(t) = \nu_0$ holds at the following times:
\[
t = t_2, \frac{2T}{2l} \pm t_2, \frac{4T}{2l} \pm t_2, \ldots,  \frac{(l-3)T}{2l} \pm t_2,\frac{(l-1)T}{2l} \pm t_2.
\]
 Suppose first that 
 \[
 t_0 = \frac{2jl}{4kl}T =  \frac{2ik}{4kl}T +t_2
 \]
for some $1 \leq j \leq k-1$ and $1 \leq i \leq  l-1 $. Then we have
\[
t_2 = \frac{ 2( jl - ik)}{4kl}T.
\]
Since $(\lambda_0, \nu_0)$ is a double point, there exist integers $n\neq j$ and $m \neq i $ such that  $(2nl / 4kl)T = (2mk/4kl)T \pm t_2$. Arguing as before we find $(n \mp j )l = (m \mp i )k$. However, this implies $\lvert m \mp i\rvert \geq l$, contradicting the assumption that  $1 \leq m \leq k-1$.

For completeness, one must also consider the case
\[
t_0 = \frac{2j l }{4 k l}T = \frac{ 2i k }{4 k l }T-t_2,
\]
but a similar argument leads to the same contradiction. This finishes the proof of the claim.


\medskip

We now reflect the negative part of $K$ with respect to the $q_1$-axis: denoting by $K^\pm$ its positive and negative parts, respectively, consider the union $\bar K:=K^+ \cup K^-$. We may think of it as a piecewise smooth orbit (not smooth at points lying on the $q_1$-axis) that collides with both $E$ and $M$. Let $\bar{\gamma}$ be its $T$-periodic parametrization, chosen so that the satellite collides with $E$ at $t=0$ and with $M$ at $t=T/2$. As before, we denote its restriction to $[0,T/2]$   by $\bar{\gamma}$.

\medskip

\noindent
{\it {\bf Claim4.} Every self-intersection point of $\bar{\gamma}$ is a double point.}

As before, denote by $\lambda_{\max}$ the maximum value of the variable $\lambda$ along $K$. It suffices to show that for every point $(\lambda_0, \nu_0) \in (-\lambda_{\max}, \lambda_{\max}) \times ( -\pi, \pi)$ the ellipse $\lambda = \lambda_0$ and the hyperbola $\nu=\nu_0$ intersect at most at two points. This follows easily by inspecting the Hamiltonian \eqref{eq:FofEuler}. Note that if the intersections consist of two points, then they are either two single points $(q_1, \pm q_2)$ that are mirror images with respect to the $q_1$-axis or a double point.

\medskip

Arguing as in Claim1, we find that  any time $t_0 \in (0,T/2)$ representing a double point of $\bar{\gamma}$ lies   in  the set
\begin{equation*}\label{eq:SSSset}
\bar{S} = \bigg\{ \frac{2\alpha k + \beta l}{4kl} T \in \bigg[ 0,\frac{T}{2}\bigg]  \; \bigg| \; \alpha, \beta \in \Z    \bigg\}
\end{equation*}
with cardinality $\# \bar{S} = 2 kl +1$. As shown in   Claim2, among these,   the $2k+l$ points correspond  to the maxima and minima of $\lambda$ and $\nu$.  Hence, the remaining $(2k-1)(l-1)$ points represent the double points of $\bar{\gamma}$.

 Note that reflecting the negative part $K^- $ with respect to the $q_1$-axis introduces additional intersection points -- namely, the points where $K^-$ intersects with $K^+$.  Therefore,  in order to determine the exact number of self-intersection points of $K$, such mixed intersections between $K^\pm$ must be excluded from the count.



\medskip

\noindent
{\it {\bf Claim5.} If $t_0 \in (0, T/2)$ corresponds to a double point along $\gamma$, then it must be of the form  $t_0 = \left( 2\alpha k + (2 \beta +1)l \right)T / 4kl$ for some $\alpha, \beta \in \Z$.}

As in the proof of Claim1, let $(\lambda_0, \nu_0) = (\lambda(t_0), \nu(t_0))$ for some $t_0 \in (0, T/2)$ be a double point of $\gamma$.  
Let   $t_1 \in (0,T/4k) $ be the least time such that $\lambda(t_1)=\lambda_0$. Then $\lambda(t) = \lambda_0$ occurs at the following times:
\begin{itemize}
 \item if $k$ is odd, then     
\[
t =\begin{cases}     \displaystyle     \frac{T}{4k} \pm t_1, \frac{5T}{4k} \pm t_1, \ldots,\frac{(2k-5)T}{4k}\pm t_1   , \frac{(2k-1)T}{4k}\pm t_1                & \text{ if } \lambda_0 >0 , \\         \\    \displaystyle  \frac{3T}{4k} \pm t_1 , \frac{7T}{4k}\pm t_1, \ldots,  \frac{(2k-7)T}{4k}\pm t_1   ,   \frac{  (2k-3)T }{4k}\pm t_1      & \text{ if } \lambda_0 <0,        \end{cases}
\]

\item if $k$ is even, then 
\[
t =\begin{cases}    \displaystyle    \frac{T}{4k} \pm t_1, \frac{5T}{4k} \pm t_1, \frac{(2k-7)T}{4k}\pm t_1   ,\ldots, \frac{(2k-3)T}{4k}\pm t_1              & \text{ if } \lambda_0 >0 , \\       \\  \displaystyle   \frac{3T}{4k} \pm t_1 , \frac{7T}{4k}\pm t_1, \ldots,  \frac{(2k-5)T}{4k}\pm t_1   ,   \frac{  (2k-1)T }{4k}\pm t_1       & \text{ if } \lambda_0 <0.        \end{cases}
\]

\end{itemize}
We also let    $t_2 \in (0,T/2l)$ be the least time such that $\nu(t_2)=\nu_0$. Then  $\nu(t)= \nu_0$ occurs at the following times:
   \[
t =\begin{cases}    \displaystyle   \frac{2T}{2l}-t_2, \frac{4T}{2l}-t_2, \ldots,     \frac{(l-3)T}{2l}-t_2  , \frac{(l-1)T}{2l}-t_2         & \text{ if } \nu_0 >0 , \\       \\  \displaystyle  t_2,  \frac{2T}{2l}+t_2,   \ldots,     \frac{(l-3)T}{2l}+t_2  , \frac{(l-1)T}{2l}+t_2     & \text{ if } \nu_0 <0.       \end{cases}
\]

We consider the case where $k$ is odd. Assume that both $\lambda_0$ and $\nu_0$ are positive, and suppose that    
\[
t_0 = \frac{(4j+1)lT}{4kl} + t_1  = \frac{4ikT}{4kl}-t_2
\]
for some integers $  0 \leq  j  \leq (k-1)/2 $ and $1 \leq i \leq (l-1)/2$.
As in the previous argument,  we find   $m \neq i$ and $n \neq j$ such that either (i) $ (4n+1)lT/4kl + t_1 = 4mkT/4kl - t_2$; or (ii)  $ (4n+1)lT/4kl - t_1 = 4mkT/4kl - t_2$. The first case leads to a contradiction with the fact that $k$ and $l$ are relatively prime integers. Assuming the last case, we find
\[
t_0 = \frac{  (2n+2j+1)l + 2(i-m)k}{4kl}T,
\]
which is of the  desired form and thus proves the claim in this setting. The remaining cases--such as when either $\lambda_0<0$ or $\nu_0$, or both--can be treated in an analogous manner.

The same reasoning applies to the case where $k$ is even. This finishes the proof of the claim.

\medskip

Consider the subset
\[
S_0 = \left\{   \frac{ 2\alpha k + (2\beta+1) l }{4kl}T        \relmiddle| \alpha, \beta \in \Z \right\}  \subset \bar{S} .
\]
The cardinality of $S_0$ is $\# S_0 =  kl$. Among these, the $k$ points $(2j+1)T/4k,0\leq j \leq  k-1$ correspond to the maximum of $\lambda $ and do not represent double points of $K$.   
By Claim3, which asserts that there are no double points on the $q_1$-axis, it follows that when $k=l=1$,  the collision-collision orbit $K$ admits no self-intersection points. This proves the  assertion of the proposition in the case $k=l=1$.

Assume now that $l \geq 3$. Since every point in the set $\bar{S}$, excluding those   corresponding to the maximum and minimum of the variables $\lambda$ and $\nu$, represents a double point of $\bar{\gamma}$, Claim5 tells us that the remaining $\# S_0 - k =  k (l-1)$ points in $S_0$ represent self-intersections of the  $\gamma$, see Figure \ref{fig:ex2}.
\begin{figure}[h]
  \centering
  \includegraphics[width=1.0\linewidth]{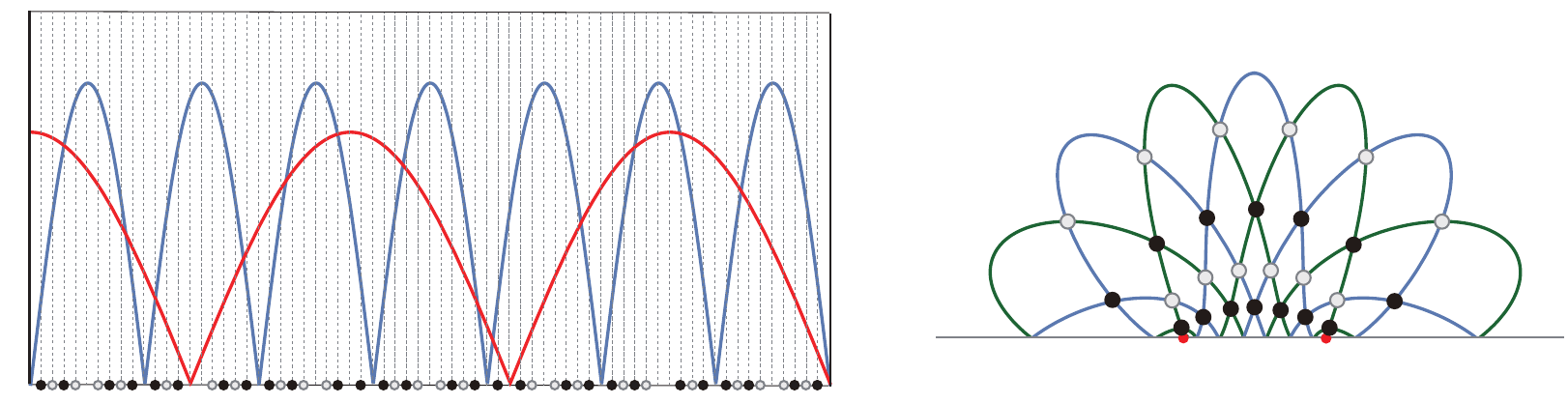}
 \caption{The case $(k,l)=(7,5)$.  Left: The graphs of $\lambda= \lambda(t)$ (blue) and $\nu=\nu(t)$ (red) over  the   interval $[0,T/2]$, which is divided into $70$ equal parts by dashed vertical lines. The black dots represent self-intersections of $K$, while the gray dots correspond to intersections between $K^\pm$. Right: The orbit $\bar{K}$, obtained by taking the union of $K^+$ (blue) and the reflection of $K^- $ (green)  with respect to the $q_1$-axis. Here $K^\pm$ indicate  the positive and negative parts of a collision-collision orbit $K$ lying on a $T_{7,5}$-torus, respectively}
 \label{fig:ex2}
\end{figure}  This proves the assertion of the proposition when $l$ is odd. 
  This completes the proof of the proposition. 
\end{proof}

   \begin{remark}\label{rmk:useful} 
In this remark we collect several facts about collision-collision orbits that follow from the arguments used in the proof of the previous proposition. These facts will be used in the computation of the invariant $(\mathcal{J}_{E,M},n)$ in Section \ref{subsec:EM}.

 \medskip

 Suppose that $l$ is even. Let $K$ be  a collision-collision orbit at $E$, lying on a $T_{k,l}$-torus. We denote its  parametrization again by $\gamma$.

   \medskip

\noindent
{\bf Fact(1e).} The intersections of $\gamma$ with the $q_1$-axis consist of:
\begin{itemize}
\item[-] the two collision points at $E$ and $M$;

\item[-] the single point $P$; and

\item[-] $(k+l-3)/2$ double points.

\end{itemize}
The order of these points along the $q_1$-axis depend on the parity of $l$ mod $4$ are as follows:
\begin{itemize}
\item if $l \in 4\N$, then the ordering is
\[
  P, \;  (\tfrac{l}{4}-1) \text{ double points}, \;  E, \;  \tfrac{k-1}{2} \text{ double points} , \; M, \; \tfrac{l}{4} \text{ double points}  
\]

\item if $l \in 4\N_0+2$, then the ordering is  
\[
    \tfrac{l-2}{4}  \text{ double points}, \;  E, \;  \tfrac{k-1}{2} \text{ double points} , \; M, \; \tfrac{l-2}{4} \text{ double points}, \; P
\]
     
\end{itemize}

   \medskip

\noindent
{\bf Fact(2e).} In the proof of Claim2, we showed that each double point of the positive part $\gamma^+$ lies at the intersection of the ellipse  $\lambda(t)$ and the hyperbola $\nu( t)$, evaluated at $t = mT / 4kl = nT / 4kl$, 
where $ 0 < m= 2\alpha k + \beta l < l$ and $  0 < n = 2 \alpha' k + \beta' l < 2k$ for some  integers $\alpha, \beta,\alpha'$, and $\beta'$.  Since gcd$(k,l/2)=1$, the possible values of $m$ and $n$ range over the even integers:  $m=2j,\; j=1, \ldots,l/2-1$, and $n=2i,\; i=1,\ldots, k-1$.

It follows that  
there exist $0 =\lambda_0 < \lambda_1 < \lambda_2 < \cdots < \lambda_{l/2-1} < \lambda_{\max}$   and
   $-\pi = \nu_0 < \nu_{ 1} <   \cdots < \nu_{k- 1} < \nu_k = 0$ such that all self-intersection points of the collision-collision orbit $K$ occur at the  intersections of the ellipses $\lambda = \lambda_j $ and the hyperbolae $\nu=\nu_i $,  see Figure \ref{fig:examplesforRemark}. These coordinates satisfy 
\[
   \lambda_j = \lambda\left(  \frac{2j T}{4kl}   \right) = \lambda\left( \frac{2l \pm 2j}{4kl}T     \right)    = \cdots =   \lambda\left( \frac{(k-1)l \pm 2j}{4kl}T    \right)  
   \]
   and
   \[
     \nu_i =   \nu \left(   \frac{2i T}{4kl}   \right)     = \nu \left( \frac{4k  \pm 2i  }{4kl}T    \right)    =\cdots =  
     \begin{cases} \displaystyle \nu \left( \frac{kl  - 2i }{4kl}T \right) & \text{  if $ l \in 4\N$,} \\ \\  \displaystyle
     \nu \left( \frac{(l-2)k   \pm 2i }{4kl}T   \right) & \text{  if $  l \in 4\N_0 +2$.} \end{cases}
   \]

   \medskip

\noindent
{\bf Fact(3e).} We now take a closer look at the intersection $K \cap ( \lambda  = \lambda_{\max}). $ 
In view of Fact(2e), it suffices to determine the values  $\nu_j$ at which $\lambda  =   \lambda_{\max}$. Recall that $\lvert \lambda\rvert  = \lambda_{\max}$  occurs  at times of the form $t=(2a-1)lT/4kl, a \geq 1$. 
If $l \in 4\N$, then   $(2a-1)l \in 4\N$, but not in   $  2\N \setminus 4\N$. Hence,   the hyperbolae $\nu=\nu_0, \nu_2, \ldots, \nu_{k-1}$ intersect the ellipse $\lambda = \lambda_{\max}$.
If instead $l \in 4\N_0+2$, then $(2a-1)l \in 2\N \setminus 4 \N$, and thus the intersections occur at $\nu=\nu_1, \nu_3, \ldots, \nu_{k }$. 
 In summary, the intersections between the collision-collision orbit and the ellipse $\lambda = \lambda_{\max}$ are given by
 \[
K \cap (\lambda  = \lambda_{\max}) = \begin{cases}   (\lambda  = \lambda_{\max}) \cap (\nu=\nu_{\rm{even}})       & l \in 4\N, \\     (\lambda  = \lambda_{\max}) \cap (\nu=\nu_{\rm{odd}})  &   l \in 4 \N_0+2. \end{cases}
 \]

   \medskip

\noindent
{\bf Fact(4e).}  As shown in Fact(1e), there are $(k-1)/2$ double points on the ellipse $\lambda = \lambda_0$.  On each of the other ellipses $\lambda = \lambda_j, j=1, \ldots, l/2-1$, there are exactly $k$ double points. We now explain this in more detail. 

Recall from Fact(2e) that the values $\lambda = \lambda_{\rm even}$ and $\lambda = \lambda_{\rm odd}$ are attained at times of the form
\[
t_{\rm even} = \frac{T}{4kl} ( 2  \spadesuit  l \pm 2 {\rm even}), \quad t_{\rm odd} = \frac{T}{4kl} ( 2    \spadesuit   l \pm 2 {\rm odd}), 
\]
respectively. In a similar way, we have   
\[
t_{\rm even}' = \frac{T}{4kl} ( 4   \spadesuit  k \pm 2 {\rm even}), \quad t_{\rm odd}' = \frac{T}{4kl} ( 4    \spadesuit  k \pm 2 {\rm odd}), 
\]
for the values $\nu = \nu_{\rm even}$ and $\nu = \nu_{\rm odd}$, 
respectively. Since $l$ is even,  both  $t_{\rm even}$ and $t_{\rm even}'$ are divisible by $4$, while $t_{\rm odd}$ and $t_{\rm odd}'$ are only divisible  by $2$. This observation implies the following classification of  double points along $K$: 
   \begin{itemize}
   \item   the ellipse $\lambda = \lambda_0$ intersects   the hyperbolae $\nu=\nu_2, \nu_4, \ldots, \nu_{k-1}$, yiedling $(k-1)/2$ double points;
   \item   each ellipse  $\lambda = \lambda_{\rm even}$ intersects    the hyperbolae $\nu=\nu_0, \nu_2, \ldots, \nu_{k-1}$, yiedling $k$ double points; and
      \item  each ellipse    $\lambda = \lambda_{\rm odd}$   intersect the hyperbolae $\nu=\nu_1, \nu_3,\ldots, \nu_k$, also yielding $k$ double points. 
   \end{itemize}

 \medskip

We now turn to the case where $l$ is odd. Let $K$ be one of the two collision-collision orbits of type II lying  on a $T_{k,l}$-torus, and denote its parametrization by $\gamma$, as in  the previous proposition.

   \medskip

\noindent
\textbf{Fact(1o).} The collision-collision orbit $K$ intersects the $q_1$-axis in exactly  $k+l$ points. These intersection points are ordered as
\begin{center}
    $\frac{l-1}{2}$ single points, $E$, $k-1$ single points, $M$, $\frac{l-1}{2}$ single points
\end{center}

   \medskip

\noindent
\textbf{Fact(2o).} There exist values $0 < \lambda_1 < \cdots < \lambda_{(l-1)/2} < \lambda_{\max}$ and $-\pi < \nu_1 < \cdots < \nu_k < 0$, such  that
the self-intersection points of the collision-collision orbit $K$ lie at the intersections of the ellipses $\lambda = \lambda_j$ and the hyperbolae $\nu = \nu_i$, see Figure \ref{fig:examplesforRemark}.

\medskip

Other structural facts of this type can be formulated in a similar manner, but are omitted here as they are not needed for the arguments that follow.

 \begin{figure}[h]
  \centering
  \includegraphics[width=1.0\linewidth]{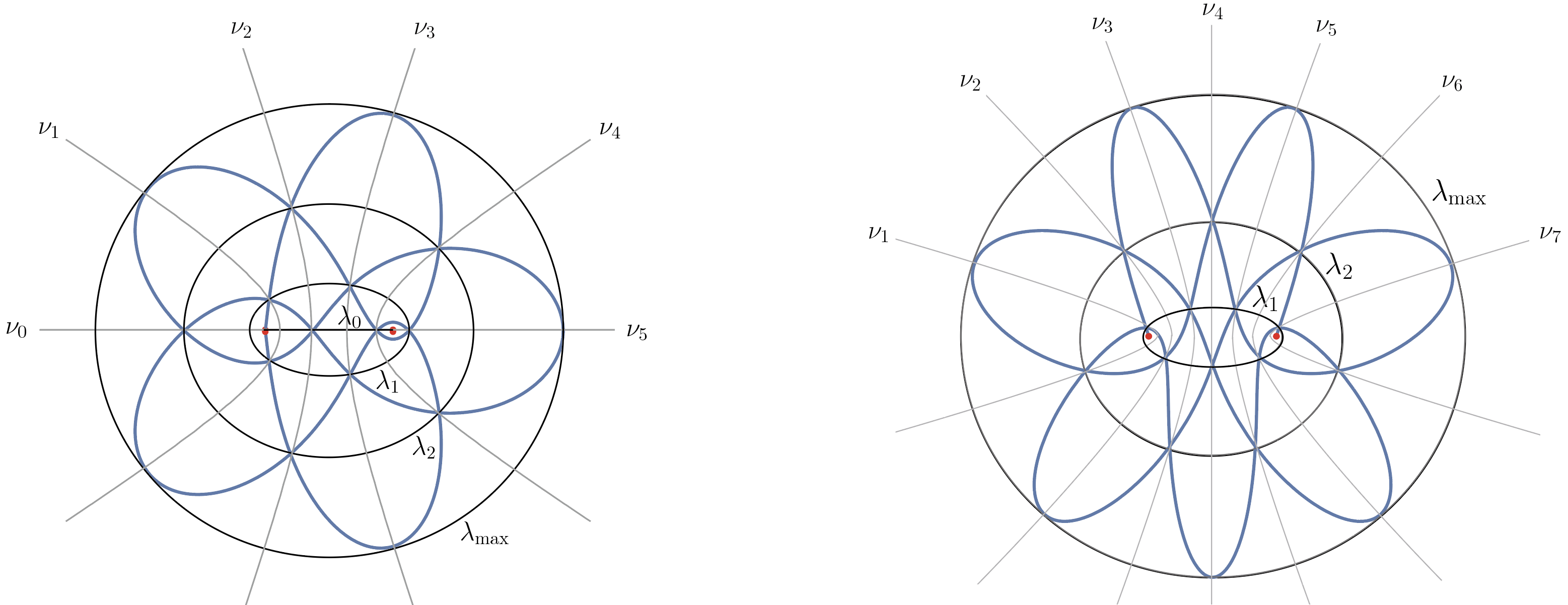}
 \caption{  Illustrations of Fact(2e) and Fact(2o) in   Remark  \ref{rmk:useful}, shown for the cases   (Left) $k=5, \; l=6 $  and (Right) $k=7,\; l=5$ 
 }
 \label{fig:examplesforRemark}
\end{figure}  

     \end{remark}

The previous proposition, combined with     Proposition \ref{prop:formula of collision orbits} and Lemma \ref{lem:Lregion}, implies the following assertion.

   \begin{proposition}\label{Prop:J0}
   Every $T_{k,l}$-type orbit in the $L$-region satisfies
 \[
 \mathcal{J}_0 = kl-k+1.
 \]
   \end{proposition}


   \subsection{The invariants $\mathcal{J}_E$ and  $\mathcal{J}_M$}
Let $K$ be a collision-collision orbit lying on a $T_{k,l}$-torus.  To determine its invariants $\mathcal{J}_E$ and $\mathcal{J}_M$, we again distinguish cases according to the parity of $l$.

   \medskip

   \noindent
   \textit{Case 1.}  $l$ is even.

  Recall that $K$ collides twice with the same primary and is $q_1$-symmetric.  As before, without loss of generality, we may assume that the collisions occur at  $E$.    
We begin by computing $\mathcal{J}_E(K)$.
In order to apply the Levi-Civita map $L_E(z)= z^2$ around $E$, we normalize the positions of the primaries by setting $E = (0,0)$ and $M=(1,0)$.

 \begin{lemma}\label{lem:simgleLE} The preimage $L_E^{-1}(K)$ is connected and is symmetric with respect to both the horizontal and vertical axes. 
 \end{lemma}
 \begin{proof} 
Let $\widetilde K_1$ and $\widetilde K_2$ denote two connected components of $L_E^{-1}(K)$.
 Since $L_E$ is a $2$-to-$1$ covering branched at $E$, it follows that the origin is a double point of $L_E^{-1}(K)$.
 This in particular implies that $\widetilde{K}_1$ and $\widetilde{K}_2$ coincide locally near the origin.  
Moreover, since they are globally related by a rotation of angle $\pi$ about the origin,   we conclude that $\widetilde{K}_1 = \widetilde{K}_2$, implying that $L_E^{-1}(K)$ is connected. Finally, since $K$ is $q_1$-symmetric, and the map $L_E$ is a $2$-to-$1$ branched covering, the preimage $L_E^{-1}(K)$ inherits symmetry with respect to both coordinate axes.   This finishes the proof of the lemma.
 \end{proof}

We denote $\widetilde{K} = L_E^{-1}(K)$.  
Since the collision at $E$ is regularized, the curve $\widetilde{K}$ is regular, i.e.~it passes through the origin. In particular, it is no longer a distinguished orbit.  See Figure \ref{fig:ex9}.
 \begin{figure}[h]
     \centering
\begin{tikzpicture} 
 
\node[below] at (1.9,-0.05) {\small $E$};
     \node[below] at (3.2,-0.1) {\small $M$};

    \draw[thick] (2,0) [out=60, in=180] to (3.1,0.7); 
    \draw[thick] (3.1,0.7) [out=0, in=90] to (4,0); 
\begin{scope}[yscale=-1]
    \draw[thick] (2,0) [out=60, in=180] to (3.1,0.7); 
    \draw[thick] (3.1,0.7) [out=0, in=90] to (4,0); 
    \end{scope}

    \draw[thick] (-4,-0.8) [out=0, in=180] to (-2,0.8); 
    \draw[thick] (-2,0.8) [out=0, in=90] to (-1.1,0); 
    \draw[thick] (-4,-0.8) [out=180, in=270] to (-4.9,0); 
\begin{scope}[yscale=-1]
  \draw[thick] (-4,-0.8) [out=0, in=180] to (-2,0.8); 
    \draw[thick] (-2,0.8) [out=0, in=90] to (-1.1,0); 
    \draw[thick] (-4,-0.8) [out=180, in=270] to (-4.9,0); 

\end{scope}

    \draw[thick, ->] (-0.5,0.1) [out=30, in=150] to (1.4,0.1); 

       \filldraw[draw=black, fill=black] (2,0) circle (0.07cm);
              \filldraw[draw=black, fill=black] (3.2,0) circle (0.07cm);

             \filldraw[draw=black, fill=black] (-3,0) circle (0.07cm);
             \filldraw[draw=black, fill=black] (-4.1,0) circle (0.07cm);
             \filldraw[draw=black, fill=black] (-1.9,0) circle (0.07cm);

\node[below] at (-4.1, -0.1) {\small $M_1$};
\node[below] at (-1.9, -0.1) {\small $M_2$};
 
 \node  at (0.5, 0.8)  { $L_E$};

 \node  at (3,-1)  {\small $K$};

  \node  at (-3,-1.0)  { \small $L_E^{-1}(K)$};

      \end{tikzpicture}
    \caption{An example of a collision-collision orbit $K$ such that the preimage $L_E^{-1}(K)$ is a regular orbit }
 \label{fig:ex9}
 \end{figure}
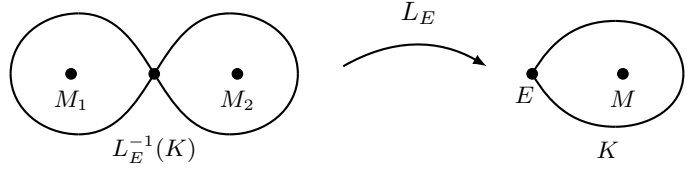  
Therefore, the invariant $\mathcal{J}_E$ cannot be computed by counting double points.
Instead, we make use of the explicit formula for $\mathcal{J}_0$ obtained in the previous section as follows.
    
As introduced earlier, we denote by $(\lambda, \nu)$  the elliptic coordinates with foci at $M_1 = (-1, 0)$ and $M_2 = (1,0)$   (recall that $L_E^{-1}(M)=\{M_1, M_2\}$). We also denote   by $(u,v)$   the elliptic coordinates with foci at $E=0$ and $M$. In these coordinates, the Levi-Civita mapping $L_E$ is given by $L_E(\lambda, \nu) = (u,v)$.  
 See Figure \ref{fig:ex10}.

 \begin{figure}[h]
     \centering
\begin{tikzpicture}

\draw[thick, blue] (0.5,0) to (4.5,0);
\draw[thick, red] (0.5,0) to (3.3,0);
\draw[thick, green] (0.5,0) to (1.9,0);

\draw[thick, green] (-2.5, 1.5) to (-2.5, -1.5);

\draw[thick, blue] (-4.5,0) to (-0.5,0);
\draw[thick, red] (-3.3,0) to (-1.7,0);
 \filldraw[draw=black, fill=black] (-3.3,0) circle (0.07cm);
 \filldraw[draw=black, fill=black] (-1.7,0) circle (0.07cm);
  \filldraw[draw={rgb:black,5;white,5}, fill={rgb:black,5;white,5}] (-2.5,0) circle (0.08cm);
 \node [below]  at (-3.3 ,-0.1)  {\small $M_1$};
 \node [below]  at (-1.7 ,-0.1)  {\small $M_2$};
 \node [below]  at (-2.35 ,-0.1)  { $\color{rgb:black,5;white,5}0$};

 \filldraw[draw=black, fill=black] (3.3,0) circle (0.07cm);
  \filldraw[draw={rgb:black,5;white,5}, fill={rgb:black,5;white,5}] (1.9,0) circle (0.08cm);
 \node [below]  at (3.3 ,-0.1)  {\small $M$};
 \node [below]  at (1.9 ,-0.1)  { \small $\color{rgb:black,5;white,5}0$};

 \node   at (-2.5  , -2)  { $(\lambda, \nu)$};

 \node   at (2.5 ,-2)  { $(u,v)$};

\draw[thick] (-2.5,0) ellipse (1.3cm and 1cm);
\draw[thick] (2.5,0) ellipse (1.3cm and 1cm);
 \filldraw[draw=purple, fill=purple] (-2.5,1) circle (0.09cm);
 \filldraw[draw=purple, fill=purple] (-2.5,-1) circle (0.09cm);
 \filldraw[draw=purple, fill=purple] (1.2,0) circle (0.09cm);

 \draw[thick, ->, gray] (-0.8,0.7) [out=30, in=150] to (0.8 ,0.7); 

 \node[gray]  at (0 , 1.3)  { $L_E$};

     \begin{scope}[yscale=-1]
      \draw[thick, -> ] (-1,0.7) [out=30, in=150] to (1 ,0.7); 
 \node   at (0 , 1.3)  { double cover};
 \end{scope}

      \end{tikzpicture}
\caption{The map $L_E$ in the elliptic coordinates}
 \label{fig:ex10}
 \end{figure}
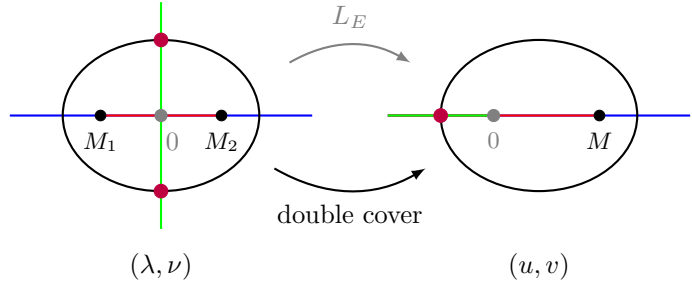

Since $K$ lies on a $T_{k,l}$-torus, the variables $u$ and $v$ complete $k$ and $l$ cycles, respectively. Under the Levi-Civita transformation $L_E$, this implies that $\lambda$ and $\nu$ complete $2k$ and $l$ cycles, respectively. Hence, $\widetilde{K} $  is a $T_{k,l/2}$-type orbit (note that gcd$(k,l/2)=1$). Thus, by  Proposition \ref{Prop:J0}, we obtain  
\[
\mathcal{J}_E(K) =   \frac{kl}{2}-k +1  \quad \text{if $l$ is even.}
\]
 
 \medskip

To determine the invariant $\mathcal{J}_M$, we normalize the positions of the two primaries by setting $E=(-1,0)$ and $M=(0,0)$ so that the Levi-Civita map around $M$ becomes $L_M(z) = z^2$. 
Let $\widetilde{K}_1$ and $  \widetilde{K}_2$ denote the two connected components of the preimage $L_M^{-1}(K)$. 
As in the previous case, both components are $T_{k,l/2}$-type orbits, and hence   we have 
\begin{equation}\label{eq:JMleven}
\mathcal{J}_M(K) = \frac{kl}{2}-k +1  \quad \text{if $l$ is even.}
\end{equation}

To make the presentation more transparent, we directly compute $\mathcal{J}_M(K)$. Note that since $K$ collides only with $E$, both components of $L_M^{-1}(K)$ are collision-collision orbits. See Figure \ref{fig:ex9LM}. Without loss of generality, we may focus on $\widetilde{K}_1$, as the value of $\mathcal{J}_M$ is independent of the choice of component. We then proceed to count self-intersection points of $\widetilde{K}_1$.

 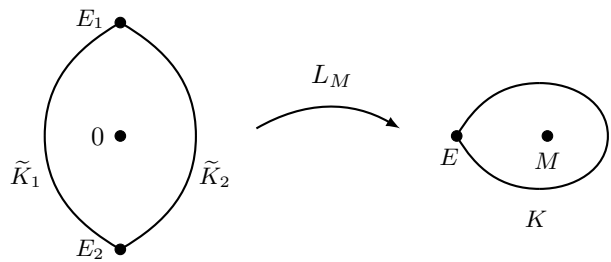
\begin{figure}[h]
     \centering
\begin{tikzpicture} 

 \begin{scope}[xshift=0.15cm]
\node[below] at (1.9,-0.05) {\small $E$};
     \node[below] at (3.2,-0.1) {\small $M$};

    \draw[thick] (2,0) [out=60, in=180] to (3.1,0.7); 
    \draw[thick] (3.1,0.7) [out=0, in=90] to (4,0); 
\begin{scope}[yscale=-1]
    \draw[thick] (2,0) [out=60, in=180] to (3.1,0.7); 
    \draw[thick] (3.1,0.7) [out=0, in=90] to (4,0); 
    \end{scope}
       \filldraw[draw=black, fill=black] (2,0) circle (0.07cm);
              \filldraw[draw=black, fill=black] (3.2,0) circle (0.07cm);
\end{scope}

\begin{scope}[xshift=0.7cm]
    \draw[thick] (-4,0 ) [out=90, in=210] to (-3,1.5); 
\begin{scope}[yscale=-1]
    \draw[thick] (-4,0 ) [out=90, in=210] to (-3,1.5); 
\end{scope}
\begin{scope}[xscale=-1, xshift=6cm]
       \draw[thick] (-4,0 ) [out=90, in=210] to (-3,1.5); 
\begin{scope}[yscale=-1]
    \draw[thick] (-4,0 ) [out=90, in=210] to (-3,1.5); 
\end{scope}
\end{scope}
                \filldraw[draw=black, fill=black] (-3,-1.5) circle (0.07cm);
                          \filldraw[draw=black, fill=black] (-3,1.5) circle (0.07cm);
             \filldraw[draw=black, fill=black] (-3,0) circle (0.07cm);
\end{scope}

    \draw[thick, ->] (-0.5,0.1) [out=30, in=150] to (1.4,0.1);

 
 \node  at (0.5, 0.8)  { $L_M$};
\node  at (-2.7,1.55)  {\small $E_1$};
\node  at (-2.7,-1.5)  {\small $E_2$};
 \node  at (3.2,-1.1)  {\small $K$};
\node at (-2.6,  0 ) {$0$};
\node at (-3.55, -0.5) {\small $\widetilde{K}_1$};
\node at (-1.05, -0.5) {\small $\widetilde{K}_2$};

      \end{tikzpicture}
    \caption{An example of a collision-collision orbit $K$ such that the preimage  $L_M^{-1}(M)$ consists of two collision-collision orbits of type II }
 \label{fig:ex9LM}
 \end{figure}

If $l \in 4\N_0+2$, then $l/2$ is odd and, hence $\widetilde{K}_1$ undergoes collisions at both $E_1 = (0,1)$ and $E_2 = (0,-1)$ and is symmetric with respect to the $x$-axis.  
In view of Proposition \ref{prop:quadruple}, it follows that $\widetilde{K}_1$ has $k(l/2 -1)/2$ double points. Therefore, we obtain 
\[
\mathcal{J}_M(K) = k\left( \frac{l}{2} -1 \right) + \frac{1}{2} + \frac{1}{2} = \frac{kl}{2} -k+1 \quad \text{if $l \in 4\N_0+2$.}
 \]
  Similarly, if $l \in 4\N$, then $l/2$ is even, and   $\widetilde{K}_1$ undergoes two collisions with either $E_1$ or $E_2$.  
 Again by  Proposition \ref{prop:quadruple}, the number of double points is given by   $(k l/2 - k -1)/2$. Using this, we compute  
 \[
 \mathcal{J}_M(K) = \frac{kl}{2} - k -1 + \frac{4}{2} + 0 = \frac{kl}{2} -k+1 \quad \text{if $l \in 4\N$.}
 \]
In both cases, we recover the same formula as in \eqref{eq:JMleven}.

   \medskip

   \noindent
   \textit{Case 2.} $l$ is odd.

As in the case $l$ is even, we normalize the positions of the primaries setting $E=(0,0)$ and $ M=(1,0)$,  so that $L_E(z)=z^2$.
Since  $w_E(K)$ is odd, the preimage $L_E^{-1}(K)$ consists of a single orbit colliding with both $M_1 = (-1,0)$ and $M_2 = (1,0)$, and the origin is no longer a singularity. 
Because $L_E$ is a 2-to-1 map,  Proposition \ref{prop:quadruple} implies that $L_E^{-1}(K)$ has $k(l-1)$ double points. Similarly, $L_M^{-1}(K)$ is also a single orbit that has collisions at both $E_1$ and $ E_2$ and has the same number of double points.  
In conclusion, the invariants $\mathcal{J}_E$ and $\mathcal{J}_M$  are given by
\[
\mathcal{J}_E (K) = \mathcal{J}_M (K) = 2k(l-1) +1 \quad \text{if $l$ is odd}.
\]
These formulas are consistent with the relation established in \cite[Corollary 6.23]{CFZ23}.
 
 \bigskip

 We have proven

\begin{proposition}\label{Prop:JEJM}
Every $T_{k,l}$-type orbit in the $L$-region satisfies
 \[
 \mathcal{J}_E    = \mathcal{J}_M   = \begin{cases} \frac{ kl}{2} - k+1 & \text{if $l$ is even,}  \\ \\ \displaystyle  2kl-2k+1 & \text{if $l$ is odd.}\end{cases}
 \]
 \end{proposition}


 \subsection{The invariant  $(\mathcal{J}_{E,M}, n)$}\label{subsec:EM}

 We normalize the positions of the two primaries to $E=(-1,0)$ and $M=(1,0)$ and denote by $(\lambda,\nu)$ the associated elliptic coordinates as before.

The coordinate line $\lambda = \lambda_0 >0 $  corresponds to an ellipse given by 
\begin{equation}\label{eq:ell} 
\frac{x^2}{a^2} + \frac{y^2}{b^2} = 1, \quad \text{where }\; a = \cosh \lambda_0  \; \text{ and } \; b = \sinh \lambda_0 ,
\end{equation} 
while the cooordinate line $\nu =\nu_0 \in (-\pi, 0)$ corresponds to a hyperbola given by 
\begin{equation}\label{eq:hyper} 
\frac{x^2}{a^2} - \frac{y^2}{b^2} = 1, \quad \text{where } \; a = \cos \nu_0 \;  \text{ and } \; b = \sin \nu_0  ,  
\end{equation}
provided that $\nu_0 \neq -\pi/2$. In the case  $\nu_0=-\pi/2$, the coordinate line coincides with the $y$-axis.

Let $B \colon \C \setminus \{ 0 \} \to \C$ denote the Birkhoff regularization map as defined in \eqref{eq:Birkhoff}. We extend it to a branched double covering $B \colon  \C \cup \{ \infty \}  \to \C \cup \{ \infty\} $ by defining it to send both $0$ and $\infty$ to $\infty$. In the following lemma, we collect some basic properties of this extended map. The proofs are straightforward and omitted.

\begin{lemma}\label{lem:propofB} The Birkhoff regularization map $B \colon \C \cup \{\infty\} \to \C \cup \{ \infty \}$ satisfies the following properties:  
    
\begin{enumerate}
\item  The arcs $(0,1) \times \{ 0 \}$ and $(1,\infty) \times \{0\}$ are bijectively mapped onto $(1, \infty) \times \{0\}$, with the former reversing and the latter preserving it.   In a similar way, the arcs $(-\infty, -1) \times \{ 0 \}$ and $(-1, 0) \times \{ 0 \}$ are bijectively mapped onto $(-\infty, -1) \times \{ 0 \}$,   preserving and reversing the orientations, respectively.

\item  The preimage of an ellipse of the form \eqref{eq:ell} is the union of the two circles  of radii $a \pm b$, centered at the origin. In particular, the preimage of the coordinate line $\lambda =0$ is the unit circle.

\item Let $\nu=\nu_0$ be a coordinate line corresponding to a hyperbola of the form \eqref{eq:hyper}. If $\nu_0 \in (-\pi/2, 0)$, then its preimage is given by the pair of rays $\{ y = \pm (b/a)x \mid x >0 \}$. If $\nu_0 \in (-\pi, -\pi/2)$, then the preimage is $\{ y = \pm (b/a)x \mid x < 0 \}$. 
The preimage of $\nu=-\pi/2$, i.e.~the $q_2$-axis, is the $y$-axis, see Figure \ref{fig:ex11} for illustrations.

 \begin{figure}[h]
     \centering
\begin{tikzpicture} [scale=0.8]

\draw[thick, ->] (-5,0) to (-1,0);
\draw[thick,->] (-3,-1.8) to (-3,1.8);
\draw[thick,->] (1,0) to (5,0);
\draw[thick,->] (3,-1.8) to (3,1.8);
    \draw[thick, ->] (-0.8,1) [out=30, in=150] to (0.8,1); 
\node  at (0, 1.5)  { $B$};

\draw[thick, cyan, postaction={decorate},  decoration={markings, mark=at position 0.5 with {\arrow{<}}}] 
          (-3,0) -- (-2,0);

\draw[thick, cyan, postaction={decorate},  decoration={markings, mark=at position 0.8 with {\arrow{>}}}] 
          (-2,0) -- (-1.2,0);

\draw[thick, cyan, postaction={decorate},  decoration={markings, mark=at position 0.8 with {\arrow{>}}}] 
          (4,0) -- (4.8,0);

\draw[thick, orange, postaction={decorate},  decoration={markings, mark=at position 0.5 with {\arrow{<}}}] 
          (-3,0) -- (-4,0);

\draw[thick, orange, postaction={decorate},  decoration={markings, mark=at position 0.8 with {\arrow{>}}}] 
          (-4,0) -- (-5,0);

\draw[thick, orange, postaction={decorate},  decoration={markings, mark=at position 0.8 with {\arrow{>}}}] 
          (2,0) -- (1,0);



      

\draw[thick, blue, postaction={decorate}, decoration={markings, mark=at position 0.3 with {\arrow{>}}}] 
        (-4,0) arc[start angle=180, end angle=0, x radius=1cm, y radius=1cm];
\draw[thick, blue, postaction={decorate}, decoration={markings, mark=at position 0.3 with {\arrow{>}}}] 
        (-4,0) arc[start angle=180, end angle=360, x radius=1cm, y radius=1cm];

\draw[thick, blue, postaction={decorate},  decoration={markings, mark=at position 0.3 with {\arrow{>}}}] 
          (2,0) -- (4,0);

           \filldraw[draw=black, fill=black] (-4,0) circle (0.08cm);
           \filldraw[draw=black, fill=black] (-2,0) circle (0.08cm);      
           \filldraw[draw=black, fill=black] ( 4,0) circle (0.08cm);
           \filldraw[draw=black, fill=black] (2,0) circle (0.08cm);
     \filldraw[draw=black, fill={rgb:black,5;white,5}] ( -3,1) circle (0.08cm);
           \filldraw[draw=black, fill={rgb:black,5;white,5}] (-3,-1) circle (0.08cm);
           \filldraw[draw=black, fill={rgb:black,5;white,5}] (3,0) circle (0.08cm);

\node [below] at (-1.8, 0)  { $1$};
\node[below]  at (-4.3, 0)  { $-1$};
\node [right] at (-3, 1.2)  { $1$};
\node[right]  at (-3, -1.2)  { $-1$};
\node [below] at (2, -0.1)  { $-1$};
\node[below]  at (4, -0.1)  { $1$};


\begin{scope}[yshift=-4.5cm]

\draw[thick, ->] (-5,0) to (-1,0);
\draw[thick,->] (-3,-1.8) to (-3,1.8);
\draw[thick,->] (1,0) to (5,0);
\draw[thick,->] (3,-1.8) to (3,1.8);
    \draw[thick, ->] (-0.8,1) [out=30, in=150] to (0.8,1); 
\node  at (0, 1.5)  { $B$};



      

\draw[thick,  magenta, postaction={decorate}, decoration={markings, mark=at position 0.4 with {\arrow{>}}}] 
        (-1.6,0) arc[start angle=0, end angle=360, x radius=1.4cm, y radius=1.4cm];
 
\draw[thick,  magenta, postaction={decorate}, decoration={markings, mark=at position 0.4 with {\arrow{<}}}] 
        (-2.4,0) arc[start angle=0, end angle=360, x radius=0.6cm, y radius=0.6cm];

 \draw[thick, blue, postaction={decorate}, decoration={markings, mark=at position 0.8 with {\arrow{<}}}]        (-2,0) arc[start angle=0, end angle=180, x radius=1cm, y radius=1cm];

 \draw[thick, magenta, postaction={decorate}, decoration={markings, mark=at position 0.4 with {\arrow{>}}}]        (4.5,0) arc[start angle=0, end angle=360, x radius=1.5cm, y radius=1cm];

 \draw[thick, blue, postaction={decorate}, decoration={markings, mark=at position 0.25 with {\arrow{>}}}]        (-4,0) arc[start angle=180, end angle=360, x radius=1cm, y radius=1cm];

\draw[thick, blue, postaction={decorate},  decoration={markings, mark=at position 0.35 with {\arrow{>}}}] 
          (2,0) -- (4,0);

            \filldraw[draw=black, fill=black] (-4,0) circle (0.08cm);
            \filldraw[draw=black, fill=black] (-2,0) circle (0.08cm);      
            \filldraw[draw=black, fill=black] ( 4,0) circle (0.08cm);
            \filldraw[draw=black, fill=black] (2,0) circle (0.08cm);

 \node [below] at (-1.8, 0)  { $1$};
 \node[below]  at (-4.3, 0)  { $-1$};
 \node [below] at (2, -0.1)  { $-1$};
 \node[below]  at (4, -0.1)  { $1$};

\end{scope}


\begin{scope}[yshift=-9.2cm]

\draw[thick, ->] (-5,0) to (-1,0);
\draw[thick,->] (-3,-2) to (-3,2);
\draw[thick,->] (1,0) to (5,0);
\draw[thick,->] (3,-2) to (3,2);
    \draw[thick, ->] (-0.8,1) [out=30, in=150] to (0.8,1); 
\node  at (0, 1.5)  { $B$};

\draw[thick, olive, postaction={decorate},  decoration={markings, mark=at position 0.9 with {\arrow{>}}}] 
          (-3,0) -- (-3,1.8);

\draw[thick, olive, postaction={decorate},  decoration={markings, mark=at position 0.5 with {\arrow{<}}}] 
          (-3,0) -- (-3,1);

\draw[thick, olive, postaction={decorate},  decoration={markings, mark=at position 0.6 with {\arrow{>}}}] 
          ( 3,0) -- ( 3,1.8);

\draw[thick, blue, postaction={decorate},  decoration={markings, mark=at position 0.9 with {\arrow{>}}}] 
          (-3,0) -- (-3,-1.8);

\draw[thick, blue, postaction={decorate},  decoration={markings, mark=at position 0.5 with {\arrow{<}}}] 
          (-3,0) -- (-3,-1);

\draw[thick, blue   , postaction={decorate},  decoration={markings, mark=at position 0.6 with {\arrow{>}}}] 
          ( 3,0) -- ( 3,-1.8);

\draw[thick, cyan  , postaction={decorate},  decoration={markings, mark=at position 0.6 with {\arrow{<}}}] 
          ( -3,0) -- (-1.2, 1.8  );

\draw[thick, cyan  , postaction={decorate},  decoration={markings, mark=at position 0.6 with {\arrow{<}}}] 
          ( -3,0) -- (-1.2, -1.8  );

\draw[thick,  magenta  , postaction={decorate},  decoration={markings, mark=at position 0.6 with {\arrow{<}}}] 
          ( -3,0) -- (-4.8, 1.8  );

\draw[thick, magenta  , postaction={decorate},  decoration={markings, mark=at position 0.6 with {\arrow{<}}}] 
          ( -3,0) -- (-4.8, -1.8  );

\draw[thick, cyan  , postaction={decorate},  decoration={markings, mark=at position 0.6 with {\arrow{>}}}] 
          ( -3,0) -- (-2.2, 0.8  );

\draw[thick, cyan  , postaction={decorate},  decoration={markings, mark=at position 0.6 with {\arrow{>}}}] 
          ( -3,0) -- (-2.2, -0.8  );

\draw[thick, magenta , postaction={decorate},  decoration={markings, mark=at position 0.6 with {\arrow{>}}}] 
          ( -3,0) -- (-3.8,0.8  );

\draw[thick, magenta , postaction={decorate},  decoration={markings, mark=at position 0.6 with {\arrow{>}}}] 
          ( -3,0) -- (-3.8, -0.8  );



      

 \draw[thick, lightgray ]        (-2,0) arc[start angle=0, end angle=180, x radius=1cm, y radius=1cm];

 \draw[thick, lightgray]        (-4,0) arc[start angle=180, end angle=360, x radius=1cm, y radius=1cm];

\draw[thick, magenta, postaction={decorate},  decoration={markings, mark=at position 0.4 with {\arrow{<}}}] 
          (2.5,0)  [out=90, in=330] to (1.2,1.8); 
\begin{scope}[yscale=-1]

\draw[thick, magenta, postaction={decorate},  decoration={markings, mark=at position 0.4 with {\arrow{<}}}] 
          (2.5,0)  [out=90, in=330] to (1.2,1.8); 
          \end{scope}

\begin{scope}[xscale=-1, xshift=-6cm]
    
\draw[thick, cyan, postaction={decorate},  decoration={markings, mark=at position 0.4 with {\arrow{<}}}] 
          (2.5,0)  [out=90, in=330] to (1.2,1.8); 
\begin{scope}[yscale=-1]

\draw[thick, cyan, postaction={decorate},  decoration={markings, mark=at position 0.4 with {\arrow{<}}}] 
          (2.5,0)  [out=90, in=330] to (1.2,1.8); 
          \end{scope}
\end{scope}

           \filldraw[draw=black, fill=cyan] (-2.28,0.72) circle (0.08cm);
           \filldraw[draw=black, fill=cyan] (-2.28,-0.72) circle (0.08cm); 
           \filldraw[draw=black, fill=magenta] (-3.72,0.72) circle (0.08cm);
           \filldraw[draw=black, fill=magenta] (-3.72,-0.72) circle (0.08cm);

           \filldraw[draw=black, fill=black] (-4,0) circle (0.08cm);
           \filldraw[draw=black, fill=black] (-2,0) circle (0.08cm);      
           \filldraw[draw=black, fill=black] ( 4,0) circle (0.08cm);
           \filldraw[draw=black, fill=black] (2,0) circle (0.08cm);
     \filldraw[draw=black, fill={rgb:black,5;white,5}] ( -3,1) circle (0.08cm);
           \filldraw[draw=black, fill={rgb:black,5;white,5}] (-3,-1) circle (0.08cm);
           \filldraw[draw=black, fill={rgb:black,5;white,5}] (3,0) circle (0.08cm);
           \filldraw[draw=black, fill=magenta] (2.5,0) circle (0.08cm);
           \filldraw[draw=black, fill=cyan] (3.5,0) circle (0.08cm);

 \node [below] at (-1.8, 0)  { $1$};
 \node[below]  at (-4.3, 0)  { $-1$};
 \node [right] at (-3, 1.2)  { $1$};
 \node[right]  at (-3, -1.2)  { $-1$};
 \node [below] at (2, -0.1)  { $-1$};
 \node[below]  at (4, -0.1)  { $1$};

\end{scope}


      \end{tikzpicture}
\caption{Visualization of properties of the Birkhoff regularization map $B$}
\label{fig:ex11}
 \end{figure}
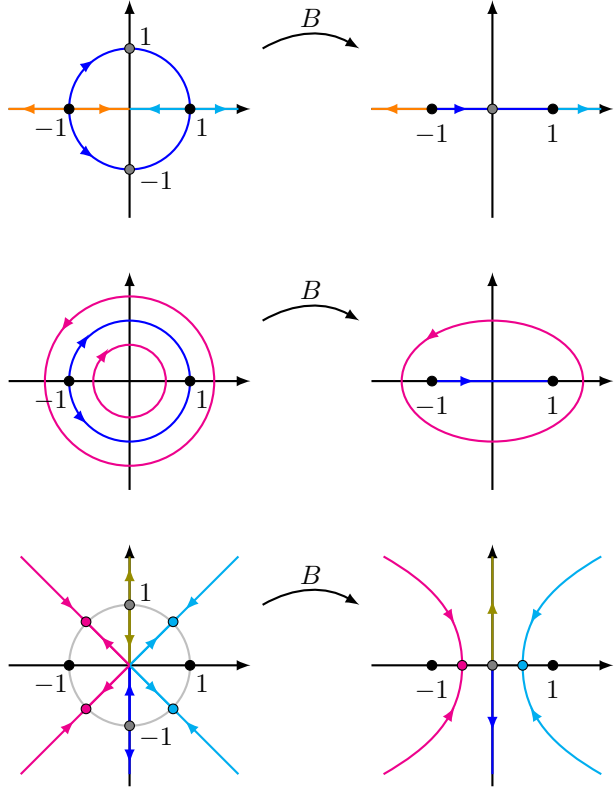

\end{enumerate}
\end{lemma}

Based on the properties described in the preceding lemma, we now investigate the preimage of a collision-collision orbit under the Birkhoff regularization map.

\begin{lemma}\label{lem:Bcoll}
Let $K$ be a collision-collision orbit lying on a $T_{k,l}$-torus. Then the following assertions hold:

  \begin{enumerate}
  \item The preimage $B^{-1}(K)$ consists of a single arc whose endpoints lie at $(\pm 1,0)$, depending on the parity of $l$. We denote this arc by  $\tilde K = B^{-1}(K)$.
  Moreover, $\tilde K$ is invariant under the map $\phi(z)=1/z. $   


\item Suppose that  $l$ is  even and that $K$ collides twice with $E$,  so that $\tilde K$ passes through the point $(-1,0)$.  Then the intersections of $\tilde K$ with the arc $(1,\infty) \times \{ 0 \}$ consist of $l/4$ double points if $l \in 4 \N$ and $(l-2)/4$ double points and a single point if $l \in 4 \N_0+2$.


  \item Suppose that $l$ is odd, so that $K$ collides with both primaries. Then both points $(\pm1, 0)$ are single points of $\tilde{K}$, and $\tilde{K}$ intersects the arc $(1,\infty) \times \{ 0 \}$ at $(l-1)/2$ single points.

  \end{enumerate}

\end{lemma}
\begin{proof}
  Since the map $B$ fixes two points $(\pm 1,0)$, it is immediate that each component of $B^{-1}(K)$ must start and end at either  $(-1,0)$ or $(1,0)$. The rest of the first assertion follows from a similar argument as in  Lemma \ref{lem:simgleLE}, replacing the $\pi$-rotation with the symmetry under the map $\phi$. This proves the first assertion.
  The second and third assertions follow from Lemma \ref{lem:propofB}(1), together with   Fact(1e) and Fact(1o) in   Remark \ref{rmk:useful}. This completes the proof of the lemma.
\end{proof}

 We are in a position to determine the quantity $n(K)$.

\begin{proposition}\label{prop:leavf} For any $T_{k,l}$-type orbit $K$, we have   $n(K)=l$.
\end{proposition}
\begin{proof} As before, we may assume that $K$ is a collision-collision orbit. Set $\tilde K=B^{-1}(K)$ as in Lemma \ref{lem:Bcoll}(1).  By definition, $n(K)$ is the absolute value of the winding number $w_0(\tilde K)$, which in our setting equals the signed count of the intersection points of $\tilde K$  with the arc $(0,\infty)\times\{0\}. $


Let $z(t)=( x(t), y(t))$ be a parametrization of $\tilde K$ and suppose that $\dot y(t)>0$ at some point $(x(t),0)$, where $x(t)>1$. We evaluate the image under the map $\phi(z) = 1/z$, see Lemma \ref{lem:Bcoll}(1). Then 
 \begin{align*}
\frac{d}{dt} {\rm Im} \phi (z(t)) = \frac{d}{dt}  \frac{ -y(t)}{x(t)^2 + y(t)^2}     = - \frac{y'(t)}{x(t)^2} <0.
 \end{align*}
Since $\phi$ reverses orientation, this implies that the corresponding intersection points on $(0,1) \times \{0\}$ and $(1, \infty) \times \{ 0 \}$ that are related via the $\phi$-symmetry contribute the same sign to the winding number.  
 The case $\dot{y}(t)<0$  is analogous. Thus, it suffices to consider only the intersections of $\tilde K$ with the arc $[1, \infty) \times \{ 0 \}$.

We first claim that every intersection point of $\tilde K$ with $(1, \infty) \times \{ 0 \}$  contributes  the same sign to the winding number.  To prove this, let $(q_1(t),q_2(t))$ be a parametrization of $K$ and let $(q_1(t_0), 0), (q_1(t_1), 0)$  be two consecutive intersection points of $K$ with $(1,\infty) \times \{0\}$. That is,   there exists no $t _2\in (t_0, t_1)$ such that $K$ intersects $(1,\infty) \times \{0\}$ at time $t=t_2$.
Recall that in the elliptic coordinates $(\lambda, \nu)$  the arc $(1,\infty) \times \{0\}$ corresponds to the coordinate line $\nu=0$. Since $p_\nu = \dot{\nu}$ is nonvanishing along lemniscate motions, see Lemma \ref{lem:phaseportrait}(1), evaluating the time derivative of $q_2 = \sinh \lambda \sin \nu$ at $\nu=0$ gives rise to 
$\dot{q}_2  =\sinh \lambda \cdot \dot{\nu}$, so that $\dot{q}_2(t_0)$ and $\dot{q}_2(t_1)$ have opposite (resp. same) signs if and only if
the part of $K$ over the interval $  (t_0, t_1)$   intersects the coordinate line $\lambda=0$, that is, the arc $(0,1) \times \{ 0 \}$,  an odd (resp. even) number of times.
Moreover, along the arc $(1,\infty) \times \{ 0 \}$ we have
\[
 \dot{q}_2 (t) = \frac{ \dot{y}(t)}{2} \left( 1 - \frac{1}{x (t)^2} \right),
\]
from which it follows that 
\begin{equation}\label{eq:q2y}
{\text{$\dot{y}(t)$ and $\dot{q}_2 (t)$ have } } \begin{cases} {\text{ same sign }} \\ {\text{ opposite sign }} \end{cases}  { \text{ if and only if }} \;\; \begin{cases} x>1 \\ 0<x<1 .\end{cases}
\end{equation}

We now write $B^{-1}( q_1(t_0), 0)=\{ a_1, a_2\}$ and $B^{-1}(q_1(t_1), 0) = \{ b_1, b_2\}$. By Lemma \ref{lem:propofB}(1) the four points $a_1, a_2, b_1, b_2$ lie on $\left( (0,\infty) \setminus \{ 1 \} \right) \times \{ 0 \}$. Recall that each pair $(a_1, a_2)$ and $(b_1, b_2)$ consists of intersection points of the same sign.   
Let $a_i, b_j$ be consecutive intersection points of $\tilde K$. They belong to the same connected component of $\left( (0,\infty) \setminus \{ 1 \} \right) \times \{ 0 \}$   if and only if $\tilde K$ intersects the unit circle an even number of times between $a_i$ and $b_j$. Due to Lemma \ref{lem:propofB}(2), this is equivalent to the condition that the part of $K$ over the interval $  \in (t_0, t_1)$  intersects the coordinate line $\lambda =0$ an even number of times, which, as shown earlier, is in turn equivalent to $\dot q_2(t_0)$ and $\dot q_2(t_1)$ having the same sign as intersection points. 
Therefore, by equation \eqref{eq:q2y} the two points $a_i, b_j$ have the same sign as the intersection points. 

Similarly, if $a_i$ and $b_j$ are in different components of $\left( (0,1) \cup ( 1, \infty) \right) \times \{0 \}$, then $\dot q_2(t_0)$ and $\dot q_2(t_1)$ have opposite signs, and again \eqref{eq:q2y} implies that $a_i, b_j$ still have the same sign.   

Consequently, all points $a_1 , a_2, b_1, b_2$   have the same sign, and hence all intersection points of $\tilde K$ with $\left( (0,1) \cup (1, \infty)\right)\times \{0 \}$ contribute  the same sign to the winding number.  This proves the claim.

Assume first that $l$ is even, so that $K$ collides twice with $E$.    Hence, $\tilde K$ does not pass the point $(1,0). $ In view of Lemma \ref{lem:Bcoll}(2),  we see that $\tilde K$ intersects the positive $x$-axis precisely $l$ times. By the preceding argument, all these intersection points contribute  the same sign. Therefore, we conclude that $w_0(\tilde K) = \pm l$ when $l$ is even.

Now consider the case where $l$ is odd, so that the point $(1,0)$ is a single point of $\tilde K$. Note that the argument above concerning the sign of intersection points applies to any $T_{k,l}$-type orbit, regardless of whether it is a collision-collision orbit or not. Since $K$ belongs to a $T_{k,l}$-torus, which is a smooth one-parameter family of $T_{k,l}$-type orbits, the intersection point $(1,0)$ must have the same sign as the other intersection points of $\tilde K$   with the arc $(1,\infty) \times \{ 0 \}$. Therefore, by Lemma \ref{lem:Bcoll}(3),  $\tilde K$ again intersects the positive $x$-axis precisely $l$ times, and all intersections contribute  the same sign. It follows that $w_0(\tilde K)= \pm l $ also in this case.

To conclude, we have $n(K) = \lvert w_0(\tilde K)\rvert = l $, which completes the proof of the proposition.
\end{proof}

It remains to determine the quantity $\mathcal{J}_{E,M}$. By definition, it suffices to determine $J^+(\widetilde{K})$. 
 Given a collision-collision orbit $K$, choose any parametrization $\gamma(t)$ of $K$. Then it gives rise to a parametrization $\tilde{\gamma}(t)$ of the preimage $\tilde K = B^{-1}(K)$.

\begin{lemma} 
The curve  $\tilde{\gamma}$  rotates in the same direction at every point. 
\end{lemma}
\begin{proof}  By the proof of Proposition \ref{prop:leavf}, it is enough to show that every ray emanating from the origin intersects $\tilde K$ precisely $l$ times.  
Recall from the proof of Proposition \ref{prop:quadruple}   that for any fixed value $\bar{\nu} \in (0, \pi)$, the corresponding hyperbola $\nu = \bar{\nu}$ intersects the orbit $K$  at exactly $l$ points. Since the map $B$ is a  $2$-to-$1$ branched covering, it follows that $\tilde K$ intersects the preimage of the hyperbola $\nu=\bar{\nu}$  exactly $2l$ times. According to Lemma \ref{lem:propofB}(3), the preimage of the hyperbola consists of two rays given by $\theta = \pm \bar{\theta}$ for some angle $\bar{\theta} \in \R$. Therefore,   each ray intersects $\tilde K$ exactly $l$ times, from which the lemma is proved. 
\end{proof}

The previous lemma enables us to provide an algorithm for drawing a loop homotopic to $\tilde K$ without encountering any disaster, following the approach of \cite[Section 3]{KKJplus}.  
For the readers' convenience, we include such an algorithm below for the case $l \in 4 \N$,   obtained by Fact(4e) in Remark \ref{rmk:useful}.  Let $K$ be a collision-collision orbit at $E$. 
\begin{enumerate}[(i)]
\item Choose real numbers $0< r_{-l/2} <  \cdots < r_{-1} < r_0=1 < r_1 < \cdots <   r_{l/2}$ and draw circles $C_{ j} $ centered at the origin with radius $r_j$. The two circles $C_{\pm l/2}$ correspond to the preimage of the ellipse $\lambda = \lambda_{\max}$ under the Birkhoff  map $B$.
Similarly, choose angular values $-\pi=\theta_0 < \theta_1 < \cdots <\theta_{k-1} < \theta_k = 0$ and draw rays emanating from the origin at angles $\theta = \pm \theta_i$.
\item On each circle $C_j$,   mark $k$ points according to the following rule: 

\begin{itemize}
    \item if  $j$ is odd, then mark the   intersections of $C_j$ with   rays at angles $\theta = \pm  \theta_1, \pm \theta_3, \ldots,   \theta_{k }$; 
 
  \item if $j$ is even, then mark  the intersections of $C_j$ with $\theta = \theta_0, \pm \theta_2, \ldots, \pm \theta_{k-1}$. 
\end{itemize}
Each marked point is denoted by $a_{i,j} =  C_i \cap \{ \theta =\theta_j \}$.  All points $a_{i,j}$ with $i   \neq \pm l/2$ correspond to   double points of $\tilde{K}$. Note that  $a_{0, 0} = (-1,0)$.

\item Starting from   $a_{-l/2, 0}$,    connect the points $a_{i,j}$ in the following order using straight line segments: 
\begin{align*}
& \qquad \qquad   a_{- { l}/{2} , 0}, \;  a_{- { l}/{2}+1, 1}, \; a_{-l/2+2,2}, \;  \ldots, \;    a_{  {l}/{2}, l  }, \; a_{{l}/{2}-1, l+1 }, \; \ldots, \; a_{- {l}/{2}, 2l} ,\\
& \qquad \qquad  a_{- {l}/{2}+1, 2l+1}, \;  \ldots, \; a_{l/2, 3l}, \; a_{l/2-1, 3l+1}, \;, \ldots, \; a_{-l/2, 4l}, \;   \ldots,     \\
& \qquad \qquad  a_{l/2, 5l}, \; \ldots,\; a_{-l/2, 6l},  \; \ldots, \; a_{l/2, (k-2)l}
  \;\ldots,  \;a_{- {l}/{2}, (k-1)l } = a_{- {l}/{2}, 0}.
\end{align*}
Here, all second sub-indices of $a_{i,j}$ are considered modulo $k$.

\item Finally, smooth out the corners of the resulting piecewise-linear arc to obtain  a smooth loop that is homotopic to   $\tilde{K}$ without encountering any disasters.
See Figure \ref{motionsd}. 

 \begin{figure}[h]
\begin{subfigure}{0.43\textwidth}
  \centering
  \includegraphics[width=0.9\linewidth]{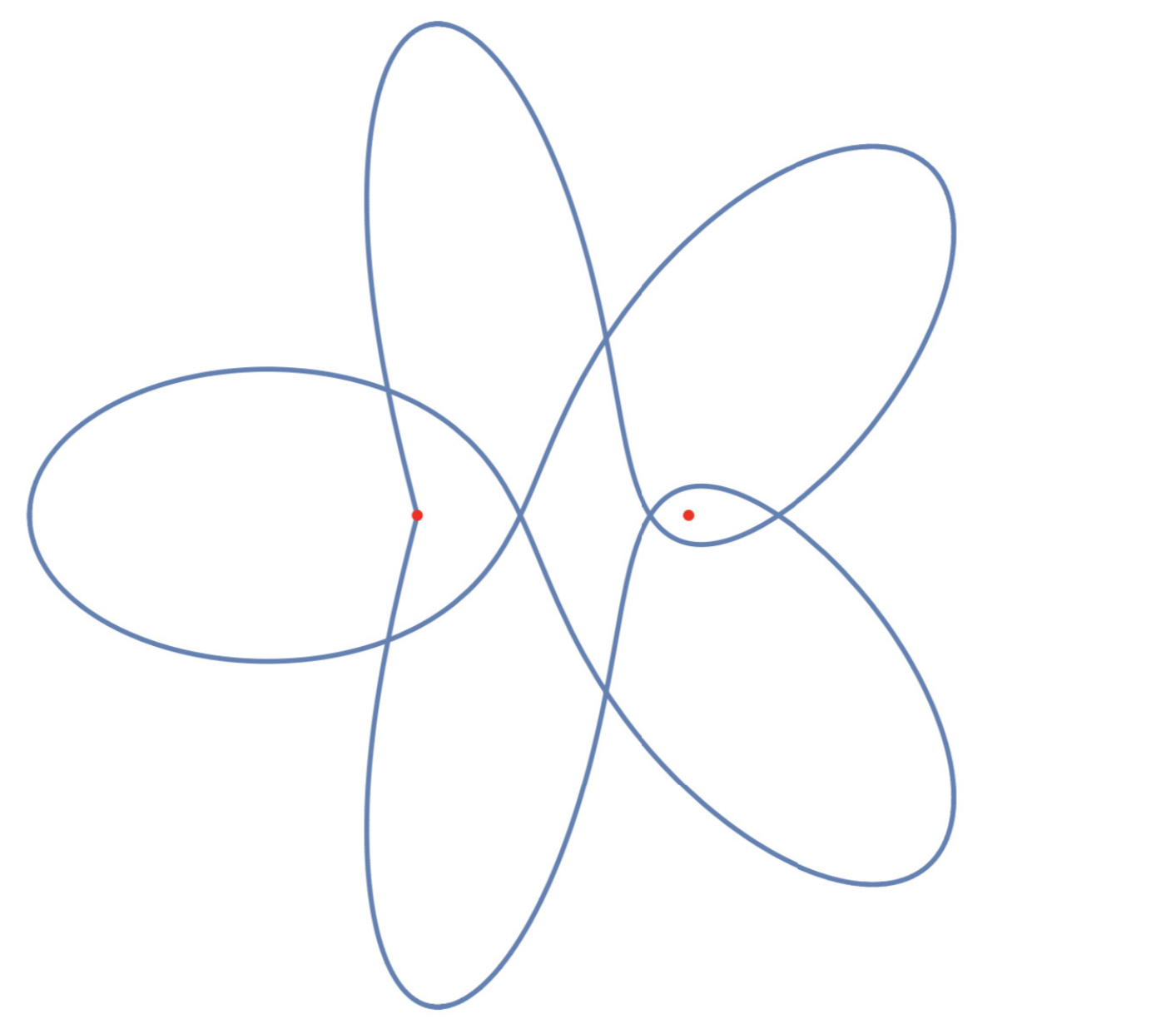}

\end{subfigure}
 \begin{subfigure}{0.43\textwidth}
 \begin{tikzpicture}[scale=0.7]
 
 \draw [fill, white] (1, 3.3) circle [radius=0.05];

 \begin{scope}[yscale=1, xscale=-1 ]

 \draw [thick] (-0.6,0) to ( -1.24,  0.86);
  \draw [thick] (-0.6,0) to ( -1.24, -0.86);
 \draw [thick] (-0.5,2.35) to ( -1.24,  0.86);
 \draw [thick] (-0.5,-2.35) to ( -1.24, -0.86);
 \draw [thick] (-0.5,2.35) to (1.3, 2.95);
 \draw [thick] (-0.5,-2.35) to (1.3, -2.95);
 \draw [thick]  (2.05, 1.25)to (1.3, 2.95);
 \draw [thick]  (2.05, -1.25)to (1.3, -2.95);
 \draw [thick]  (1.5,0)   to (2.05, 1.25);
 \draw [thick]  (1.5,0)   to (2.05, -1.25);
  \draw [thick]  (3.2,0)   to (2.05, -1.25);
  \draw [thick]  (3.2,0)   to (2.05, 1.25);
 \draw [thick]    ( -2.6, -1.85) to (-0.5,-2.35);
 \draw [thick]  ( -2.6, 1.85) to (-0.5, 2.35);
 \draw [thick]  (-2.4, 0) to ( -2.6, -1.85);
 \draw [thick]  (-2.4, 0) to( -2.6, 1.85);
\draw[thick] ( -1.24,  0.86) to (-2.4, 0);
 \draw[thick] ( -1.24, -0.86) to (-2.4, 0);
 \draw[thick] ( -1.24, -0.86) to (-0.12, -0.6);
 \draw[thick] ( -1.24,  0.86)to (-0.12,  0.6);
 \draw[thick] (  0.62, 1.36  ) to (-0.12,  0.6);
 \draw[thick] (  0.62, -1.36  ) to (-0.12, - 0.6);
 \draw[thick] (   0.52, -0.32   ) to ( 0.62, -1.36);
 \draw[thick] (   0.52,  0.32   ) to ( 0.62,  1.36);
 \draw[thick] (    1.5, 0 ) to ( 0.52,  0.32);
 \draw[thick] (   1.5, 0 ) to ( 0.52,  -0.32);
\draw[thick] ( -0.6, 0) to (-0.12, 0.6);
 \draw[thick] ( -0.6, 0) to (-0.12,-0.6);
\draw[thick] ( 0.52, 0.32) to (-0.12, 0.6);
\draw[thick] ( 0.52, -0.32) to (-0.12, -0.6);
\draw[thick] ( 0.52, 0.32) to (0.52, -0.32);
\draw[thick]  (-0.5, 2.35) to (0.62, 1.36);
 \draw[thick]  (-0.5, -2.35) to (0.62, -1.36);
\draw[thick]  (  2.05, 1.25) to (0.62, 1.36);
\draw[thick]  (  2.05, -1.25) to (0.62, -1.36);

  \draw [gray,thin] (0,0) circle [radius=1.5];
    \draw[gray,thin] (0,0) circle [radius=0.6];
  \draw [gray,thin](0,0) circle [radius=2.4];
  \draw [gray, thin](0,0) circle [radius=3.2];

\draw[gray,thin] (0,0) to (3.2,0);
\draw[gray,thin] (0,0) to (-3.2,0);
\draw[gray,thin] (0,0) to  ( 1.3, 2.95);
\draw[gray,thin] (0,0) to  ( 1.3, -2.95);
\draw [gray,thin](0,0) to ( -2.6, 1.85);
\draw[gray,thin] (0,0) to ( -2.6, -1.85);
\draw[gray,thin] (0,0) to (-0.67, 3.15);
\draw[gray,thin] (0,0) to (-0.67, -3.15);
\draw[gray,thin] (0,0) to (2.75, 1.67);
\draw[gray,thin] (0,0) to (2.75, -1.67);

 \draw [fill] (-2.4,0) circle [radius=0.05];
\draw [fill] (3.2,0) circle [radius=0.05];
\draw [fill] (1.5,0) circle [radius=0.05];
\draw [fill] (-0.6,0) circle [radius=0.05];
\draw [fill] (-0.5, -2.35) circle [radius=0.05];
\draw [fill] (-0.5,  2.35) circle [radius=0.05];
\draw [fill] (-0.12, -0.6) circle [radius=0.05];
\draw [fill] (-0.12, 0.6) circle [radius=0.05];

\draw [fill] (1.3, -2.95) circle [radius=0.05];
\draw [fill] (1.3, 2.95) circle [radius=0.05];
\draw [fill]  ( -2.6, -1.85) circle [radius=0.05];
\draw [fill]  ( -2.6, 1.85) circle [radius=0.05];
\draw [fill]  ( -1.24, 0.86) circle [radius=0.05];
\draw [fill]  ( -1.24, -0.86) circle [radius=0.05];
\draw [fill] (0.62,1.36) circle [radius=0.05];
\draw [fill] (0.62,-1.36) circle [radius=0.05];
\draw[fill]  (2.05, 1.25) circle [radius =0.05];
\draw[fill]  (2.05, -1.25) circle [radius =0.05];
\draw[fill]  (0.52, 0.32) circle [radius =0.05];
\draw[fill]  (0.52, -0.32) circle [radius =0.05];
\end{scope}
\end{tikzpicture}

\end{subfigure}
 \caption{A $T_{5,4}$-type collision-collision orbit $K$ (left) and its preimage $B^{-1}(K)$ before smoothing (right)}
\label{motionsd}
\end{figure}
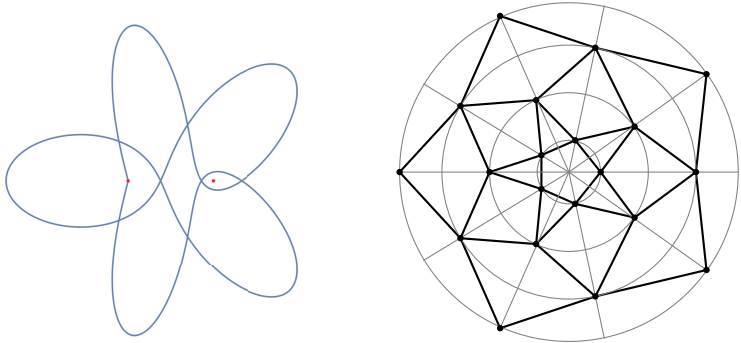

\end{enumerate}
\medskip
 
 We now apply Viro's formula   as in \cite{KKJplus} and obtain the following, which   completes the proof of the theorem \ref{thmmain}.

 \begin{proposition} Every $T_{k,l}$-type orbit in the L-region satisfies
   \[
 \mathcal{J}_{E,M} (K) =J^+(\widetilde{K}) = 1 - k+ kl - l^2 \text{  modulo } 2n(K) = 2l.
\]  
 \end{proposition}

\begin{remark}
Using the fact that families of periodic orbits in the rotating Kepler problem arises as  bifurcations of $S^1$, one readily obtains that the ${\mathcal{J}}_{E,M}$-invariant of a $T_{k,l}$-type orbit in the Euler problem  coincides with  the $J^+$-invariant of a retrograde $T_{k,l-k}$-type orbit in the rotating Kepler problem, provided  $k<l,$ or a direct $T_{k,k-l}$-type orbit, provided   $k>l.$ For explicit formulas,  see   \cite[Section 4]{KKJplus}.
\end{remark}

 \begin{remark}

 Alternatively, the invariant $\mathcal{J}_{E,M} $ can be computed using \cite[Theorem C]{MaiAl2022},  which states  that the $J^+$-invariant of  $\tilde K$, which bifurcates from the $k$-fold covering of a generic immersion $K$, satisfies
\begin{equation*}
    J^+(\tilde{K})= k^2 J^+(K) - (k^2-1) + n_{\tilde{K} } - k^2 n_K  ,
\end{equation*}
where $n_K$ and $n_{\tilde{K}}$ indicate the number of double points along $K$ and $\tilde{K}$, respectively.

To apply this to our setting, suppose that $K$ is a collision-collision orbit on a $T_{k,l}$-torus. Then $B^{-1}(K)$ can be thought of as a closed curve, bifurcating from $ l$-fold covering of the circle $S^1$, and we obtain 
\begin{align*}
\mathcal{J}_{E,M}(K) &= J^+(B^{-1}(K)) \\
&= l^2 J^+(S^1) - (l^2-1) + n_{B^{-1}(K)} -l^2  n_{S^1}  \\
&= 1 - l^2 + n_{B^{-1}(K)} \quad {(\text{mod $2l$})}   ,
\end{align*}
where we have used the facts that $J^+(S^1)=0$ and $n_{S^1}=0$.
Recall from Lemma \ref{lem:Lregion} that if $l$ is odd, then $K$ collides with both $E$ and $M$.  Since the Birkhoff regularization map $B$ is a branched $2$-to-$1$ covering, it follows from Proposition \ref{prop:quadruple} that $B^{-1}(K)$ has $k l-k$ double points. In the case where $l$ is even, $K$ collides only with one primary, say $E$,   and has $( kl - k -1)/2$ double points. Then, $B^{-1}(K)$ has $kl-k = (kl-k-1)+1$ double points, where the additional $+1$ accounts for the fact that the point $(-1,0) = B^{-1}(E)$ becomes a double point. In both cases,   we conclude that  $n_{B^{-1}(K)} = k l-k, $ and thus
\[
\mathcal{J}_{E,M}(K) =  1 - l^2 + kl -k \quad {(\text{mod $2l$})}.   
 \]

 \end{remark}


\bibliographystyle{abbrv}
\bibliography{mybibfile}

\end{document}